\documentclass[10pt]{article}
\usepackage[a4paper, total={7in, 9in}]{geometry} 
\usepackage{graphicx} 
\usepackage{amsthm}
\usepackage{amsfonts}
\usepackage{amssymb}
\usepackage{amsmath}
\usepackage{mathtools}
\usepackage[maxbibnames=99]{biblatex}
\usepackage{xcolor}
\usepackage{esint} 

\usepackage[citecolor=blue,colorlinks]{hyperref}

\newtheorem{innercustomgeneric}{\customgenericname}
\providecommand{\customgenericname}{}
\newcommand{\newcustomtheorem}[2]{%
  \newenvironment{#1}[1]
  {%
   \renewcommand\customgenericname{#2}%
   \renewcommand\theinnercustomgeneric{##1}%
   \innercustomgeneric
  }
  {\endinnercustomgeneric}
}

\newcustomtheorem{customthm}{Theorem}
\newcustomtheorem{customlemma}{Lemma}

\newtheorem{thm}{Theorem}[section]
\newtheorem{corollary}[thm]{Corollary}
\newtheorem{lemma}[thm]{Lemma}
\newtheorem{proposition}[thm]{Proposition}

\newtheorem*{thm*}{Theorem}
\newtheorem*{corollary*}{Corollary}
\newtheorem*{lemma*}{Lemma}
\newtheorem*{proposition*}{Proposition}

\theoremstyle{definition}
\newtheorem{definition}[thm]{Definition}
\newtheorem*{definition*}{Definition}
\newtheorem{remark}[thm]{Remark}
\newtheorem{example}[thm]{Example}
\newtheorem*{remark*}{Remark}

\newcommand{\Epi}{\mathsf{Epi}}
\newcommand{\bb}[1]{\mathbb{#1}}
\newcommand{\ssf}[1]{\mathsf{#1}}
\newcommand{\supp}{\mathsf{supp}}
\newcommand{\Per}{\mathsf{Per}}
\newcommand{\lip}{\mathsf{lip}}
\newcommand{\Lip}{\mathsf{Lip}}
\newcommand{\Hyp}{\mathsf{Hyp}}
\newcommand{\Tan}{\mathsf{Tan}}
\newcommand{\A}{\mathsf{A}}
\newcommand{\Ch}{\mathsf{Ch}}
\newcommand{\Ric}{\mathsf{Ric}}
\newcommand{\Test}{\mathsf{Test}}
\newcommand{\Hess}{\mathsf{Hess}}
\newcommand{\RCD}{\mathsf{RCD}}
\newcommand{\Graph}{\mathsf{Graph}}

\newcommand{\aH}{\mathsf{H}}
\newcommand{\BV}{\mathsf{BV}}
\newcommand{\W}{\mathsf{W}}
\newcommand{\X}{\mathsf{X}}
\newcommand{\M}{\mathsf{M}}
\newcommand{\x}{\mathsf{x}}
\newcommand{\Osc}{\mathsf{Osc}}
\newcommand{\sd}{\mathsf{d}}
\newcommand{\sL}{\mathsf{L}}
\newcommand{\m}{\mathfrak{m}}
\newcommand{\mms}[3]{(\mathsf{#1},\mathsf{#2},\mathfrak{#3})}
\newcommand{\pmms}[4]{(\mathsf{#1},\mathsf{#2},\mathfrak{#3},\mathsf{#4})}

\makeatletter
\newsavebox\myboxA
\newsavebox\myboxB
\newlength\mylenA

\newcommand*\xoverline[2][0.75]{%
    \sbox{\myboxA}{$\m@th#2$}%
    \setbox\myboxB\null
    \ht\myboxB=\ht\myboxA%
    \dp\myboxB=\dp\myboxA%
    \wd\myboxB=#1\wd\myboxA
    \sbox\myboxB{$\m@th\overline{\copy\myboxB}$}
    \setlength\mylenA{\the\wd\myboxA}
    \addtolength\mylenA{-\the\wd\myboxB}%
    \ifdim\wd\myboxB<\wd\myboxA%
       \rlap{\hskip 0.5\mylenA\usebox\myboxB}{\usebox\myboxA}%
    \else
        \hskip -0.5\mylenA\rlap{\usebox\myboxA}{\hskip 0.5\mylenA\usebox\myboxB}%
    \fi}
\makeatother

\title{Half Space Property in $\RCD(K,N)$ spaces}
\author{Alessandro Cucinotta \and Andrea Mondino}

\begin{document}

\maketitle

\begin{abstract}
The goal of this note is to prove the Half Space Property for $\RCD(0,N)$ spaces, namely that if $\mms{X}{d}{m}$ is a parabolic $\RCD(0,N)$ space and $ C \subset \ssf{X} \times \bb{R}$ is locally the boundary of a perimeter minimizing set and it is contained in a half space, then $C$ is a locally finite union of horizontal slices. \par 
The same result is proved for $\RCD(K,N)$ spaces, for any $K\in \mathbb{R}$ and $N\in (1,\infty)$,  under the stronger assumption that $C$ is the boundary of a \emph{globally} perimeter minimizing set.
\par
As a consequence, we obtain oscillation estimates and a Half Space Theorem for minimal hypersurfaces in products $\M \times \bb{R}$, where $\M$ is a parabolic smooth manifold (possibly weighted and with boundary), satisfying a Ricci curvature lower bound. \par 
On the way of proving the Half Space Property, we also extend to the $\RCD$ setting some classical results on Green's functions and parabolic manifolds.
\end{abstract}

\tableofcontents

\section{Introduction}
In \cite{HM} Hoffman and Meeks proved the Half Space Theorem, stating that a connected, proper, non-planar minimal surface in $\bb{R}^3$ cannot be contained in a half space. Analogous results in the Riemannian setting were then obtained by different authors (\cite{HLRS}, \cite{DBM}, \cite{DBH}, \cite{EH}, etc.) and in \cite{RSS} Rosenberg, Schulze and Spruck gave the following definition.

\begin{definition*}
   A Riemannian manifold $\M$ has the Half Space Property if every proper
minimal hypersurface $S \subset \M \times \bb{R}$ not intersecting $\M \times \{0\}$ equals a horizontal slice $\M \times \{c\}$. 
\end{definition*}

The study of this property, which generalizes the Half Space Theorem, is linked to the understanding of the following more general problem: what are the conditions such that two minimal submanifolds $S_1$, $S_2$ of a Riemannian manifold $\M$ must intersect? And if they do not intersect, do they influence each other's geometry? \par 
Classical results in this direction, aside from the Half Space Theorem, are the fact that the only positive solutions of the minimal surface equation in Euclidean space are constant  functions (\cite{BGM}) and Frankel's Theorem, stating that two compact minimal hypersurfaces in a closed manifold with positive Ricci curvature must intersect (\cite{F}). \par 
More recently, in the aforementioned work \cite{RSS}, it was proved that the Half Space Property holds for parabolic manifolds with bounded sectional curvature and the result was then extended to parabolic manifolds with a uniform lower bound on the Ricci curvature in \cite{CMMR}. It was then proved in the very recent work \cite{DingCap} that even the lower bound on the Ricci curvature can be removed, so that the Half Space Property holds for general parabolic manifolds. 
We recall that a manifold is parabolic if it does not admit a positive Green's function for the Laplacian and we note that without the parabolicity assumption the Half Space Property fails in general, the catenoid in $\bb{R}^4$ being a counterexample. \par 
The validity of the Half Space Property for parabolic smooth Riemannian manifolds with lower Ricci curvature bounds and the recent generalization of Frankel's Theorem to the setting of $\RCD$ spaces (\cite{Weak}), i.e.\;non-smooth spaces with a synthetic lower bound on the Ricci curvature, suggest that also the Half Space Property might hold in the non-smooth setting. \par
We recall that an $\RCD(K,N)$ space is a metric measure space where $K \in \bb{R}$ plays the role of a lower bound on the Ricci curvature, while $N \in [1,+\infty)$ plays the role of an upper bound on the dimension; this class includes measured Gromov-Hausdorff limits of smooth manifolds with uniform Ricci curvature lower bounds and finite dimensional Alexandrov spaces with sectional curvature bounded from below.
Moreover, since a hypersurface in a smooth manifold is minimal if and only if it is locally the boundary of a perimeter minimizing set, it is natural, in the non-smooth setting, to replace minimal hypersurfaces with local boundaries of local perimeter minimizers.
Recall that a set $C\subset \ssf{X}$ is locally the boundary of a perimeter minimizing set if, for every $x\in \X$ there exists a (small) metric ball $B_r(x)$ and a set of finite perimeter $E\subset \ssf{X}$, minimizing the perimeter in  $B_r(x)$, such that $C\cap B_r(x)=\partial E\cap B_r(x)$ (see also Remark \ref{Rboundary}).
\par 
We can now state the main results of this note.

\begin{thm} \label{CT1}
    Let $\mms{X}{d}{m}$ be a parabolic $\RCD(0,N)$ space and let $C \subset \ssf{X} \times \bb{R}$.
     Then the following two assertions are equivalent:
     \begin{enumerate}
\item $C$ is locally the boundary of a perimeter minimizing set, and $C$ is contained in a half space, i.e.\;there exists $t_0\in \bb{R}$ such that $C\subset \ssf{X}\times [t_0,\infty)$;
     \item  $C$ is a locally finite union of horizontal slices, i.e.\;there exists a discrete set $D \subset [t_0,\infty)$ such that
    \[
    C=\bigcup_{d \in D} \X \times \{d\}.
    \]
     \end{enumerate}
\end{thm}

Under the additional assumption that $C$ is the boundary of a \emph{globally} perimeter minimizing set, a result analogous to Theorem \ref{CT1} holds for possibly negative lower bounds on the synthetic Ricci curvature. This is the content of Theorem \ref{CT2Mod} below (see also Theorems \ref{T|GlobalMin} and \ref{T|MinPerSlow} for slightly more general results).

\begin{thm}\label{CT2Mod}
    Let $\mms{X}{d}{m}$ be a parabolic $\RCD(K,N)$ space and let $E \subset \ssf{X} \times \bb{R}$.
     Then the following two assertions are equivalent:
     \begin{enumerate}
     \item  $E$ is perimeter minimizing, it has connected boundary, and it is contained in a half space, i.e.\;there exists $t_0\in \bb{R}$ such that $E \subset \ssf{X}\times [t_0,\infty)$;
\item $E= \X \times [t_0,\infty)$ for some $t_0 \geq 0$.
     \end{enumerate}
\end{thm}
The previous theorems, in particular, are part of a wider class of recent results that aim at generalizing to the non-smooth setting properties of perimeter minimizing sets (\cite{APPS}, \cite{FSM}, \cite{BPSrec}, \cite{Weak}, etc.). It is interesting to note that, while in the smooth category the Half Space Property holds for parabolic manifolds even without the lower bound on the Ricci curvature, the same is not true for metric measure spaces (see Example \ref{Ex1}). 
In particular, in the non-smooth setting, once we remove the $\RCD$ assumption, it is not clear which conditions should be imposed (in addition to parabolicity) to guarantee the validity of the Half Space Property. \par
Specializing Theorem \ref{CT1} to the smooth category, we obtain that the Half Space Property holds for certain free boundary minimal surfaces on weighted manifolds with boundary. This is the content of Theorem \ref{CT2}. 
Given a manifold $(\ssf{M},g)$ we denote by $\mathfrak{m}_g$ its volume measure and by $\ssf{d}_g$ its distance. If $V:\ssf{M} \to \bb{R}$ is a smooth function, we say that the metric measure space $(\ssf{M}^n,\ssf{d}_g,e^{-V} \m_g)$ is a weighted manifold. A hypersurface $S$ in the weighted manifold $\ssf{M}$ is minimal if it is a critical point of the weighted area functional. We say that the boundary of $\M$ is convex if its second fundamental form w.r.t.\;the inward pointing unit normal is positive.

\begin{thm}\label{CT2}
Let $(\ssf{M}^n,\ssf{d}_g,e^{-V} \m_g)$ be a parabolic weighted manifold with convex boundary such that there exists $N>n$ satisfying
    \[
    \Ric_\M+\Hess_V-\frac{\nabla V \otimes \nabla V}{N-n} \geq 0 \quad \text{ on } \ssf{M} \setminus \partial \ssf{M}.
    \]
If $S \subset \ssf{M} \times (0,+\infty)$ is a properly embedded connected minimal hypersurface intersecting $\partial \ssf{M} \times \bb{R}$ orthogonally, then $S$ is a horizontal slice.
\end{thm}

The previous result is new already in the boundaryless weighted setting and in the framework of unweighted manifolds with boundary. \par
A consequence of Theorem \ref{CT2Mod} is that the oscillation of area minimizing hypersurfaces in an appropriate class of pointed manifolds grows with a uniform rate as one moves away from the base point in each manifold.
This is stated precisely in Theorem \ref{CT3} below. \par
We say that a function $P:[0,+\infty) \to (0,+\infty)$ is a modulus of parabolicity if $\int_1^{+\infty} tP(t)^{-1} dt = + \infty$ and we say that a pointed metric measure space $\pmms{X}{d}{m}{x}$ has modulus of parabolicity $P$ if $\m(B_r(\ssf{x})) \leq \m(B_1(\x)) P(r)$ for every $r\geq 0$. We say that a hypersurface $S \subset \X$ is an area minimizing boundary in an open set $A \subset \X$ if the exists a set $E \subset A$ minimizing the perimeter in $A$ such that $\partial E \cap A=S$.
If $S \subset \ssf{X} \times \bb{R}$ and $(\ssf{x},0) \in S$ we define the oscillation of $S$ as
    \[
    \Osc_{\ssf{x},r}(S):=\sup \{|t|:(y,t) \in S \cap B_r(\ssf{x}) \times (-r,r)\}.
    \]
\begin{thm} \label{CT3}
    Let $K \in \bb{R}$, let $n \in \bb{N}$, and let $P$ be a modulus of parabolicity. For every $t,r,T>0$ there exists $R>0$ such that for every pointed manifold $(\ssf{M}^n,g,\x)$ with modulus of parabolicity $P$ and $\Ric_M \geq K$,
    and for every area minimizing boundary $S$ in $B_R (\ssf{x}) \times (-R,R)$ containing $(\ssf{x},0)$, if $\Osc_{\ssf{x},r}(S) \geq t$ then $\Osc_{\ssf{x},R}(S) \geq T$.
\end{thm}

\begin{remark*}
In Section \ref{S1} we actually prove a more general result dealing with weighted manifolds with boundary and with a stronger notion of oscillation (see Theorem \ref{T9}), but we preferred to state Theorem \ref{CT3} in this form for the sake of simplicity.
\end{remark*}
Now we comment on the proof of Theorem \ref{CT1}.
\\
The proof of the implication  $2 \Rightarrow 1$ is a standard calibration argument, we will thus only discuss the implication $1 \Rightarrow 2$.

The proof that we give is totally different from the ones covering the analogous results in the smooth setting. 
In the works \cite{RSS} and \cite{CMMR} the general principle is that if $S \subset \ssf{M} \times \bb{R}$ is minimal and contained in a half space then, given a compact set $D \neq \emptyset$, it will be possible to construct a solution to the minimal surface equation $u: \ssf{M} \setminus D \to \bb{R}$ which is bounded from below and is not constant if $S$ is not a slice. The next step is then to prove a suitable generalization of De Giorgi's a priori gradient estimate for solutions of the minimal surface equation and to use the lower bound on $u$ to obtain a uniform bound on its gradient.
Finally one combines the parabolicity of $\ssf{M}$ with the bound on the gradient of $u$ to deduce that also the graph of $u$ is parabolic, and since $u$ is harmonic on its graph, by general properties of parabolic manifolds, it has to be constant and $S$ has to be a slice. \par 
There are several points of this strategy that it is not clear how to adapt to the non-smooth framework. 
A key issue is that the a priori gradient estimate on $u$ under a lower Ricci curvature bound is a consequence of the second variation formula for minimal surfaces, and such formula relies crucially on the smoothness of the ambient space.
Moreover, even if such estimate were available, the graph of a regular function on an $\RCD$ space is not known, in general, to be an $\RCD$ space and so, losing even this mild regularity, it would be difficult to adapt the previous technique, as it is based on the parabolicity of the graph. \par 
Our proof is instead closer to the one used in \cite{Weak} to generalize Frankel's Theorem. The key idea taken from such proof is that if $M$ is a manifold with non-negative Ricci curvature and $S$ is a minimal hypersurface contained in $\ssf{M} \times (0,+\infty)$, then the function $\bar{\ssf{d}}:=\ssf{d}_S-\ssf{d}_{\ssf{M} \times \{0\}}$, defined on $\ssf{M} \times (-\infty,0)$, has negative Laplacian, it is positive and if it is constant then $S$ is a slice. 
Hence the goal is to use the maximum principle for super-harmonic functions to show that $\bar{\ssf{d}}$ is constant. This is immediate if $\ssf{M}$ is compact, as in this case $\bar{\ssf{d}}$ achieves its minimum on all the points of the geodesic realizing the distance between $\ssf{M} \times \{-1\}$ and $S$. \par 
If $\ssf{M}$ is not compact, the argument is by contradiction. Exploiting the geometric properties of distance functions and the super-harmonicity of $\bar{\ssf{d}}$, one shows that if $\bar{\ssf{d}}$ is not constant, then there exist an open bounded set $A \subset \ssf{M} \times (-\infty,0)$ and $\tau >0$ such that 
\[
\Delta e^{-\bar{\ssf{d}}} \geq 0 \text{ on } \ssf{M} \times (-\infty,0), \quad \Delta e^{-\bar{\ssf{d}}} \geq \tau \text{ on } A.
\]
If the product $\ssf{M} \times \bb{R}$ is parabolic, then there exists a continuous function $\phi:\ssf{M} \times \bb{R} \to \bb{R}$ such that
\[
\Delta \phi \geq 0 \text{ on } \ssf{M} \times (-\infty,0), \quad \Delta \phi \geq -\tau \text{ on } A
\]
which tends to $-\infty$ at infinity (as it happens on $\bb{R}$ and $\bb{R}^2$ because of their parabolicity). As a consequence, the function $e^{-\bar{\ssf{d}}}+\phi$ is sub-harmonic and has a maximum in $\ssf{M} \times (-\infty,0]$. Moreover, since $e^{-\bar{\ssf{d}}}$ increases as we move down a vertical line, by constructing $\phi$ in an appropriate way we can assume that the maximum of $e^{-\bar{\ssf{d}}}+\phi$ actually lies in $\ssf{M} \times (-\infty,0)$. The maximum principle then implies that $e^{-\bar{\ssf{d}}}+\phi$ is constant, contradicting the condition at infinity on $\phi$. 
\par 
If $\ssf{M} \times \bb{R}$ is not parabolic the previous argument has to be modified, and working with metric measure spaces instead of Riemannian manifolds becomes crucial. Consider on $\ssf{M} \times \bb{R}$ the distance $\ssf{d}_\times$ induced by the product Riemannian metric and the measure $\mathfrak{m}_\times$ which is the product volume measure.
Using the parabolicity assumption on $\ssf{M}$ and the particular form of the function $e^{-\bar{\ssf{d}}}$, it is possible to replace the measure $\mathfrak{m}_\times$ with a measure $\tilde{\mathfrak{m}}_\times$ such that the product space $(\ssf{M} \times \bb{R},\ssf{d}_\times,\tilde{\mathfrak{m}}_\times)$ is an $\RCD(K,N)$ space (possibly for $K<0$), it is parabolic and, calling $\tilde{\Delta}$ the Laplacian in this modified space, we have that there exists $\tau'>0$ such that
\[
\tilde{\Delta} e^{-\bar{\ssf{d}}} \geq 0 \text{ on } \ssf{M} \times (-\infty,0), \quad \tilde{\Delta} e^{-\bar{\ssf{d}}} \geq \tau' \text{ on } A.
\]
To conclude, we can now repeat the previous argument involving the construction of an auxiliary function $\phi$ and the maximum principle that, in this case, is used on the modified space $(\ssf{M} \times \bb{R},\ssf{d}_\times,\tilde{\mathfrak{m}}_\times)$. \par 
This strategy works for $\RCD(0,N)$ spaces, but not on $\RCD(K,N)$ spaces when $K<0$, as in this case it is not true in general that the function $\bar{\ssf{d}}$ is super-harmonic. In particular, the problem of generalizing Theorem \ref{CT1} to $\RCD(K,N)$ spaces remains open. Theorem \ref{CT2Mod} is a step in this direction. For an overview of the proof of Theorem \ref{CT2Mod}, we refer to the beginning of Section \ref{SGlobal}.
\par
Finally, to implement the previously outlined strategy, we need to generalize to $\RCD(K,N)$ spaces some classical properties of parabolic manifolds. This is obtained by extending the work done in \cite{BSflow}, \cite{Weak} and \cite{BScon} on Green's functions in the non smooth setting. 
Theorems \ref{CT4} and \ref{CT5} are the main results in this regard.
Given an $\RCD(K,N)$ space $\mms{X}{d}{m}$ we denote by $p_t$ the corresponding heat kernel and we refer to Section \ref{S2} for the precise definitions of Green's functions, regular sets,  and capacities. 

\begin{thm} \label{CT4}
    Let $\pmms{X}{d}{m}{x}$ be a pointed $\RCD(K,N)$ space with infinite diameter. The following are equivalent:
    \begin{enumerate}
        \item There is no positive Green's function on $\ssf{X}$ with pole $\ssf{x}$.
        \item For every $(x,y) \in X \times X$ we have $\int_1^{+ \infty} p_t(x,y) \, dt =+ \infty$.
        \item There exists $(x,y) \in X \times X$ such that $\int_1^{+ \infty} p_t(x,y) \, dt =+ \infty$.
        \item Let $\{B_i\}_{i \in \bb{N}}$ be an exhaustion of regular sets containing $\ssf{x}$. Let $G_\ssf{x}^i$ be the Green's function on $B_i$ with pole $\ssf{x}$. Then for every $y \neq \x$ we have $G_\ssf{x}^i(y) \to + \infty$ as $i \to + \infty$.
    \item For every compact set $K \subset X$, it holds $\ssf{Cap}(K)=0$.
    \end{enumerate}
\end{thm}

\begin{thm} \label{CT5}
    Let $\mms{X}{d}{m}$ be an $\RCD(K,N)$ space. If
    there exists $x \in X$ such that 
    \[
    \int_1^{+ \infty} \frac{t}{\mathfrak{m}(B_t(x))} \, dt=+ \infty,
    \]
    then $\ssf{X}$ is parabolic. This condition is also necessary if $K=0$.
\end{thm} 

The note is organized as follows. 
Section \ref{Sprel} contains preliminaries.
Section \ref{S2} concerns the theory of Green's functions on $\RCD$ spaces and contains the proofs of Theorems \ref{CT4} and \ref{CT5}.
In Section \ref{S3}, we rigorously prove, following the previously outlined strategy, the Half Space Property for $\RCD(0,N)$ spaces, i.e. Theorem \ref{CT1}. In Section \ref{SGlobal}, we prove Theorem \ref{CT2Mod}, while in the final section we deal with applications in the smooth setting, proving Theorems \ref{CT2} and \ref{CT3}.

\medskip

\textbf{Acknowledgements.}
A.M.\;is supported by the European Research Council (ERC), under the European Union Horizon 2020 research and innovation programme, via the ERC Starting Grant  “CURVATURE”, grant agreement No. 802689.

\section{Preliminaries} \label{Sprel}
Throughout the note we will work on metric measure spaces $\mms{X}{d}{m}$, where $(\ssf{X},\ssf{d})$ is a separable complete metric space where balls are precompact and $\m$ is a non-negative Borel measure on $\ssf{X}$ which is finite on bounded sets and whose support is the whole $\ssf{X}$. 
Given a smooth manifold $(\ssf{M},g)$ we denote by $\sd_g$ and $\m_g$ respectively the distance induced by $g$ and the volume measure.
Given an open set $\Omega \subset \ssf{X}$ we denote by $\Lip(\Omega)$, $\Lip_{loc}(\Omega)$ and $\Lip_c(\Omega)$ respectively Lipschitz functions, locally Lipschitz and Lipschitz functions with compact support in $\Omega$. If $f \in \Lip(\Omega)$, we denote its Lipschitz constant by $\sL(f)$. If $f \in \Lip_{loc}(\Omega)$ and $x \in \Omega$ we define
\[
\lip(f)(x):= \limsup_{y \to x} \frac{|f(x)-f(y)|}{\ssf{d}(x,y)}.
\]
Given a closed interval $I \subset \bb{R}$, we say that a curve $\gamma: I \to \X$ is a geodesic if its length coincides with the distance between its endpoints. Unless otherwise specified we assume that geodesics have constant unit speed.
\par
We will consider pointed measured Gromov Hausdorff convergence of metric measure spaces, and we refer to \cite{GMS13} for the relevant background.
We recall that in the case of a sequence of uniformly locally doubling metric measure spaces $(\ssf{X}_i,\ssf{d}_i, \m_i,\ssf{x}_i)$
(as in the case of $\RCD(K, N)$ spaces), pointed measured Gromov-Hausdorff convergence to $\pmms{X}{d}{m}{x}$ can be equivalently characterized by asking for the existence of a proper
metric space $(\ssf{Z}, \ssf{d}_z )$ such that all the metric spaces $(\ssf{X}_i
, \ssf{d}_i)$ are isometrically embedded
into $(\ssf{Z}, \ssf{d}_z )$, $\ssf{x}_i \to \ssf{x}$ and $\m_i \to \m$ weakly in duality with continuous boundedly supported functions in $\ssf{Z}$. In this case we say that the convergence is realized in the space $\ssf{Z}$ (see \cite{GMS13}).
Given a metric measure space $\mms{X}{d}{m}$ and $x \in \ssf{X}$, we denote by $\Tan_x(\ssf{X})$ the (possibly empty) collection of (isometry classes of) metric measure spaces that are pointed measured Gromov Hausdorff limits as $r \downarrow 0$ of the family $(\ssf{X},r^{-1}\ssf{d},\m(B_r(x))^{-1}\m,x)$.

\subsection{Sobolev spaces and the $\RCD$ condition} \label{SS1}

We now recall the definition of Sobolev spaces in the setting of metric measure spaces, the main references being \cite{Cheeger}, \cite{AGS}, \cite{AGS2} and \cite{Giglimem}.

\begin{definition}[Sobolev spaces]
    Let $\mms{X}{d}{m}$ be a metric measure space, $\Omega \subset \ssf{X}$ an open set and let $p>1$. A function $f \in \sL^p(\Omega)$ is said to be in the Sobolev space $\W^{1,p}(\Omega)$ if there exists a sequence of locally Lipschitz functions $\{f_i\}_{i \in \bb{N}} \subset \Lip_{loc}(\Omega)$ converging to $f$ in $\sL^p(\Omega)$ such that
    \[
    \limsup_{i \to + \infty} \int_\Omega \lip(f_i)^p \, d \m < + \infty.
    \]
     A function $f \in \sL^p_{loc}(\Omega)$ is said to be in the Sobolev space $\W^{1,p}_{loc}(\Omega)$ if for every $\eta \in \Lip_c(\Omega)$ we have $f \eta \in \W^{1,p}(\Omega)$.
\end{definition}

For any $f \in \W^{1,p}(\Omega)$ one can define an object $|\nabla f|$ (a priori depending on $p$, but independent of the exponent in the spaces that we will work on) such that for every open set $A \subset \Omega$ we have
    \[
    \int_A |\nabla f|^p \, d \m =
    \inf \Big\{ \liminf_{n \to + \infty} \int_A \lip(f_n)^p \, d \m \Big| (f_n)_n \subset \sL^p(A) \cap \Lip_{loc}(A), \|f_n-f\|_{\sL^p(A)} \to 0 \Big\}.
    \]
    The quantity in the previous expression will be called $p$-Cheeger energy and denoted by $\Ch_p(f)$ while $|\nabla f|$ will be called relaxed gradient. We define $\|f\|_{\W^{1,p}(\Omega)}:=\|f\|_{\sL^p(\Omega)}+\Ch_p(f)$. One can check that with this norm the space $\W^{1,p}(\Omega)$ is Banach. Later we will often write $\Ch$ in place of $\Ch_2$ for simplicity of notation.
    A deep result of Cheeger \cite{Cheeger} implies that if $f\in \Lip_{loc}(X)$ then $|\nabla f|=\lip(f), \, \m$-a.e., provided $\mms{X}{d}{m}$ is a PI-space, so in particular this holds for $\RCD(K,N)$-spaces, $N\in [1,\infty)$. 
We now introduce functions of bounded variation following \cite{Mir} (see also \cite{ADM}).

\begin{definition}[Functions of bounded variation]
    Let $\mms{X}{d}{m}$ be a metric measure space and let $\Omega \subset \ssf{X}$ be an open set. A function $f \in \sL^1(\Omega)$ is said to be of bounded variation if there exists a sequence of locally Lipschitz functions $\{f_i\}_{i \in \bb{N}} \subset \Lip_{loc}(\Omega)$ converging to $f$ in $\sL^1(\Omega)$ such that
    \[
    \limsup_{i \to + \infty} \int_\Omega \lip(f_i) \, d \m < + \infty.
    \]
    The space of such functions is denoted by $\BV(\Omega)$. A function $f \in \sL^1_{loc}(\Omega)$ is said to be in $\BV_{loc}(\Omega)$ if for every $\eta \in \Lip_c(\Omega)$ we have $f \eta \in \BV(\Omega)$.
\end{definition}

    For any $f \in \BV(\Omega)$ and any open set $A \subset \Omega$ we define
    \[
    |Df|(A) =
    \inf \Big\{ \liminf_{n \to + \infty} \int_A \lip(f_n) \, d \m \Big| (f_n)_n \subset \sL^1(A) \cap \Lip_{loc}(A), \|f_n-f\|_{\sL^1(A)} \to 0 \Big\}.
    \]
    One can check that the quantity in the previous expression is the restriction to the open subsets of $\Omega$ of a finite measure. We define $\|f\|_{\BV(\Omega)}:=\|f\|_{\sL^1(\Omega)}+|Df|(\Omega)$. One can check that with this norm the space $\BV(\ssf{X})$ is Banach. A function $f$ belongs to $\W^{1,1}(\Omega)$ if $f \in \BV(\Omega)$ and $|Df| \ll \m$. In this case we denote by $|\nabla f|$ the density of $|Df|$ with respect to $\m$.

\begin{definition}
    Let $\mms{X}{d}{m}$ be a metric measure space and let $\Omega \subset \ssf{X}$ be an open set. For every $p \geq 1$ we denote by $\W^{1,p}_0(\Omega)$ the closure in $\W^{1,p}(\Omega)$ of $\Lip_c(\Omega)$.
\end{definition}

The previous definition is particularly suitable for proper spaces, while other definitions are to be preferred if one works without this assumption. Since we will work in the setting of $\RCD(K,N)$ spaces with $N\in [1,\infty)$ (and these are proper) we do not consider more general definitions.

\begin{definition}[Infinitesimal Hilbertianity]
    We say that a metric measure space $\mms{X}{d}{m}$ is infinitesimally Hilbertian if the space $\W^{1,2}(\ssf{X})$ is a Hilbert space.
\end{definition}

 If $\mms{X}{d}{m}$ is infinitesimally Hilbertian and $\Omega \subset \ssf{X}$ an open set, for every $f,g \in \W^{1,1}_{loc}(\Omega)$ we define the measurable function $\nabla f \cdot \nabla g:\Omega \to \bb{R}$ by 
 \[
 \nabla f \cdot \nabla g:=\frac{|\nabla (f+g)|^2-|\nabla(f-g)|^2}{4}.
 \]
As a consequence of the infinitesimal Hilbertianity assumption, the previously defined product of gradients is bilinear in both entries.
We then recall the definition of Laplacian in the metric setting.

 \begin{definition}[Functions with $\sL^2$ Laplacian]
 Let $\mms{X}{d}{m}$ be infinitesimally Hilbertian and let $\Omega \subset \ssf{X}$ be an open set.
    Let $f \in \W^{1,2}(\Omega)$. We say that $f \in D(\Delta,\Omega)$ if there exists a function $h \in \sL^2(\Omega)$ such that
    \[
    \int_\Omega gh \, d \m=-\int_\Omega \nabla g \cdot \nabla f \, d \m \quad \text{for any } g \in \W^{1,2}_0(\Omega).
    \]
    In this case we say that $\Delta f=h$ in $\Omega$.
\end{definition}

We also have the following more general definition.

\begin{definition}[Functions with measure-valued Laplacian]
Let $\mms{X}{d}{m}$ be infinitesimally Hilbertian and let $\Omega \subset \ssf{X}$ be an open set.
    Let $f \in \W^{1,1}_{loc}(\Omega)$ and let $\mu$ be a Radon measure on $\Omega$. We say that $\Delta f=\mu$ in $\Omega$ in distributional sense if
    \[
    \int_\Omega g \, d\mu=-\int_\Omega \nabla g \cdot \nabla f \, d \m \quad \text{for any } g \in \Lip_c(\Omega).
    \]
\end{definition}

We now recall some properties of $\RCD(K,N)$ spaces, i.e. infinitesimally Hilbertian metric measure spaces with Ricci curvature bounded from below by $K \in \bb{R}$ and dimension bounded from above by $N \in [1,+\infty)$ in synthetic sense. \par
The Riemannian Curvature Dimension condition $\RCD(K, \infty)$ was introduced in \cite{AGS2} (see
also \cite{Giglimem,AGMR}) combining the Curvature Dimension condition $\ssf{CD}(K,\infty)$, previously pioneered
in \cite{S1,S2} and independently in \cite{V}, with the infinitesimal Hilbertianity assumption.
The finite dimensional counterpart $\RCD(K,N)$ is obtained coupling the finite dimensional Curvature Dimension condition $\ssf{CD}(K,N)$ with the infinitesimal Hilbertianity assumption and was formalized in \cite{Giglimem}. For a thorough introduction to the topic we refer to the survey \cite{Asurv} and the references therein. 
Let us also mention that in the literature one can find also the (a priori weaker) $\RCD^*(K,N)$. It was proved in \cite{EKS, AMS15}, that $\RCD^*(K,N)$ is equivalent to the dimensional Bochner inequality. Moreover, \cite{CavallettiMilman} (see also \cite{Liglob}) proved that $\RCD^*(K,N)$ and $\RCD(K,N)$ coincide.

We now recall the properties that we will use later on in the note.
The $\RCD(K,N)$ condition implies that the measure is locally doubling (see \cite{S1}) and the validity of a Poincaré inequality (see \cite{PI}). In particular if $f$ is a locally Lipschitz function on a $\RCD(K,N)$ space, its relaxed gradient coincides with its local Lipschitz constant $\lip(f)$ (see \cite[Theorem $12.5.1$]{HKST} after \cite{Cheeger}).

We now recall a Bishop-Gromov type inequality for $\RCD(K,N)$ spaces that can be found in \cite{S2}.
We define $V_{K,N}(r):=N\omega_N \int_0^r (s_{K,N}(t))^{N-1} \, dt$, where $\omega_N$ is the volume of the $N$-dimensional Euclidean ball and
\[
s_{K,N}(t):=
\begin{cases}
    \sqrt{\frac{N-1}{K}}\sin \Big(t\sqrt{\frac{K}{N-1}}\Big) & K >0 \\
    r & K=0 \\
    \sqrt{\frac{N-1}{-K}} \sinh\Big(t\sqrt{\frac{-K}{N-1}}\Big) & K < 0.
\end{cases}
\]
\begin{proposition} \label{P6}
    Let $(\X,\sd,\m)$ be an $\RCD(K,N)$ space. Then for every $x \in \X$ and $R>r>0$ we have
    \[
    \frac{\m(B_R(x))}{\m(B_r(x))} \leq \frac{V_{K,N}(R)}{V_{K,N}(r)}.
    \]
\end{proposition}

\begin{remark} \label{R2}
    Let $(\X,\sd,\m)$ be an $\RCD(K,N)$ space.
    Proposition \ref{P6}, together with the explicit expression of $V_{K,N}(r)$, implies that for every $x \in \X$ and $R>0$ there exists a constant $C>0$ such that $\m(B_r(x)) \geq Cr^N$ for every $r \in (0,R)$.
\end{remark}

The next theorem follows from \cite{MondinoNaber,BScon}. We denote by $\sd_e$ the Euclidean distance.
\begin{thm}
    Let $(\X,\sd,\m)$ be an $\RCD(K,N)$ space. There exists $k \in \bb{N} \cap [1,N]$, called essential dimension of $\X$, such that for $\m$-a.e. $x \in \X$ we have $\Tan_x(\X)=\{(\bb{R}^k,\sd_e)\}$. Any such point will be called a regular point for $\X$.
\end{thm}

The next result follows from \cite{KellMondino,GigliPasq} (see also \cite[Theorem $4.1$]{AHT}).

\begin{proposition} \label{P31}
    Let $(\X,\sd,\m)$ be an $\RCD(K,N)$ space of essential dimension $k>1$. Then for $\m$-a.e. $x \in \X$ we have
    \[
    \lim_{r \to 0^+}\frac{\m(B_r(x))}{r^k} \in (0,+\infty).
    \]
\end{proposition}

The next proposition is taken from \cite[Theorem $1.1$]{KL}.

\begin{proposition} \label{P33}
    Let $(\X,\sd,\m)$ be an $\RCD(K,N)$ space of essential dimension $1$. Then $\X$ is isometric to a closed connected subset of $\bb{R}$ or to $S^1(r):=\{x \in \bb{R}^2: |x|=r \}$.
\end{proposition}

We now recall some properties of the heat flow in the $\RCD$ setting, referring to \cite{AGMR,AGS2} for the proofs of these results.
Given an $\RCD(K,N)$ space $(\X,\sd,\m)$, the heat flow $P_t:\sL^2(\X) \to \sL^2(\X)$ is the $\sL^2(\X)$-gradient flow of the Cheeger energy $\Ch$.

It turns out that one can obtain a stochastically complete heat kernel $p_t:\X \times \X \to [0,+\infty)$, so that
the definition of $P_t(f)$ can be then extended to $\sL^{\infty}$ functions by setting
\[
P_t(f)(x):=\int_\X f(y) p_t(x,y) \, d\m(y).
\]
The heat flow has good approximation properties, in particular if $f \in \W^{1,2}(\X)$, then $P_t (f) \to f$ in $\W^{1,2}(\X)$; while if $f \in \sL^{\infty}(\X)$, then $P_t f \in \Lip(\X)$ for every $t>0$.

We conclude the section with two propositions on the heat kernel taken respectively from \cite[Theorem $1.1$]{Harheat} and \cite[Lemma $3.3$]{Harheat} (see also \cite{GaroMondino}).

\begin{proposition} \label{Hestim}
    Let $(\X,\sd,\m)$ be an $\RCD(K,N)$ space. There exist constants $C>1$ and $c \geq0$ such that for every $x,y \in \X$ and every $t>0$
    \[
    \frac{1}{C\m(B_{\sqrt{t}}(x))} \exp{\Big\{ \frac{-\sd^2(x,y)}{3t}-ct \Big\}} \leq p_t(x,y) \leq \frac{C}{\m(B_{\sqrt{t}}(x))} \exp{ \Big\{ \frac{-\sd^2(x,y)}{5t}+ct \Big\}}.
    \]
    Moreover for every $x\in \X$ we have that $p_t(x,\cdot)$ is locally Lipschitz in $\X \setminus\{x\}$ and for $\m$-a.e. $y \in \X$ and every $t>0$
    \[
    \lip(p_t(x,\cdot))(y)=|\nabla p_t(x,\cdot)|(y) \leq \frac{C}{\sqrt{t} \m(B_{\sqrt{t}}(x))} \exp{ \Big\{ \frac{-\sd^2(x,y)}{5t}+ct \Big\}}.
    \]
    If $K=0$ we can take $c=0$ in the previous estimates.
\end{proposition}

\begin{proposition} \label{Harnack}
    Let $(\X,\sd,\m)$ be an $\RCD(K,N)$ space with $K<0$ and $N \in [1,+\infty)$. For any $0<s<s+1 \leq t<+ \infty$ and $x,y,z \in \X$ it holds
    \[
    p_t(x,y) \geq p_s(x,z) \exp{\Big\{-\frac{\sd(x,z)^2}{2e^{2K/3}} \Big\}}\Big(\frac{1-e^{K/3}}{1-e^{2K/3}}\Big)^{N/2} \Big(\frac{1-e^{2Ks/3}}{1-e^{2K(t-1/2)/3}}\Big)^{N/2}.
    \]
\end{proposition}

Finally, we recall a tensorization property of infinitesimally Hilbertian spaces that can be found in \cite{Tens} and \cite{AGS2}.
Given $f :\X \times \bb{R} \to \bb{R}$ and $(x,t) \in \X \times \bb{R}$ we denote by $f^t$ and $f^x$ respectively the restriction of $f$ to $\X \times \{t\}$ and to $\{x\} \times \bb{R}$. We denote by $\sd_\times$ and $\m_\times$ respectively the product distance and the product measure in the space $\X \times \bb{R}$.

\begin{proposition} \label{P38}
Let $(\X,\sd,\m)$ be an $\RCD(K,N)$ space and let $f \in \Lip(\X \times \bb{R})$. Then we have
    \[
    |\nabla f|^2(x,t)=|\nabla f^x|^2(t)+|\nabla f^t|^2(x) \quad \text{for } \m_\times \text{-a.e. } (x,t) \in \X \times \bb{R}.
    \]
\end{proposition}

\subsection{Poisson problem and regular sets} \label{SS3}
For the results of this subsection, unless otherwise specified, $(\X,\sd,\m)$ is an 
 $\RCD(K,N)$ space and $\Omega \subset \X$ is an open set.
The next proposition is the maximum principle on $\RCD(K,N)$ spaces and can be found in \cite{MP}.

\begin{proposition}
     Let $\Omega \subset \X$ be open and connected and let $f \in \W^{1,2}(\Omega) \cap C(\bar{\Omega})$ be a function whose distributional Laplacian satisfies $\Delta f \geq 0$. If there exists $x \in \Omega$ such that
     $
     f(x)=\max_{y \in \bar{\Omega}}f(y),
     $
     then $f$ is constant.
\end{proposition}

The next proposition is taken from \cite[Theorem $2.58$]{Weak}.

\begin{proposition}
    Let $B \subset \subset \X$ be open with $\m(\X \setminus B)>0$, $f \in  C(\bar{B})$, and $g \in \W^{1,2}(\X)$. Then there exists $h \in \W^{1,2}(B)$ satisfying       
       \[
       \begin{cases}
       \Delta h= f \quad  \in B\\
       h+g \in  \W^{1,2}_0(B).
       \end{cases}
       \]
\end{proposition}

The next two propositions are taken from \cite{j13}.
\begin{proposition} \label{P7}
    Let $f \in D(\Delta,\Omega)$ be such that $\Delta f$ is continuous, then $f$ has a locally Lipschitz representative.
\end{proposition}

\begin{proposition} \label{P9}
    Let $r>0$ and $x \in \X$, then there exists a constant $c>0$ depending on $r$ such that if $f \in D(\Delta,A)$ for every $A \subset \subset \X$ and $\Delta f$ is continuous, then
    \[
    \|\nabla f\|_{\sL^{\infty}(B_r(x))} \leq c(\|f\|_{{\sL^{\infty}(B_{8r}(x))}}+\|\Delta f\|_{{\sL^{\infty}(B_{8r}(x))}}).
    \]
\end{proposition}

\begin{definition}[Regular sets]
    An open precompact subset $B \subset \bar{B} \subsetneq \X$ is said to be regular if for every
    $f \in \W^{1,2}(B)$ admitting a representative which is continuous on $\partial B$, there exists a function $u \in D(\Delta,B) \cap C(\bar{B})$ such that $\Delta u=0$ on $B$, $u-f \in \W^{1,2}_0(B)$ and $u=f$ on $\partial B$.
\end{definition}

The next proposition concerns existence of regular sets and is taken from \cite[Theorem $14.1$]{bjorn}. In the aforementioned book the setting is the one of doubling spaces supporting a Poincaré inequality and Newtonian Sobolev spaces, but the statement that we are interested in, concerning precompact sets, holds also in the $\RCD(K,N)$ setting (taking into account that Newtonian Sobolev spaces coincide with the ones we are using by \cite[Theorem $6.2$]{AGS}).

\begin{proposition}
    Let $\X$ have infinite diameter. There exists a sequence $\{ \Omega_i \}_{i \in \bb{N}}$ of regular sets such that $\Omega_i \subset \subset \Omega_{i+1}$ for every $i \in \bb{N}$ and $\X=\cup_i \Omega_i$. Such a sequence of sets is called an exhaustion of $\X$ with regular sets.
\end{proposition}

The next proposition follows from \cite[Theorem $4.1$]{WeakBar}, noting that by the maximum principle functions with negative Laplacian are super-harmonic in the sense of \cite{WeakBar}. The two subsequent corollaries are immediate consequences of Proposition \ref{P29}.

\begin{proposition} \label{P29}
    An open precompact set $B \subset \bar{B} \subsetneq \X$ is regular if for every $x \in \partial B$ there exists a barrier i.e. a function $u_x \in D(\Delta,B) \cap C(\bar{B})$, strictly positive on $B$, such that $\Delta u_x \leq 0$ and 
    \[
    \lim_{y \to x}u_x(y)=0.
    \]
\end{proposition}

\begin{corollary} \label{C4}
    An open precompact set $B \subset \bar{B} \subsetneq \X$ is regular if and only if for every $x \in \partial B$ there exists a function $u_x \in D(\Delta,B) \cap C(\xoverline{B})$, strictly positive on $B$, such that $\Delta u_x=0$ and
    \[
    \lim_{y \to x}u_x(y)=0.
    \]
\end{corollary}

\begin{corollary} \label{C5}
    Let $(\X,\sd,\m)$ and $(\X',\sd',\m')$ be $\RCD(K,N)$ spaces for $K \leq 0$. Let $B \subset \subset \X$ and $B' \subset \subset \X'$ be regular sets. Then $B \times B'$ is regular in the product space.

\end{corollary}

We now recall the definition of test functions in the $\RCD(K,N)$ setting, referring to \cite{Giglitaylored} for the proofs of the results that we subsequently list.

\begin{definition}[Test functions]
    We set
    \[
    \Test(\X):=\Big\{
    \phi \in D(\Delta,\X) \cap \sL^{\infty}(\X): |\nabla f| \in \sL^{\infty}(\X), \Delta f \in \sL^{\infty}(\X) \cap \W^{1,2}(\X) 
    \Big\}
    \]
    and
    \[
    \Test_c(\Omega):= \Big\{
    \phi \in \Test(\X): \supp(\phi) \subset \subset \Omega 
    \Big\}.
    \]
\end{definition}

It has been proved that $\Test(\X)$ is dense in $\W^{1,2}(\X)$, while $\Test_c(\Omega)$ is dense in $\W^{1,2}_0(\Omega)$. Moreover if $f \in \W^{1,2}(\X) \cap \sL^{\infty}(\X)$ then $P_t(f) \in \Test(\X)$. Finally, it has been proved in \cite{AMSbakry} that if $K \subset \Omega$ is compact then there exists $\eta \in \Test_c(\Omega)$ such that $\eta=1$ on a neighbourhood of $K$.

We then give a version of Weyl's Lemma in $\RCD(K,N)$ spaces that follows from \cite[Theorem $1.3$]{weyl}.
\begin{proposition} \label{P36}
    Let $u \in \sL^1_{loc}(\Omega)$ be a function such that
    \[
    \int_\Omega u \Delta \phi \, d\m=0 \quad \text{for every } \phi \in \Test_c(\Omega),
    \]
    then $u \in \W^{1,2}_{loc}(\Omega)$.
\end{proposition}

\subsection{Sets of finite perimeter} \label{SS4}
For the results of this section, unless otherwise specified, we will implicitly assume that we are working on a fixed $\RCD(K,N)$ space $(\X,\sd,\m)$.

\begin{definition}[Perimeter of a set]
    Let $E \subset \X$. We say that $E$ has locally finite perimeter if $1_E \in \BV_{loc}(\X)$. For every Borel subset $B \subset \X$. We denote $|D1_E|(B)$ by $P(E,B)$.
\end{definition}

\begin{definition}[Convergence in $\sL^1$ sense]
    Let $(\X_i,\sd_i,\m_i,\ssf{x}_i)$ be a sequence of $\RCD(K,N)$ spaces converging in pmGH sense to $(\ssf{Y},\sd,\m,\ssf{y})$. We say that the Borel sets $E_i \subset \X_i$ of finite measure converge in $\sL^1$ sense to a set $E \subset \ssf{Y}$ of finite measure if $\m_i(E_i) \to \m(E)$ and $1_{E_i}\m_i \to 1_E \m$ weakly in duality w.r.t. continuous compactly supported functions in the space $(\ssf{Z},\sd_z)$ realizing the pmGH convergence. \par 
    We say that the Borel sets $E_i \subset \X_i$ converge in $\sL^1_{loc}$ sense to a set $E \subset \ssf{Y}$ if $E_i \cap B_r(x_i) \to E \cap B_r(y)$ in $\sL^1$ sense for every $r>0$.
\end{definition}

The next proposition is taken from  \cite[Corollary $3.4$]{ABS19}.

\begin{proposition}
    Let $(\X_i,\sd_i,\m_i,\ssf{x}_i)$ be a sequence of pointed $\RCD(K,N)$ spaces converging in pmGH sense to $(\ssf{Y},\sd,\m,\ssf{y})$. Let $E_i \subset \X_i$ be Borel sets such that
    \[
    \sup_{i \in \bb{N}} P(E_i,B_r(x_i)) < + \infty \quad \text{for every } r>0.
    \]
    Then there exists a (non relabeled) subsequence and a Borel set $F \subset \ssf{Y}$ such that $E_i \to F$ in $\sL^1_{loc}$.
\end{proposition}

\begin{definition}[Reduced boundary]
    Let $E \subset \X$ be a set of finite perimeter. We say that a regular point $x \in \X$ such that $x \in \partial E$ is in the reduced boundary $\mathcal{F}E$ of $E$ if for every sequence $\{\epsilon_i\}_{i \in \bb{N}}$ decreasing to zero the sets $E_i:=E$ in the rescaled spaces $(\X,\epsilon_i^{-1}\sd,\m(B_{\epsilon_i}(x))^{-1}\m,x)$ converge in $\sL^1_{loc}$ to a half space in $\bb{R}^k$.
\end{definition}

The next proposition is taken from \cite[Corollary $3.15$]{BPSrec}.

\begin{proposition}
    Let $E \subset \X$ be a set of finite perimeter. Then the perimeter measure is concentrated on $\mathcal{F} E$.
\end{proposition}

The next proposition is taken from
\cite[Theorem $5.2$ and Proposition $6.1$]{BPS}. We first introduce some notation.
Let $A \subset \subset \X$ be a set of finite perimeter and let $g \in \Lip_{loc}(\X)$. Assume that $g$ has distributional Laplacian which is a finite Radon measure. Then there exist measures $\mu_1,\mu_2 << |D 1_A|$ such that as $t \to 0$ we have
\[
1_A \nabla P_t(1_A) \cdot \nabla g \to \mu_1
\quad
\text{and}
\quad
1_{{^c}A} \nabla P_t(1_A) \cdot \nabla g \to \mu_2
\]
in weak sense testing against functions in $\Lip_c(\X)$. We denote the density of $\mu_1$ and $\mu_2$ w.r.t. $|D1_A|$ respectively
by
\[
(\nabla g \cdot \nu_E)_{int}
\quad
\text{and}
\quad
(\nabla g \cdot \nu_E)_{ext}.
\]
Moreover we denote by $A^{(1)}$ the set of points of $\X$ where $A$ has density $1$.

\begin{proposition} \label{P30}
    Let $A \subset \subset \X$ be a set of finite perimeter and let $g \in \Lip_{loc}(\X)$. Assume that $g$ has distributional Laplacian which is a finite Radon measure. Then for any $f \in \Lip_c(\X)$ we have
    \[
    \int_{A^{(1)}} f \, d \Delta g+\int_A \nabla f \cdot \nabla g \, d\m= -\int_{\mathcal{F}A} f(\nabla g \cdot \nu_E)_{int} \, d \Per
    \]
    and
    \[
    \int_{A^{(1)} \cup \mathcal{F}A} f \, d \Delta g+\int_A \nabla f \cdot \nabla g \, d\m= -\int_{\mathcal{F}A} f(\nabla g \cdot \nu_E)_{ext} \, d \Per.
    \]
\end{proposition}

We now turn our attention to minimal sets.

\begin{definition}[Perimeter minimizing sets]
Let $\Omega \subset \X$ be an open set.
    Let $E \subset \Omega$ be a set of locally finite perimeter. We say that $E$ is globally perimeter minimizing in $\Omega$ (or simply perimeter minimizing in $\Omega$) if for every $x \in \Omega$, $r>0$ and $F \subset  \Omega$ such that $F \Delta E \subset \subset B_r(x) \cap \Omega$ we have that $P(E,B_r(x) \cap \Omega) \leq P(F,B_r(x) \cap \Omega)$. If we say that $E$ is perimeter minimizing we implicitly mean that $\Omega=\X$.
\end{definition}

\begin{definition}[Locally perimeter minizing sets]
Let $\Omega \subset \X$ be an open set.
    Let $E \subset \Omega$ be a set of locally finite perimeter. We say that $E$ is locally perimeter minimizing in $\Omega$ if for every $x \in \Omega$ there exists $r>0$ such that for every $F \subset \Omega$ such that $F \Delta E \subset \subset B_r(x) \cap \Omega$ we have $P(E,B_r(x) \cap \Omega) \leq P(F,B_r(x) \cap \Omega)$. 
    If we say that $E$ is locally perimeter minimizing we implicitly mean that $\Omega=\X$.
\end{definition}

\begin{remark} \label{Rboundary}
Locally perimeter minimizing sets admit both a closed and an open representative, and these have the same boundary which in addition coincides with the essential boundary and it is $\m$-negligible (see \cite{Dens}). Whenever we consider the boundary of a locally perimeter minimizing set, we will implicitly be referring to the boundary of its closed (or open) representative.
\end{remark}

\begin{definition}[Sets that are locally boundaries of perimeter minimizers] \label{D3}
We say that a set $C \subset \X$ is locally the boundary of a perimeter minimizing set if for every $x \in \X$ there exists an open neighbourhood $U_x$ of $x$ and a set $E \subset U_x$ minimizing the perimeter in $U_x$ such that $U_x \cap C=\partial E \cap U_x$.
\end{definition}

It is immediate to check that if $C \subset \X$ is locally the boundary of a perimeter minimizing set then it is closed.
The next proposition can be found in \cite{Weak} (see also \cite{lapb} for the extension to the collapsed case).
\begin{proposition} \label{P35}
    Let $(\X,\sd,\m)$ be an $\RCD(0,N)$ space, let $\Omega \subset \X$ be an open set and let $E \subset \Omega$ be the relatively closed representative of a locally perimeter minimizing set in $\Omega$. Let $\sd_E:{^c}E \to \bb{R}$ be the distance function from $E$. Then $\Delta \sd_E \leq 0$ in distributional sense on every open subset $\Omega' \subset \subset {^c}E \cap \mathcal{K}$, where
    \[
    \mathcal{K}:=\{x \in \X: \exists y \in \partial E \cap \Omega: \sd_{E}(x)=\sd(x,y) \}.
    \]
\end{proposition}

The next proposition is taken from \cite[Theorem $2.43$]{Weak}.

\begin{proposition} \label{P26}
Let $(\X_i,\sd_i,\m_i,\ssf{x}_i)$ be a sequence of $\RCD(K,N)$ spaces converging in pmGH sense to $(\ssf{Y},\sd,\m,\ssf{y})$.
    Let $E_i \subset \X_i$ be a sequence of Borel sets converging in $\sL^1_{loc}$ sense to $E \subset \ssf{Y}$.
    Assume that each $E_i$ is perimeter minimizing in $B_{r_i}(\ssf{x}_i)$ and that $r_i \uparrow + \infty$. Then $E$ is perimeter minimizing and in the metric space realizing the convergence we have that $\partial E_i \to \partial F$ in Kuratowski sense.
\end{proposition}

We conclude the section by recalling the global Bernstein property in Euclidean space. This is a classical result and the proof can be found in \cite[Theorem $17.4$]{Giusti}. 

\begin{thm} \label{T10}
    Let $E \subset \bb{R}^N$ be a non empty perimeter minimizing set contained in a half space. Then $E$ is itself a half space.
\end{thm}

\section{Green's functions in parabolic metric measure spaces} \label{S2}

The goal of this section is to prove Theorems \ref{CT4} and \ref{CT5}. The techniques that we use generally follow the lines of the arguments in \cite{BSflow,Weak,BScon}, although we work in higher generality since we do not assume any non collapsing condition on the space $\X$. Throughout the section we assume that we are working on a fixed $\RCD(K,N)$ space $(X,\sd,\m)$.

\begin{definition}[Parabolic $\RCD(K,N)$ spaces] \label{D1}
The space $(\X,\sd,\m)$ is said to be parabolic if $\int_1^{+ \infty} p_t(x,y) \, dt =+ \infty$ for every $(x,y) \in \X \times \X$.
\end{definition}

In the smooth setting there are many possible equivalent definitions of parabolicity. In Theorem \ref{CT4} we show that also in the $\RCD(K,N)$ setting the previous definition is equivalent to the more classical one involving non existence of Green's functions. In the $\RCD(K,N)$ setting it is convenient to start from Definition \ref{D1} because it allows to exploit the good properties of the heat flow. 

\begin{proposition} \label{P10}
    The space $(\X,\sd,\m)$ is parabolic if and only if there exists $(x,y) \in \X \times \X$ such that $\int_1^{+ \infty} p_t(x,y) \, dt =+ \infty$.
    \begin{proof}
    It follows immediately by Proposition \ref{Harnack}.
    \end{proof}
\end{proposition}

We now define Green's functions on $\RCD(K,N)$ spaces.
\begin{definition}[Regular sets]
    Let $A \subset \subset \X$ be a regular set and let $x \in A$. We say that 
    \[
    G_x \in \W^{1,1}_0(A) \cap C(\bar{A} \setminus \{x\})
    \]
    is a Green's function on $A$ with pole $x$ if 
    $\Delta G_x=-\delta_x$ 
    in distributional sense in $A$ and $G_x=0$ on $\partial A$.
    Similarly we say that 
    \[
    G_x \in \W^{1,1}_{loc}(\X) \cap C(\X \setminus \{x\})
    \]
    is a Green's function on $\X$ with pole $x$ if 
    $\Delta G_x=-\delta_x$ 
    in distributional sense.
\end{definition}

Note that our definition of Green's functions on proper subsets of $\X$ requires a pointwise boundary condition as this will be important later on in the proof of Theorem \ref{CT1}. This is also the reason why we only define Green's functions on regular sets and not on arbitrary bounded precompact sets. We will indeed show that with our definition if a bounded precompact set admits a Green's function then it is regular and viceversa. \par 
The next proposition shows that a regular set admits at most one Green's function.

\begin{proposition}[Uniqueness of Green's functions]
    Let $A \subset \X$ be a regular set and let $x \in A$. Then there exists at most one Green's function with pole $x$ on $A$.
    \begin{proof}
        Let $f_1,f_2$ be two such functions. Then $f_1-f_2$ is harmonic in distributional sense, so that by Proposition \ref{P36} it belongs to $\W^{1,2}_{loc}(A)$.
        This implies that for every $\Omega \subset \subset A$ we have $f_1-f_2 \in D(\Delta,\Omega)$ and $\Delta (f_1-f_2)=0$. In particular, since we have that $f_1-f_2 \in C(\bar{A} \setminus \{x\})$, we obtain that $f_1-f_2 \in C(\bar{A})$. \par
        Suppose now that $f_1-f_2 \neq 0$, so that there exists a point where such difference is positive (or negative). Restricting to an appropriate precompact set of $A$ and using the maximum principle we obtain a contradiction.
    \end{proof}
\end{proposition}

Now we give a series of lemmas that are needed to prove existence of Green's functions on regular domains, which is the content of Proposition \ref{P1}. The proof of the next lemma is immediate.

\begin{lemma} \label{L5}
     Let $x \in \X$ and let $r>0$. If $f:\X \to \bb{R}$ satisfies $\lip(f)(y) \leq c$ for every $y \in B_r(x)$, then $f$ is Lipschitz in $B_{r/4}(x)$.
\end{lemma}

Let $T>0$ and define $G^T_x:\X \to \bb{R}$ by $G^T_x(y):=\int_0^Tp_t(x,y) \, dt$.
The function $G^T_x$ is summable as 
        \[
        \int_\X G^T_x(y) \, d\m(y)=\int_0^T \int_\X p_t(x,y) \, d\m(y) \, dt=T.
        \]
        As a consequence, $G^T_x$ is real valued $\m$-a.e. on $\X$ and the definition is well posed.

\begin{lemma} \label{L16}
 The function $G^T_x$ is locally Lipschitz in $\X \setminus \{x\}$.
    \begin{proof}
         Let $y \in \X \setminus \{x\}$. Thanks to Proposition \ref{Hestim} there exist constants $c_1,c_2>0$ such that for every $t \in (0,T)$ the function $p_t(x,\cdot)$ is Lipschitz in $B_{\sd(x,y)/8}(y)$ with local Lipschitz constant bounded from above by
        \[
        \frac{c_1}{\sqrt{t} \m(B_{\sqrt{t}}(x))} \exp{ \Big\{ \frac{-\sd^2(x,y)}{c_2t} \Big\}}.
        \]
        As a consequence, for every $z \in B_{\sd(x,y)/8}(y)$, using Fatou Lemma we get
        \begin{equation} \label{E7}
        \lip(G^T_x)(z) = \limsup_{z_1 \to z} \int_0^T\frac{|p_t(x,z_1)-p_t(x,z)|}{\sd(z_1,z)} \, dt
        \leq \int_0^T \frac{c_1}{\sqrt{t} \m(B_{\sqrt{t}}(x))} \exp{ \Big\{ \frac{-\sd^2(x,y)}{c_2 t} \Big\}} \, dt .
        \end{equation}
       Moreover from Remark \ref{R2} we have that there exists a constant $C>0$ such that $\m(B_r(x)) \geq Cr^N$ for every $r \leq \sqrt{T}$, so that the r.h.s. of \eqref{E7} is bounded. Hence $G^T_x$ is locally Lipschitz in a neighbourhood of $y \in \X \setminus \{x\}$ because of Lemma \ref{L5}. Since $y$ was arbitrary $G^T_x$ is locally Lipschitz in $\X \setminus \{x\}$.
       \end{proof}
\end{lemma}

\begin{lemma} \label{L7}
         $G^T_x \in \W^{1,1}_{loc}(\X)$.
         \begin{proof} 
         Consider for every $\epsilon >0$ the function $G^T_{\epsilon,x}:\X \to \bb{R}$ given by
        \[
        G^T_{\epsilon,x}(y)
        :=\int_\epsilon^Tp_t(x,y) \, dt.
        \]
        It is easy to see that these converge in $\sL^1(\X)$ to $G^T_x$ and, using Proposition \ref{Hestim} and Fatou Lemma, that for every $y \in \X$
       \[
        \lip(G^T_{\epsilon,x})(y)
        \leq \int_\epsilon^T \frac{c_1}{\sqrt{t} \m(B_{\sqrt{t}}(x))} \exp{ \Big\{ \frac{-\sd^2(x,y)}{c_2t} \Big\}} \, dt 
        \leq \int_\epsilon^T \frac{c_3}{\sqrt{t} \m(B_{\sqrt{t}}(x))} \, dt < + \infty.
        \]
        In particular $G^T_{\epsilon,x} \in \Lip(\X) \subset W^{1,1}_{loc}(\X)$.
       So if we prove that for every ball $B_r(x) \subset \X$ with $r \geq T+1$, the functions $G^T_\epsilon$ restricted to $B_r(x)$ are a Cauchy sequence in $\W^{1,1}(B_r(x))$, we obtain that the limit $G^T_x$ is in $\W^{1,1}_{loc}(\X)$ as well. So we compute
       \begin{align} 
       \begin{split} \label{E8}
       \int_{B_r(x)} |\nabla(G^T_{\epsilon_1,x}-G^T_{\epsilon_2,x})|(y) d\m(y) &=\int_{B_r(x)} \lip \Big(\int_{\epsilon_1}^{\epsilon_2}p_t(x,\cdot)dt \Big)(y)d\m(y) 
       \\
       &\leq \int_{B_r(x)} \int_{\epsilon_1}^{\epsilon_2} \frac{c_4}{\sqrt{t} \m(B_{\sqrt{t}}(x))} \exp{ \Big\{ \frac{-d^2(x,y)}{c_5t} \Big\}} \, dt \, d\m(y),
       \end{split}
       \end{align}
       where for the last inequality we passed the Lipschitz constant inside the integral using Fatou Lemma. Switching the integrals and applying a version of Fubini's Theorem (i.e. Cavalieri principle, see \cite{AmbFuscPall}) we then obtain that the expression in \eqref{E8} can be bounded as follows:
    \begin{align*}
        \int_{\epsilon_1}^{\epsilon_2} \int_{B_r(x)}  &\frac{c_4}{\sqrt{t} \m(B_{\sqrt{t}}(x))} \exp{ \Big\{ \frac{-d^2(x,y)}{c_5t} \Big\}} \, d\m(y) \, dt
    \\
    &=
        \int_{\epsilon_1}^{\epsilon_2} \frac{c_4}{\sqrt{t} \m(B_{\sqrt{t}}(x))} \int_0^1 \m \Big( \Big\{ \exp{ \Big\{ \frac{-\sd^2(x,y)}{c_5t} \Big\}} >s \Big\} \cap B_r(x) \Big ) \,  ds \, dt
    \\
       &\leq
       \int_{\epsilon_1}^{\epsilon_2} \frac{c_6}{\sqrt{t}} \sup_{z \in [0,T]} \int_0^{+ \infty} e^{-c_7 s}\frac{\m(B_{\sqrt{sz}}(x) \cap B_r(x))}{\m(B_{\sqrt{z}}(x))} \, ds \, dt.
       \end{align*}
       We now prove that the supremum in the previous expression is finite, keeping in mind that this is sufficient to prove that our sequence is Cauchy. We compute:
       \begin{align*}
       \sup_{z \in [0,T]} &\int_0^{+ \infty} e^{-c_7 s}\frac{\m(B_{\sqrt{sz}(x)} \cap B_r(x))}{\m(B_{\sqrt{z}(x)})} \, ds 
       \\
       &\leq
       \int_0^1 e^{-c_7s} \, ds
       + \sup_{z \in [0,T]} \Big(\int_1^{r^2/z} e^{-c_7 s}\frac{\m(B_{\sqrt{sz}(x)} )}{\m(B_{\sqrt{z}(x)})} \, ds
       + \int_{r^2/z}^{+ \infty} e^{-c_7 s}\frac{\m(B_r(x) )}{\m(B_{\sqrt{z}(x)})} \, ds \Big) 
       \\
       & \leq c_8
       + \sup_{z \in [0,T]} \Big(\int_1^{r^2/z} e^{-c_7 s}\frac{V_{K,N}(B_{\sqrt{sz}(x)} )}{V_{K,N}(B_{\sqrt{z}(x)})} \, ds
       + c_9 e^{-c_7r^2/z}z^{-N/2} \Big) 
       \end{align*}
       where the last inequality is a consequence of Remark \ref{R2}. It is clear that the finiteness of the last expression follows if we can prove the finiteness of 
       \[
       \sup_{z \in [0,T]} \int_1^{r^2/z} e^{-c_7 s}\frac{V_{K,N}(B_{\sqrt{sz}(x)} )}{V_{K,N}(B_{\sqrt{z}(x)})} \, ds.
       \]
       Since when $s \in (1,r^2/z)$ we have $\sqrt{zs} \in (0,r)$, there exists a constant $c_{10}$ depending on $r$ such that $V_{K,N}(B_{\sqrt{zs}(x)}) \leq c_{10}(sz)^{N/2}$ whenever $s \in (1,r^2/z)$. Hence combining with Remark \ref{R2} once again we get
       \[
       \sup_{z \in [0,T]} \int_1^{r^2/z} e^{-c_7s}\frac{V_{K,N}(B_{\sqrt{sz}(x)} )}{V_{K,N}(B_{\sqrt{z}(x)})} \, ds 
        \leq
       \sup_{z \in [0,T]} \int_1^{r^2/z} c_{11}e^{-c_7 s}s^{N/2}\, ds
       \leq 
       \int_1^{+ \infty} c_{11}e^{-c_7 s}s^{N/2}\, ds < + \infty.
       \]
       Concluding the proof that $G^T_x \in \W^{1,1}(B_r(x))$.
       \end{proof}
\end{lemma}

\begin{lemma} \label{L6}
    Let $\Omega \subset \X$ be an open set and let $x \in \Omega$ and $g \in \W^{1,1}_0(\Omega)$. If for every $\phi \in \Test_c(\Omega,\sd,\m)$ we have
       \[
       \int_\Omega g(y) \Delta \phi(y) \, d\m(y)=-\phi(x),
       \]
       then $\Delta g=-\delta_x$ in distributional sense in $\Omega$.
       \begin{proof}
           Let $\phi \in \Lip_c(\Omega)$ and consider its extension to $0$ in $\X$, which we still denote by $\phi$. Consider the heat flow  $P_t \phi \in \Test(\X,\sd,\m)$.
           Let now $\eta \in \Test_c(\Omega)$ be a positive function which is equal to $1$ on an open set containing the support of $\phi$. Since the test functions are an algebra, we get $\eta P_t\phi \in \Test_c(\Omega,\sd,\m)$. Moreover $\eta P_t\phi \to \phi$ in $\W^{1,2}(\X)$ and uniformly, while the functions $\eta P_t\phi$ are uniformly Lipschitz because $\phi$ is bounded. Hence using the Dominated Convergence Theorem we get
           \[
       \int_\Omega \nabla g(y) \cdot \nabla \phi(y) \, d\m(y)=
       \lim_{t \to 0} \int_\Omega \nabla g(y) \cdot \nabla (\eta P_t\phi) \, d\m(y)
       =
       -\lim_{t \to 0} \int_\Omega g(y)  \Delta (\eta P_t\phi) \, d\m(y)
       =
       \lim_{t \to 0} \eta P_t\phi (x)=\phi(x),
       \]
       concluding the proof.
       \end{proof}
\end{lemma}

\begin{lemma} \label{L11}
    Let $B \subset \X$ be a regular set, $x \in B$, $f \in C(\X)$, and $g \in \Lip_{loc}(\bar{B} \setminus \{x\})$. There exists $h \in \W^{1,2}(B) \cap C(\bar{B})$ such that $\Delta h=f$ on $B$, $h+g$ is identically zero on $\partial B$ and for every positive $\eta \in \Lip_c(B)$ equal to $1$ in a neighbourhood of $x$ and taking value in $[0,1]$ we have $(1-\eta)(h+g) \in \W^{1,2}_0(B)$. \par 
    Moreover, if in addition we have that $g \in \W^{1,2}(B)$, then we can ask that $h+g \in \W^{1,2}_0(B)$.
       \begin{proof}
           Let $U \supset \supset \bar{B}$ be an open precompact subset of $\X$ and consider the function $h_1 \in \W^{1,2}_0(U)$ such that $\Delta h_1=f$ on $U$. 
           By Proposition \ref{P7} this function is Lipschitz on $\bar{B}$. \par 
           Let now $\eta \in \Lip_c(B)$ be as in the statement, then $(1-\eta)g \in \W^{1,2}(B) \cap C(\bar{B})$ so that, $B$ being regular, there exists
            $h_2 \in \W^{1,2}(B) \cap C(\bar{B})$, harmonic on $B$, such that $h_2+h_1+g$ is identically zero on $\partial B$ and $h_1+h_2+(1-\eta)g \in \W^{1,2}_0(B)$. We then set $h:=h_2+h_1$ and we note that this function has all the required properties. \par 
            The case when $g \in \W^{1,2}(B)$ is analogous.
       \end{proof}
\end{lemma}

We are finally in position to prove the existence of Green's functions.

\begin{proposition}[Existence of Green's functions] \label{P1}
   Let $B \subset \X$ be a regular set and let $x \in B$. For every $T>0$ there exists $g_x^T \in \W^{1,2}(B) \cap C(\bar{B} )$ such that $G_x:=G^T_x+g^T_x$ is the Green's function on $B$ with pole $x$. \par 
   Moreover for every $\eta \in \Lip_c(B)$ equal to $1$ in a neighbourhood of $x$ and taking value in $[0,1]$ we have $(1-\eta) G_x \in \W^{1,2}_0(B)$.
    \begin{proof}
      Let $g^T_x \in \W^{1,2}(B) \cap C(\bar{B} )$ be the function given by Lemma \ref{L11} such that 
      \[
      \Delta g_x^T=-p_T(x,\cdot), \quad g_x^T+G_x^T=0 \quad \text{on } \partial B
      \]
      and such that for every $\eta$ as in the statement 
      $
      (1-\eta)(g_x^T+G_x^T) \in \W^{1,2}_0(B).
      $
       Now we define $G_x:B \to \bb{R}$ by $G_x(y):=G^T_x(y)+g^T_x(y)$ and we note that by construction $G_x \in \W^{1,1}_0(B) \cap C(\bar{B} \setminus \{x\})$.
       \par 
       We now check that $G_x$ is a Green's functions: let $\phi \in \Test_c(B,\sd,\m)$ and note that
       \begin{align*}
       \int_B & G^T_x(y) \Delta \phi(y) \,d\m(y)=\int_0^T \int_\X p_t(x,y)\Delta \phi(y)\, d\m(y) \, dt
       =\int_0^T P_t\Delta \phi(x) \,dt\\
       &=\int_0^T \Delta P_t \phi(x)\, dt= \int_0^T \frac{d}{dt} P_t \phi(x)\, dt= \int_Bp_T(x,y) \phi(y) \, d\m(y)-\phi(x),
       \end{align*}
       where the last inequality comes from the fact that the trajectories of the heat flow at a fixed point are locally absolutely continuous. From the previous computation and the definition of $G_x$ we deduce that for every $\phi \in \Test_c(B,\sd,\m)$ we have
       \[
       \int_B G_x(y) \Delta \phi(y) \, d\m(y)=-\phi(x),
       \]
       which by Lemma \ref{L6} implies that
       $G_x$ is a  Green's function for the Laplacian.
    \end{proof}
\end{proposition}

We have two corollaries.
\begin{corollary} \label{CcontA}
    Let $B \subset \X$ be regular, let $x \in B$, and let $G_x$ be the Green's function on $B$ with pole $x$. Then, there exists a continuous representative of $G_x$ with values in $[0,+\infty]$.
    \begin{proof}
        Taking into account Proposition \ref{P1}, it is enough to prove the claim for $G_x^T:=\int_0^Tp_t(x,\cdot)\, dt$. If $\int_0^T p_t(x,x)<+\infty$, then $G_x^T$ is continuous with values in $\bb{R}$ by Proposition \ref{Hestim} and Dominated Convergence Theorem. If instead $\int_0^T p_t(x,x)=+\infty$, then Proposition \ref{Hestim} implies
        \[
        \int_0^T \frac{1}{\m(B_{\sqrt{t}}(x))} \, dt=+\infty.
        \]
        Applying Proposition \ref{Hestim} once again, it follows that $G_x^T(y) \to +\infty$ as $y \to x$.
    \end{proof}
\end{corollary}

\begin{corollary} \label{C3}
    The function $G_x$ constructed in Proposition \ref{P1} is positive.
    \begin{proof}
    We would like to use the maximum principle on $G_x$, but this function is only $\W^{1,1}(B) \cap C(\bar{B}  \setminus \{x\})$, while we would need a function in $\W^{1,2}(B) \cap C(\bar{B})$ to apply such theorem. \par
        For every $\epsilon>0$ consider the function $G^\epsilon_x \in \W^{1,2}(B) \cap C(\bar{B})$ given by
        \[
        G_x^{\epsilon}(y):=
        \int_\epsilon^T p_t(x,y) \, dt + g_x^T(y).
        \]
        We have that $\Delta G_x^{\epsilon}=-p_\epsilon(x,\cdot)$ and $G_x^{\epsilon}=-\int_0^{\epsilon}p_t(x,\cdot) \, dt$ on $\partial B$. \par 
        By the maximum principle, this function attains its minimum on the boundary, so that on $B$ we have
        \begin{equation} \label{E9}
        G_x^{\epsilon} \geq \min_{y \in \partial B} -\int_0^{\epsilon}p_t(x,y) \, dt 
        \end{equation}
        At the same time it is easy to check that as $\epsilon$ decreases to zero also the quantity in the right hand side of \eqref{E9} goes to zero. Moreover $G_x^{\epsilon} \to G_x$ pointwise in $B \setminus \{x\}$ as $\epsilon \to 0$, so that passing to the limit in \eqref{E9} we conclude.
    \end{proof}
\end{corollary}

 As a consequence of this corollary we get the following proposition.
 
\begin{proposition} \label{P16}
    Let $B_1 \subset \subset B_2$ be regular sets, let $x \in B_1$ and let $G^1_x$ and $G^2_x$ be the corresponding Green's functions. Then $G^2_x \geq G^1_x$.
    \begin{proof}
        By Proposition \ref{P36} the restriction to $\bar{B}_1$ of the difference  $G^2_x-G^1_x$ lies in $W^{1,2}_{loc}(B_1) \cap C(\bar{B}_1)$. \par 
        At the same time $G^2_x-G^1_x \geq 0$ on $\partial B_1$ by Corollary \ref{C3}, so that using the maximum principle we conclude.
    \end{proof}
\end{proposition}

The next proposition shows that we can use the Green's function to construct functions with prescribed Laplacian.

\begin{proposition}[Representation of functions with prescribed Laplacian] \label{P20}
    Let $B \subset \X$ be regular, and let $\phi \in \Lip_c(B)$. Then 
    \[
    f(x):=\int_B G_x(y) \phi(y) \, d\m(y)
    \]
    is the unique function such that $f \in D(\Delta,B) \cap \W^{1,2}_0(B) \cap  C(\bar{B})$ with $\Delta f=-\phi$ and $f=0$ on $\partial B$.
    \begin{proof}
        By Lemma \ref{L11} we have that there exists a function 
        \[
        f \in D(\Delta,B) \cap W^{1,2}_0(B) \cap  C(\bar{B})
        \]
        with $\Delta f=-\phi$ and $f=0$ on $\partial B$. By the maximum principle such function is unique.
        Since $f \in \W^{1,2}_0(B)$ it is easy to check that there exists a sequence $\{f_i\}_{i \in \bb{N}} \subset \Lip_c(B)$ of functions such that $\|f_i-f\|_{\W^{1,2}(B)} \to 0$ and such that for every $K \subset \subset B$ we have $f_i = f$ on $K$ for $i$ sufficiently large.
        In particular, since $f_i \in \Lip_c(B)$ and $G_x$ is a Green's function with pole $x$, we have
        \[
        f_i(x)=\int_B \nabla G_x(y) \cdot \nabla f_i (y)\, d\m(y),
        \]
        so that to conclude it is sufficient to prove that as $i \to + \infty$ we have
        \begin{equation} \label{E14}
        \int_B \nabla G_x(y) \cdot \nabla f_i(y)\, d\m(y) \to -\int_B G_x(y) \Delta f (y)\, d\m(y).
        \end{equation}
        To this aim fix $\epsilon >0$ and let $\eta \in \Lip(B)$ be a function taking value in $[0,1]$ which is equal to zero in a neighbourhood of $x$, equal to $1$ out of a precompact open subset $U$ of $B$ and such that
        \[
        \Big|\int_B G_x(y) \Delta f (y)\, d\m(y)
        -\int_{B }  (\eta G_x)(y) \Delta f (y)\, d\m(y) \Big| < \epsilon,
        \]
        and for every $i \in \bb{N}$ sufficiently large
        \[
        \Big| \int_B \nabla G_x(y) \cdot \nabla f_i (y)\, d\m(y)
        -\int_{B }  \nabla (\eta G_x)(y) \cdot \nabla f_i (y)\, d\m(y) \Big| < \epsilon.
        \]
        Note that such $\eta$ can be constructed because $\Delta f \in \sL^{\infty}(B)$ and, for $i$ sufficiently large, we have that 
        \[
        |\nabla f_i|=|\nabla f| \text{ in } U.
        \]
        In particular, thanks to the arbitrariness of $\epsilon$, to obtain \eqref{E14} it is now sufficient to show that
        \[
        \int_{B }  \nabla (\eta G_x)(y) \cdot \nabla f_i (y)\, d\m(y) \to -\int_{B }  (\eta G_x)(y) \Delta f (y)\, d\m(y) \quad \text{as } i \to + \infty.
        \]
        Since $\eta G_x \in \W^{1,2}_0(B)$ by Proposition \ref{P1}, we can rewrite the difference of the previous quantities as
         \begin{align*}
        \Big|-&\int_{B }  \nabla (\eta G_x)(y) \cdot \nabla f_i (y)\, d\m(y) - \int_{B }  (\eta G_x)(y) \Delta f (y)\, d\m(y)\Big| \\
        &=
        \Big|-\int_{B } \nabla (\eta G_x) \cdot \nabla f_i (y)\, d\m(y) + \int_{B } \nabla (\eta G_x) \cdot \nabla f (y)\, d\m(y) \Big|\\
       & \leq \|\eta G_x\|_{\W^{1,2}(B)}\|f_i-f\|_{\W^{1,2}(B)} \to 0,
        \end{align*}
        concluding the proof.
    \end{proof}
\end{proposition}

We now turn our attention to constructing a global Green's function for $\X$.

\begin{lemma} \label{L8}
Let $(\X,\sd,\m)$ be non parabolic. Let $S \subset \subset \X$ and $x \in \X$. There exists a sequence $t_i \to + \infty$ such that $p_{t_i}(x,\cdot)$ and $\int_1^{t_{i}+1}p_t(x,\cdot ) \, dt$ are uniformly bounded in $S$.
\begin{proof}
    The non parabolicity of $\X$ implies that there exists a sequence $t_i \to + \infty$ such that $p_{t_i}(x,x) \leq 1$. By Proposition \ref{Harnack} there exists a constant $C>0$ such that $p_{t}(x,y) \leq Cp_{t+1}(x,x)$ for every $y \in S$ and $t>1$. Putting these facts together it is easy to check that the sequence $\{t_i\}_{i \in \bb{N}}$ has the desired properties.
    \end{proof}
\end{lemma}

\begin{proposition}[Existence of global Green's functions on non-parabolic spaces] \label{P11}
    Let $(\X,\sd,\m)$ be non parabolic and let $x \in \X$. Then the function $G_x(y):=\int_0^{+ \infty}p_t(x,y) \, dt$ is a Green's function for $\X$ with pole $x$.
    \begin{proof}
        First of all we prove that $G_x \in \sL^1_{loc}(\X)$. Let $S \subset \Omega$ be a compact set. By Proposition \ref{Harnack} we deduce that there exists $C>0$ depending on $K,N,S$ such that $Cp_t(x,x) \geq p_{t-1}(x,y)$ for every $t>1$ and every $y \in S$. As a consequence for every $y \in S$ we have
        \begin{equation} \label{E3}
    \int_1^{+\infty}p_t(x,y) \, dt \leq C\int_2^{+\infty}p_t(x,x) \, dt < + \infty,
        \end{equation}
        so that by Dominated Convergence Theorem and the Hölder continuity of the heat kernel the function $\int_1^{+\infty}p_t(x,\cdot) \, dt$ is continuous in $S$.
        Using such continuity and the stochastic completeness of the heat kernel we then get
       \[
       \int_S \int_0^{+ \infty}p_t(x,y) \, dt \, d\m(y)
       \leq 
       \int_0^1 \int_\X p_t(x,y) \, d\m(y) \, dt + \m(S) \max_{y \in S}\int_1^{+\infty}p_t(x,y) \, dt 
       \leq 1+\m(S) \max_{y \in S}\int_1^{+\infty}p_t(x,y) \, dt,
       \] 
       showing that $G_x \in \sL^1_{loc}(\X)$, as desired. \par 
       We now claim that for every $\phi \in \Test(\X,\sd,\m) \cap \Lip_c(\X)$ we have 
       \begin{equation} \label{E31}
       \int_\X G_x(y) \Delta \phi(y) \, d\m(y)
        =-\phi(x).
       \end{equation}
       To this aim fix such a $\phi$ and observe that
       \begin{align*}
       \int_\X G_x(y) \Delta \phi(y) \, d\m(y) &=\int_0^{+ \infty} \int_\X p_t(x,y)\Delta \phi(y)\, d\m(y) \, dt
       \\
       &=\int_0^{+ \infty} P_t\Delta \phi(x) \,dt
       =\int_0^{+ \infty} \Delta P_t \phi(x)\, dt= \int_0^{+ \infty} \frac{d}{dt} P_t \phi(x)\, dt
       \\
       &= \lim_{t \to + \infty} \int_\X p_t(x,y) \phi(y) \, d\m(y)-\phi(x).
       \end{align*}
       We now claim that the limit appearing in the last expression is zero (we know that the limit exists since the integral appearing at the beginning of the computation is well defined). To prove the claim we note that since $\phi$ is supported in a compact set $S$ we have
       \begin{equation} \label{E10}
       \int_1^{+ \infty} \Big|\int_\X p_t(x,y) \phi(y) \, d\m(y) \Big| \, dt \leq c_1 \m(S) \sup_{y \in S}\int_1^{+\infty}p_t(x,y)\, dt
       \leq c_2 \m(S) \int_1^{+\infty}p_t(x,x)\, dt < + \infty,
       \end{equation}
       where in the last inequality we used Proposition \ref{Harnack}. As a consequence there exists a sequence $\{t_i\}_{i \in \bb{N}}$ increasing to $+ \infty$ such that $\int_\X p_{t_i}(x,y) \phi(y) \, d\m(y)  \to 0$, as claimed. \par 
       To conclude the proof we only need to prove that 
       \begin{equation} \label{E32}
       G_x \in \W^{1,1}_{loc}(\X) \cap C(\X \setminus \{x\}).
       \end{equation}
       Indeed if this holds, we can combine equation \eqref{E31} and Lemma \ref{L6} (which can be used in our setting via a cut off argument) to obtain that $G_x$ is a Green's function for $\X$ with pole $x$.
        \par
       To prove \eqref{E32} it is sufficient to prove that $\int_1^{+\infty}p_t(x,\cdot) \, dt \in \W^{1,1}_{loc}(\X) \cap C(\X \setminus \{x\})$ since $\int_0^{1}p_t(x,\cdot) \, dt$ lies in such space by Lemmas \ref{L16} and \ref{L7}.
       Consider the family of functions indexed by $T$ given by $\int_1^{T}p_t(x,\cdot) \, dt$ and observe that these are in $\sL^1(\X) \cap \W^{1,1}_{loc}(\X)$ and converge to $\int_1^{+\infty}p_t(x,\cdot) \, dt$ in $\sL^1_{loc}(\X)$ by repeating the argument of \eqref{E10}, so that if we prove that their Lipschitz constants are locally uniformly bounded we are done. \par 
       To this aim we recall that for every $T>0$
       we have the distributional equation
       \[
       \Delta \int_1^{T}p_t(x,\cdot) \, dt=p_T(x,\cdot)-p_1(x,\cdot).
       \]
       Fix then a compact set $S$, by Lemma \ref{L8} there exists a sequence $\{t_i\}_{i \in \bb{N}}$ increasing to $+ \infty$ such that
       \[
       \Delta \int_1^{t_{i}+1}p_t(x,\cdot ) \, dt=p_{t_i+1}(x,\cdot)-p_1(x,\cdot) \quad \text{and} \quad 
       \int_1^{t_{i}+1}p_t(x,\cdot ) \, dt
       \]
       are uniformly bounded in the compact $S$. 
    As a consequence, by the a priori estimates on the gradient of solutions of the Poisson problem given in Proposition \ref{P9}, we deduce that the functions $\int_1^{t_i+1}p_t(x, \cdot ) \, dt$ have locally uniformly bounded Lipschitz constants as desired.
    \end{proof}
\end{proposition}

\begin{remark}
    The function $G_x$ that we constructed is not only continuous  but also locally Lipschitz in $\X \setminus \{x\}$.
\end{remark}

\begin{definition}[Capacity]
    Let $(\X,\sd,\m)$ be an $\RCD(K,N)$ space. Let $ A \subset \X$ be open and let $K \subset A$ be compact. The capacity of $K$ in $A$ is defined as
    \[
    \ssf{Cap}(K,A):=\inf \Big\{ \int_A |\nabla u|^2 \, d\m \,: u \in \Lip_c(A), \, u=1 \text{ on K} \Big\}.
    \]
    The capacity of $K$ in $\X$ is denoted $\ssf{Cap}(K)$.
\end{definition}

The proof of the next result follows the lines of the corresponding one in the smooth setting (see, for instance, \cite{grig}). We remark that, if $A \subset \X$ is a regular set, $x \in A$, and $G_x$ is the Green's function on $A$ with pole $x$, then for every $c>0$ the set $\{G_x \geq c\}$ is compact in $A$ by Corollary \ref{CcontA}.

\begin{lemma} \label{LCap}
    Let $A \subset \X$ be a regular set, let $x \in A$ and let $G_x$ be the corresponding Green's function. Let $c>0$, let $F_c:=\{G_x \geq c\}$, and assume that $F_c$ has finite perimeter in $\X$. Then, $\ssf{Cap}(F_c,A)=c^{-1}$.
    \begin{proof}
        The function $u:=1 \wedge (c^{-1}G_x)$ is harmonic in $A \setminus F_c$, it is equal to $0$ on $\partial A$, and it is equal to $1$ on $\partial F_c$. Hence, since harmonic functions minimize the Dirichlet energy, it holds 
        \[
        \ssf{Cap}(F_c,A)=\int_{A \setminus F_c} |\nabla u|^2 \, d\m=c^{-2} \int_{A \setminus F_c} |\nabla G_x|^2 \, d\m.
        \]
        Let $0<\epsilon<c$ be such that $\{G_x < \epsilon\}$ has finite perimeter.
        The vector field $\nabla G_x$ is defined in a neighbourhood of $\{\epsilon < G_x < c\}$, where it is bounded and admits bounded divergence, so that we can apply Proposition \ref{P30} to get
    \[
    \int_{\{\epsilon < G_x < c\}}|\nabla G_x|^2 \, d\m
    =-c \int_{\partial \{G_x < c \} } (\nabla G_x \cdot  \nu_{\{G_x < c\}})_{ext} \, d \Per
    -\epsilon \int_{\partial \{G_x > \epsilon \}} (\nabla G_x \cdot  \nu_{\{\epsilon < G_x \}})_{ext} \, d \Per.
    \]
    Using the definitions, it is then easy to check that
    \[
    \int_{\partial \{G_x < c \} } (\nabla G_x \cdot  \nu_{\{G_x < c\}})_{ext} \, d \Per=
    -\int_{\partial \{ G_x \geq c \} } (\nabla G_x \cdot  \nu_{\{G_x \geq c \}})_{int} \, d \Per,
    \]
    so that, taking into account that $\Delta G_x=-\delta_x$ and using again Proposition \ref{P30}, we obtain
    \[
    \int_{\{\epsilon < G_x < c\}}|\nabla G_x|^2 \, d\m
    =c \int_{\partial \{ G_x \geq c\}} (\nabla G_x \cdot  \nu_{\{G_x \geq c \}})_{int} \, d \Per
    -\epsilon \int_{\partial \{G_x > \epsilon \}} (\nabla G_x \cdot  \nu_{\{G_x >\epsilon \}})_{ext} \, d \Per
    =c-\epsilon.
    \]
    Letting $\epsilon \to 0$, we obtain $ \int_{A \setminus F_c} |\nabla G_x|^2 \, d\m=c$, so that $\ssf{Cap}(F_c,A)=c^{-1}$ as claimed.
    \end{proof}
\end{lemma}

We are now in position to prove Theorem \ref{CT4}, that we recall below.

\begin{customthm}{4} 
    Let $(\X,\sd,\m,\ssf{x})$ be an $\RCD(K,N)$ metric measure space with infinite diameter. The following are equivalent:
    \begin{enumerate}
        \item There is no positive Green's function on $\X$ with pole $\ssf{x}$.
        \item For every $(x,y) \in \X \times \X$ we have $\int_1^{+ \infty} p_t(x,y) \, dt =+ \infty$.
        \item There exists $(x,y) \in \X \times \X$ such that $\int_1^{+ \infty} p_t(x,y) \, dt =+ \infty$.
        \item Let $\{B_i\}_{i \in \bb{N}}$ be an exhaustion of regular sets containing $\ssf{x}$. Let $G_\ssf{x}^i$ be the Green's function on $B_i$ with pole $\ssf{x}$. Then for every $y \neq \x$ we have $G_\ssf{x}^i(y) \to + \infty$ as $i \to + \infty$.
        \item 
        For every compact set $K \subset X$, it holds $\ssf{Cap}(K)=0$.
    \end{enumerate}
    \end{customthm}
    \begin{proof}
    The equivalence $2 \Leftrightarrow3$ has been shown in Proposition \ref{P10}, while $1 \Rightarrow 2$ was proved in Proposition \ref{P11}. So we only need to show that $2 \Rightarrow 4 \Rightarrow 5 \Rightarrow 1$. \par 
    \textbf{$(2 \Rightarrow 4)$} 
    By Proposition \ref{P1}, for every $i \in \bb{N}$ the Green's function $G^i_\x$ can be written as $G^i_\x(y)=\int_0^{T_i}p_t(\x,y)dt+g^{T_i}_\x(y)$ where $g^{T_i}_\x \in \W^{1,2}(B_i) \cap C(\bar{B_i})$ satisfies
        \[
       \Delta g^{T_i}_\x=-p_{T_i}(\x,\cdot) \quad \text{on } B_i
        \quad 
       \text{and} 
       \quad
       g^{T_i}_\x=-\int_0^{T_i}p_t(\x,\cdot) \quad\text{on } \partial B_i.
       \]
       Taking into account the parabolicity assumption, it is sufficient to prove that we can choose a sequence $T_i$ increasing to $+ \infty$ in such a way that $g^{T_i}_\x(y_0)$ is bounded from below uniformly in $i$. \par 
       Since $\Delta g^{T_i}_\x \leq 0$, by the maximum principle it is sufficient to consider the restriction of $g^{T_i}_\x$ to the boundary of $B_i$, where the minimum is attained. So the problem reduces to choosing $T_i$ increasing to $+ \infty$ in such a way that $\int_0^{T_i}p_t(\x,\cdot)\, dt$ is uniformly bounded when $y \in \partial B_i$. \par 
       This is clearly possible if we can show that for $T$ fixed, the quantity 
       \[
       \sup_{y \in \partial B_i}\int_0^{T}p_t(\x,y)\, dt
       \]
       converges to zero as $i \to + \infty$, and this follows from Proposition \ref{Hestim} and the hypothesis on the diameter. \par
    
       \textbf{$(4 \Rightarrow 5)$} Let $B \subset \X$ be a ball, and let $x \in B$. Let $\{B_i\}_{i \in \bb{N}}$ be an exhaustion of regular sets containing $B$, 
       let $G^i_x$ be the Green's functions relative to $B_i$, and let $c_i:=\min_{\partial B} G^i_x$. 
       By our assumption (combined with Proposition \ref{P16} and Dini's Lemma), it holds $c_i \to + \infty$. Set $F_i=\{G_x^i \geq c_i\}$. 
       By the maximum principle, it holds $B \subset F_i$. 
       Hence, $\ssf{Cap}(\bar{B}) \leq \ssf{Cap}(F_i)$. Combining with Lemma \ref{LCap}, it follows $\ssf{Cap}(\bar{B}) \leq c_i^{-1}$, concluding the proof of this implication. \par
            {\textbf{$(5 \Rightarrow 1)$}
       Assume by contradiction that there exists a positive Green's function $G_x$ on $\X$. Let $B \subset \X$ be a ball, and let $x \in B$. Let $\{B_i\}_{i \in \bb{N}}$ be an exhaustion of regular sets containing $B$, let $G^i_x$ be the Green's functions relative to $B_i$, and let $c_i:=\max_{\partial B} G^i_x$. Set $F_i=\{G_x^i \geq c_i\}$. 
       By the maximum principle, it holds $B \supset F_i$. 
       Hence, by Lemma \ref{LCap},
       \[
       \ssf{Cap}(\bar{B},B_i) \geq \ssf{Cap}(F_i,B_i)=c_i^{-1}.
       \]
        Since $\ssf{Cap}(\bar{B})=0$ by assumption, it follows 
        \[
        \lim_{i \to + \infty} c_i=+ \infty.
        \]
         Hence, $\max_{\partial B} G^i_x \to + \infty$. At the same time, it holds $G_x \geq G_x^i$ on $B_i$ by the maximum principle, a contradiction.}
    \end{proof}

Now we turn our attention to proving Theorem \ref{CT5}. To this aim we need the following technical lemma.

\begin{lemma} \label{L15}
    Let $(\X,\sd,\m)$ be a non parabolic $\RCD(K,N)$ space with infinite diameter and essential dimension $k \geq 2$. Let $x \in \X$ be a regular point such that
    \[
    \lim_{r \to 0^+}\frac{\m(B_r(x))}{r^k} \in (0,+\infty).
    \]
    Consider an exhaustion $\{\Omega_i\}_{i \in \bb{N}}$ of regular sets of $\X$ containing $x$, and for each of these sets consider the corresponding Green's function $G^i_x$. 
    \par 
    These functions admit continuous representatives with values in $[0,+\infty]$ such that $G_x^i(x)=+\infty$. Using these representatives, there exist $T>1$ 
    and a sequence $\{c_i\}_{i \in \bb{N}} \in (1,T)$  such that
    \[
    x \in \{y \in \Omega_i:G^i_x (y)> c_i\} \cap B_1(x) \subset \subset B_1(x) \quad \text{ for every } i \in \bb{N}
    \]
    and each of the sets $\{y \in \Omega_i:G^i_x (y)> c_i\}$ is open, has finite perimeter in $\X$ and satisfies 
    \[
    \partial\{y \in \Omega_i:G^i_x (y)> c_i\}=\partial \{y \in \Omega_i:G^i_x (y) < c_i\}=\{y \in \Omega_i:G^i_x (y) = c_i\}.
    \]
    \begin{proof}
    We will write $\{G^i_x > c_i\}$ instead of $\{y \in \Omega_i:G^i_x (y)> c_i\}$.
        Consider the Green's function $G_x$ given by Proposition \ref{P11}. Since $x$ is a regular point and $\X$ has essential dimension $k \geq 2$, using the bounds of Proposition \ref{Hestim} and the explicit expressions for $G^i_x$ and $G_x$, we get
        \begin{equation} \label{E25}
        \lim_{y \to x} G_x(y)=\lim_{y \to x} G^i_x(y)=+\infty \quad \text{for every } i \in \bb{N}.
        \end{equation}
        Hence, we have the desired continuous representatives.
        Moreover, since $G_x$ is bounded on compacts of $\X \setminus \{x\}$ we can pick $t>0$ such that $\{G_x > t \} \cap B_1(x)$ contains $x$ and
        \[
        \{G_x > t \} \cap B_1(x) \subset \subset B_1(x).
        \]
        At the same time, using again \eqref{E25} and the fact that $G_x \geq G_x^i$, we get that for every $t' \in [t,t+1]$ we have 
        \[
        \{G^i_x > t' \} \cap B_1(x) \subset \{G_x > t \} \cap B_1(x) \subset \subset B_1(x),
        \]
        and that $\{G^i_x > t' \}$ is open and contains $x$.
        Moreover, since $G^i_x$ is locally Lipschitz on compact subsets of $\X \setminus \{x\}$ (and tends to infinity in $x$, so that the level sets stay away from $x$), for $\lambda^1$-a.e. $t' \in [t,t+1]$ we have that $\{G^i_x > t' \}$ has finite perimeter and $\m(\{G^i_x = t' \})=0$. We claim that this implies 
        \begin{equation} \label{E26}
    \partial \{G^i_x > t'\}=\partial \{G^i_x < t'\}=\{G^i_x = t'\}.
    \end{equation}
    For the first equality we only show that $\partial \{G^i_x > t'\} \subset \partial \{G^i_x < t'\}$, the other case being analogous. Let $y \in \partial \{G^i_x > t'\}$ and consider a small ball $B_\epsilon(y)$ which is precompact in $X \setminus \{x\}$. 
    Observe that $G_x^i(y)=t'$ and that the ball cannot be contained in $\{G^i_x = t' \}$, because otherwise this set would have positive measure.
    At the same time it cannot be contained in $\{G^i_x \geq t' \}$, otherwise the maximum principle would imply that $G_x^i$ is constant in the ball. As a consequence such ball intersects $\{G^i_x < t' \}$, and since $\epsilon$ was arbitrarily small, we have that $y \in \partial \{G^i_x < t' \}$, as claimed. \par 
    The second equality in \eqref{E26} follows because $\{G^i_x = t'\}$ cannot contain any balls since it has zero measure, and so it is contained in $\partial \{G^i_x \neq t'\}$, but since $\partial \{G^i_x > t'\}=\partial \{G^i_x < t'\}$ we get
    \[
    \partial \{G^i_x \neq t'\}=
    \partial \{G^i_x > t'\}
    =\partial \{G^i_x < t'\}.
    \]
    The fact that $\partial \{G^i_x < t'\} \subset \{G^i_x = t'\}$ is trivial.  
    Then we define $T:=t+1$ and for every $i \in \bb{N}$ we pick $c_i \in [t,t+1)$ to be one of the $t'$ with the previously proved properties. 
    \end{proof}
\end{lemma}

The next proposition is the key step to prove Theorem \ref{CT5}. The idea is adapted from the proof of the analogous fact in the smooth setting given in \cite{LTgreen}, but some crucial adjustments have been made to avoid integrating by parts on sets that might not be regular enough. This is achieved by replacing balls in the original proof with level sets of Green's functions.

\begin{proposition}
Let $(\X,\sd,\m)$ be an $\RCD(K,N)$ space with essential dimension $k$. If
    there exists $x \in \X$ such that 
    \[
    \lim_{r \to 0^+}\frac{\m(B_r(x))}{r^k} \in (0,+\infty) \quad 
    \text{and}
    \quad    
    \int_1^{+ \infty} \frac{1}{P(B_t(x),\X)} \, dt=+ \infty,
    \]
    then $\X$ is parabolic.
    \begin{proof} 
    Consider first the case when $\X$ has finite diameter and assume by contradiction that it is non parabolic. Note that the finiteness of the diameter implies that $\X$ is compact.
    Consider then the functions $g_\epsilon(y):=\int_\epsilon^{+\infty}p_t(x,y) \, dt$, which converge in $\sL^1(\X)$ to the Green's function $G_x$ given by Proposition \ref{P11}. Using ideas appearing in previous proofs one can show that these function are Lipschitz (hence in $\W^{1,2}(\X)$), satisfy $\Delta g_\epsilon = -p_\epsilon(x,\cdot) \leq 0$ and have a minimum since $\X$ is compact. Hence by the maximum principle they are all constant, so that also $G_x$ is constant, a contradiction.
    \par 
    So we can suppose that $\X$ has infinite diameter.
    Suppose by contradiction that $\X$ is non parabolic and assume first that the essential dimension of $\X$ is not $1$.
    Pick an exhaustion $\{\Omega_i\}_{i \in \bb{N}}$ of regular sets of $\X$ containing $x$, and for each of these sets consider the corresponding Green's function $G^i_x$. 
    By Lemma \ref{L15} there exist $T>1$
    and a sequence $\{c_i\}_{i \in \bb{N}} \in (1,T)$  such that
    \[
    x \in \{G^i_x > c_i\} \cap B_1(x) \subset \subset B_1(x) \quad \text{ for every } i \in \bb{N}
    \]
    and each of the sets $\{G^i_x > c_i\}$ is open, has finite perimeter and satisfies 
    \[
    \partial \{G^i_x > c_i\}=\partial \{G^i_x < c_i\}=\{G^i_x = c_i\}.
    \]
    We can suppose without loss of generality that $\bar{B}_1(x) \subset \Omega_i$ for every $i \in \bb{N}$.
    \par
    We then define
    \[
     l_i:=\sup_{z \in B_1(x)} \{\sd(x,z): G^i_x(z)> c_i\} 
    \]
    and claim that we have
    \[
    T \geq \int_{l_i}^{\sd(\partial \Omega_i,x)} \frac{1}{P( B_r(x),\X)} \, dr.
    \]
    To prove the claim let $0 <\epsilon <<1$ such that
    \[
    \partial \{G^i_x > \epsilon\}=\partial \{G^i_x < \epsilon\}=\{G^i_x = \epsilon\}
    \]
    and such that, setting
    \[
    s_i^\epsilon:=\inf_{z \in \Omega_i} \{\sd(x,z): G^i_x(z)<\epsilon \},
    \]
    it holds $s_i^\epsilon >l_i$. The first condition can be achieved by the same argument of Lemma \ref{L15}, while the second can be obtained using that $G_x^i>0$ on $\Omega_i$ and $G_x^i=0$ on $\partial \Omega_i$. \par 
    Consider the function $g_i:\X \to \bb{R}$ such that $g(y)$  is equal to $c_i$ in $\bar{B}_{l_i}(x)$,
    is equal to $\epsilon$ in ${^c}B_{s_i^\epsilon}(x)$,
    and in $B_{s_i^\epsilon}(x) \setminus B_{l_i}(x)$ is defined by
    \[
    g_i(y):=
    (c_i-\epsilon) \frac{\int_{\sd(x,y)}^{s_i^\epsilon} {P( B_r(x),\X)}^{-1} \, dr}
        {\int_{l_i}^{s_i^\epsilon} {P( B_r(x),\X)}^{-1} \, dr}+ \epsilon.
    \]
    Let 
    \[
    A_i^\epsilon:=\{y \in B_1(x):  G_x^i(y) < c_i \} \cup \{y \in \Omega_i \cap {^c}B_1(0) : G_x^i(y) > \epsilon \}.
    \]
    Thanks to our conditions on $\{G_x^i>c_i\}$ and $\{G_x^i> \epsilon\}$ we have that the set $A_i^\epsilon$ is open, contains $B_{s_i^\epsilon}(x) \setminus \bar{B}_{l_i}(x)$ and its boundary is made of disjoint connected sets where $G_x^i$ takes either the value $c_i$ or $\epsilon$.
    As a consequence, modulo translations and inversions, we can use $G_x^i$ as a barrier in Proposition \ref{P29} to show that $A_i^\epsilon$ is a regular set.\par 
    Hence there exists a 
    unique function $h \in D(\Delta,A_i^\epsilon) \cap C(\bar{A}_i^\epsilon)$ such that $h-g_i \in \W^{1,2}_0(A_i^\epsilon)$, $\Delta h=0$ and $h-g_i=0$ on $ \partial A_i^\epsilon$.  From this it follows that $h=G_x^i$, which in particular gives $G_x^i-g_i \in \W^{1,2}_0(A_i^\epsilon)$.
    Hence, using the fact that harmonic functions minimize the Dirichlet energy, we get
    \[
    \int_{A_i^\epsilon}|\nabla G_x^i|^2 \, d\m
    \leq \int_{A_i^\epsilon}|\nabla g_i|^2 \, d\m.
    \]
    At the same time, using the coarea formula, one can check that
    \[
    \int_{A_i^\epsilon}|\nabla g_i|^2 \, d\m = (c_i-\epsilon)^2 \Big(\int_{l_i}^{s_i^\epsilon} {P( B_r(x),\X)}^{-1} \, dr\Big)^{-1},
    \]
    so that
    \begin{equation} \label{E24}
    \int_{A_i^\epsilon}|\nabla G_x^i|^2 \, d\m
    \leq
    (c_i-\epsilon)^2 \Big(\int_{l_i}^{s_i^\epsilon} {P( B_r(x),\X)}^{-1} \, dr \Big)^{-1}.
    \end{equation}
    Moreover by construction we have that the vector field $\nabla G_x^i$ is defined in a neighbourhood of $\bar{A}_i^\epsilon$, where it is bounded and admits bounded divergence, so that we can apply Proposition \ref{P30} to get
    \[
    \int_{A_i^\epsilon}|\nabla G_x^i|^2 \, d\m
    =-c_i \int_{\partial \{G_i^\epsilon < c_i \} \cap B_1(x)} (\nabla G_x^\epsilon \cdot  \nu_{A_i^\epsilon})_{ext} \, d \Per
    -\epsilon \int_{\partial \{G_i^\epsilon > \epsilon \}} (\nabla G_x^\epsilon \cdot  \nu_{A_i^\epsilon})_{ext} \, d \Per.
    \]
    Using the definitions it is then easy to check that
    \[
    \int_{\partial \{G_i^\epsilon < c_i \} \cap B_1(x)} (\nabla G_x^\epsilon \cdot  \nu_{A_i^\epsilon})_{ext} \, d \Per=
    -\int_{\partial \{G_i^\epsilon > c_i \} \cap B_1(x)} (\nabla G_x^\epsilon \cdot  \nu_{({^c}A_i^\epsilon)})_{int} \, d \Per,
    \]
    so that, taking into account that $\Delta G_i=-\delta_x$ and using again Proposition \ref{P30}, we obtain
    \[
    \int_{A_i^\epsilon}|\nabla G_x^i|^2 \, d\m
    =c_i \int_{\partial \{G_i^\epsilon > c_i \} \cap B_1(x)} (\nabla G_x^\epsilon \cdot  \nu_{({^c}A_i^\epsilon)})_{int} \, d \Per
    -\epsilon \int_{\partial \{G_i^\epsilon > \epsilon \}} (\nabla G_x^\epsilon \cdot  \nu_{A_i^\epsilon})_{ext} \, d \Per
    =c_i-\epsilon.
    \]
    Putting together this with \eqref{E24} we obtain
    \[
    c_i-\epsilon \geq \int_{l_i}^{s_i^\epsilon} {P( B_r(x),\X)}^{-1} \, dr ,
    \]
    and letting $\epsilon \to 0$ we deduce that
    \[
    T \geq \int_{l_i}^{\sd(\partial \Omega_i,x)} \frac{1}{P( B_r(x),\X)} \, dr
    \]
    for every $i \in \bb{N}$, as claimed. \par 
    Taking into account that $l_i \leq 1$
    and letting $i \to + \infty$ we obtain that
    \[
    T \geq \int_{1}^{+ \infty} \frac{1}{P( B_r(x),\X)} \, dr,
    \]
    a contradiction. \par 
    We now need to consider the case when $\X$ has essential dimension $1$. A space with infinite diameter and essential dimension $1$ must be a weighted version of either $\bb{R}$ or of a closed half line by Proposition \ref{P33}.
    In this case we can repeat an analogous proof, replacing the sets $\{G_x^i>c_i\}$ with a fixed interval containing $x$. This is sufficient because on each endpoint of this interval each $G_x^i$ will trivially have a singleton as image, allowing to construct the analogue of the function $g_i$ in the previous case. Similarly the computation to estimate $\int_{A_i^\epsilon}|\nabla G_x^i|^2 \, d\m$ can be carried out in the same way, replacing $A_i^\epsilon$ with an appropriate union of two intervals and exploiting their regularity.
    \end{proof}
\end{proposition}

\begin{corollary} \label{C11}
    Let $(\X,\sd,\m)$ be an $\RCD(K,N)$ space. If
    there exists $x \in \X$ such that 
    \[
    \int_1^{+ \infty} \frac{t}{\m(B_t(x))} \, dt=+ \infty,
    \]
    then $\X$ is parabolic.
    \begin{proof}
It is clear that if the condition of the statement holds for $x \in \X$, then, thanks to Proposition \ref{P31}, it holds also for a regular point $x' \in \X$ satisfying
    \[
    \lim_{r \to 0^+}\frac{\m(B_r(x'))}{r^k} \in (0,+\infty),
    \]
    where $k$ is the essential dimension of $\X$.
    The statement then follows from the previous proposition repeating the strategy of \cite[Corollary $2$]{GP}.
    \end{proof}
\end{corollary}

Putting this together with the heat kernel estimates of Proposition \ref{Hestim} we obtain Theorem \ref{CT5}.

\begin{customthm}{5}
    Let $(\X,\sd,\m)$ be an $\RCD(K,N)$ space. If
    there exists $x \in \X$ such that 
    \[
    \int_1^{+ \infty} \frac{t}{\m(B_t(x))} \, dt=+ \infty,
    \]
    then $\X$ is parabolic. This condition is also necessary if $K=0$.
    \begin{proof}
        Thanks to Corollary \ref{C11} we only need to prove that if $K=0$ the condition of the statement is also necessary.
        This follows immediately since by Proposition \ref{Hestim} we have that
        \[
         \int_1^{+ \infty} p_t(x,x)\, dt \leq \int_1^{+ \infty} \frac{c_1}{\m(B_{\sqrt{t}}(x))} \, dt,
        \]
        concluding the proof
    \end{proof}
\end{customthm}

{
\begin{definition}[Spaces with slow volume growth] \label{Dslow}
An $\RCD(K,N)$ space $(\X,\sd,\m)$ is said to have \emph{slow volume growth} if there exists $x \in \X$ such that
\[
    \int_1^{+ \infty} \frac{t}{\m(B_t(x))} \, dt=+ \infty.
    \]
\end{definition}
}

We conclude the section
with a proposition taken from \cite[Example $3.19$]{Heatker} concerning spaces with finite measure.

\begin{proposition}
    Let $(\ssf{M},\sd_g,e^{-V}\m_g)$ be a geodesically complete weighted manifold with finite measure. Then for every $(x,y) \in \ssf{M} \times \ssf{M}$, setting $\m:=e^{-V}\m_g$, we have 
    \[
    p_t(x,y) \to \frac{1}{\m(\ssf{M})} \text{ as } t \to + \infty.
    \]
\end{proposition}

\begin{corollary} \label{C2}
    If $(\X,\sd,\m)$ is parabolic and $(\ssf{M},\sd_g,e^{-V}\m_g)$ is a geodesically complete weighted Riemannian manifold with finite mass then the product space is parabolic.
    \begin{proof}
        It follows immediately from the previous proposition, the definition of parabolicity, and the tensorization properties of heat kernels.
    \end{proof}
\end{corollary}

\section{Proof of Theorem \ref{CT1}} \label{S3}
In this section $(\X,\sd,\m)$ will be a fixed parabolic $\RCD(0,N)$ space and $C \subset \X \times \bb{R}$ will be locally the boundary of a perimeter minimizing set. We assume moreover that $C \subset \X \times \bb{R}_+$. We will denote by $\sd_C:{^c}C \to \bb{R}$ the distance from $C$, by $\sd_0$ the signed distance from $\X \times \{0\}$ with positive sign in the negative half space and negative sign otherwise, and by $\bar{\sd}:\X \times (-\infty,0) \to \bb{R}$ the difference $\sd_C-\sd_0$. Observe that $\bar{\sd} \geq 0$ since $C$ is contained in the upper half space.
\par 
Generic points in $\X \times \bb{R}$ will be usually denoted by $\bar{x}=(x,t_x)$, the canonical distance and the canonical measure in the product space will be $\sd_\times$ and $\m_\times$ respectively. \par 
We now develop the simple machinery needed to prove Propositions \ref{P14} and \ref{P2}, which are crucial steps in the previously outlined argument to prove Theorem \ref{CT1}.

\begin{definition}[Connectors]
    Given $\bar{x} \in \X \times \bb{R}$ we say that a geodesic $\gamma:[0,\sd_C(\bar{x})] \to \X \times \bb{R}$ is a \emph{connector} if it realizes the distance between $\bar{x}$ and $C$. We say that such connector is \emph{vertical} if its image is contained in a set of the form $\{x\} \times \bb{R} \subset \X \times \bb{R}$.
\end{definition}

Connectors will be useful because their behavior is related to the gradient of the function $\bar{\sd}$.
The next lemma is taken from \cite[Lemma $2.6$]{MP}.
\begin{lemma} \label{L2}
    For $\m_\times$-a.e. $\bar{x} \in \X \times (-\infty,0)$ there exists a unique connector.
\end{lemma}

\begin{proposition}[Laplacian bounds for $\sd_C$] \label{P14}
    $\sd_C:{^c}C \to \bb{R}$ admits negative distributional Laplacian.
    \begin{proof}
        From \cite[Lemma $4.16$]{CM18} we get that $\sd_C$ admits distributional Laplacian in its domain and that the singular part of such distributional Laplacian is a negative Radon measure. 
        Hence we only need to show that the regular part of this measure is negative as well.
        To this aim let $\bar{x} \notin C$ be a point with a single 
        connector, and let $\bar{y} \in C$ be such that $\sd_C(\bar{x})=\sd(\bar{x},\bar{y})$. Let $U_{\bar{y}} \ni \bar{y}$ and $E \subset U_{\bar{y}}$ be the sets as in Definition \ref{D3}.
        We claim that there exists a neighbourhood $V_{\bar{x}}$ of $\bar{x}$ such that
        \[
        V_{\bar{x}} \subset \{\bar{z} \in {^c}C: \exists \bar{z}_1 \in U_{\bar{y}}: \sd_C(\bar{z})=\sd(\bar{z},\bar{z}_1) \}.
        \]
        If this were not the case we could find a sequence $\bar{z}_n \to \bar{x}$ whose connectors' endpoints lie outside of $U_{\bar{y}}$. Modulo passing to a subsequence we then have that the endpoints of these connectors converge to a point $\bar{z} \in {^c}U_{\bar{y}}$, and the geodesic connecting $\bar{x}$ and $\bar{z}$ will be a second connector for $\bar{x}$, a contradiction. \par 
        Hence we can apply Proposition \ref{P35} to obtain that the distributional Laplacian of $\sd_C$ is negative in $V_{\bar{x}}$. Repeating the same argument for every point with a single connector and taking the union of the resulting open sets we get an open set of full measure where the distributional Laplacian of $\sd_C$ is negative. This implies that the regular part of the distributional Laplacian of $\sd_C$  is negative on its domain, concluding the proof.
    \end{proof}
\end{proposition}

The next Lemma shows the first link between connectors and the quantity $|\nabla \bar{\sd}|$.
\begin{lemma} \label{L1}
    For $\m_\times$-a.e. $\bar{x} \in \X \times (-\infty,0)$ with a vertical connector we have $|\nabla \bar{\sd}|(\bar{x})=0$. 
    \begin{proof}
        For $\m_\times$-a.e. point in $\bar{x} \in \X \times (-\infty,0)$ we have 
        \[
        |\nabla \bar{\sd}|(\bar{x})^2=2-2\nabla \sd_C \cdot \nabla \sd_0(\bar{x})
        \]
        so that we only need to prove that for $\m_\times$-a.e. point $\bar{x}$ with a unique vertical connector we have $\nabla \sd_C \cdot \nabla \sd_0(\bar{x})=1$. Since distance functions have local Lipschitz constant $1$ everywhere, for $\m_\times$-a.e. point $\bar{x}$ we get 
        \[
        \nabla \sd_C \cdot \nabla \sd_0(\bar{x}) \leq |\nabla \sd_C| (\bar{x})|\nabla \sd_0|(\bar{x})=1.
        \]
        Moreover one can easily check that for $\m_\times$-a.e. $\bar{x}$ we have
        \[
        \nabla \sd_C \cdot \nabla \sd_0(\bar{x})=
        \lim_{\epsilon \to 0}\frac{\lip(\sd_C+\epsilon \sd_0)^2(\bar{x})-\lip(\sd_C)^2(\bar{x})}{2\epsilon}.
        \]
        Then, moving along the vertical connector, we deduce that $\lip(\sd_C+\epsilon \sd_0) \geq 1+\epsilon$, so that for $\m_\times$-a.e. $\bar{x}$
        \[
        \nabla \sd_C \cdot \nabla \sd_0(\bar{x})\geq \lim_{\epsilon \to 0}\frac{(1+\epsilon)^2-1}{2\epsilon}=1,
        \]
        concluding the proof.
    \end{proof}
\end{lemma}

\begin{definition}[Slope of a connector]
    Let $\bar{x} =(x,t_x) \in \X \times (-\infty,0)$, let $\gamma:[0,\sd_C(\bar{x})] \to \X \times \bb{R}$ be a connector and let $\bar{y}=(y,t_y):=\gamma(\sd_C(\bar{x}))$. We define the slope $\ssf{s}(\bar{x},\gamma)>0$ of $\gamma$ by
    \[
    \ssf{s}(\bar{x},\gamma):=1-\frac{|t_x-t_y|}{\sd_C(\bar{x})}.
    \]
\end{definition}

The role played by the slope of a connector is clarified in the next lemma.

\begin{lemma} \label{L3}
Let $\bar{x} \in \X \times (-\infty,0)$ and $\gamma$ be a connector for $\bar{x}$. Then
    \[
    \lip(\bar{\sd})(\bar{x}) \geq \ssf{s}(\bar{x},\gamma).
    \]
    \begin{proof}
        The projection of $\gamma$ to $\bb{R}$ is a geodesic of constant speed $1-\ssf{s}(\bar{x},\gamma)$ so that
        \[
        \sd_{0}(\gamma(s))=
        \sd_{0}(\bar{x})-(1-\ssf{s}(\bar{x},\gamma))s,
        \]
        while, since $\gamma$ is a connector, we have
        \[
        \sd_C(\gamma(s))=\sd_C(\bar{x})-s.
        \]
        Using these facts we obtain
        \[
        \bar{\sd}(\gamma(s))-\bar{\sd}(\bar{x})=
        \sd_C(\gamma(s))-\sd_0(\gamma(s))-\sd_C(\bar{x})+\sd_0(\bar{x})
        =-s+(1-\ssf{s}(\bar{x},\gamma))s=-s \ssf{s}(\bar{x},\gamma).
        \]
        Dividing both sides of the previous equation by $s$ and taking the limit as $s$ goes to zero we immediately obtain that $\lip(\bar{\sd})(\bar{x}) \geq \ssf{s}(\bar{x},\gamma)$.
    \end{proof}
\end{lemma}

The next lemma shows that the slope is continuous in points with a single connector. It is easy to check that the statement may fail for general points.
\begin{lemma} \label{L4}
Let $\bar{x} \in \X \times (-\infty,0)$ be a point admitting a unique connector $\gamma$.
    Let $\{\bar{x}_n\}_{n \in \bb{N}} \subset \X \times (-\infty,0)$ be a sequence such that $\bar{x}_n \to \bar{x}$, and $\{\gamma_n\}_{n \in \bb{N}}$ a corresponding sequence of connectors. Then 
    \[
    \lim_{n \to + \infty} \ssf{s}(\bar{x}_n,\gamma_n) = \ssf{s}(\bar{x},\gamma).
    \]
    \begin{proof} 
        For every $n \in \bb{N}$ we define 
        \[
        \bar{y}_n:=
        \gamma_n(\sd_C(\bar{x}_n)) \in C
        \]
        and we note that these points have uniformly bounded distance from $\bar{x}$ because of the triangular inequality. 
        We claim that the sequence $\bar{y}_n$ converges to the endpoint of the connector $\gamma$.
        If we are able to prove the claim, the statement of the lemma immediately follows by noting that the slope of a connector depends continuously on the positions of the starting point and the endpoint. \par 
        To prove the claim we observe that, modulo passing to a (non relabeled) subsequence, we have that $\bar{y}_n \to \bar{y} \in C$.
        Moreover, since $\sd_C(\bar{x}_n)=\sd_\times(\bar{x}_n,\bar{y}_n)$, passing to the limit and using the continuity of distance functions we obtain that
        \[
        \sd_C(\bar{x})= \sd_\times(\bar{x},\bar{y}).
        \]
        This implies, by the uniqueness of the connector in $\bar{x}$,  that $\bar{y}$ is the endpoint of $\gamma$. In particular the limit point does not depend on the subsequence that we chose, so that the full sequence $\{\bar{y}_n\}_{n \in \bb{N}}$ converges to $\bar{y}$, as claimed.
    \end{proof}
\end{lemma}

\begin{proposition} \label{P2}
    If $\bar{\sd}:\X \times (-\infty,0) \to \bb{R}$ is not constant, then there exist a ball $B \subset \X \times (- \infty,0)$ and a real number $\tau >0$ such that $|\nabla \bar{\sd}(\bar{x})| \geq \tau$ for $\m_\times$-a.e. $\bar{x} \in B$.
    \begin{proof}
        Since $\bar{\sd}$ is not constant there exists a set of positive measure in $\X \times (-\infty,0)$ where $|\nabla \bar{\sd}|>0$.
        By Lemmas \ref{L2} and \ref{L1}, this implies that there exists a point $\bar{x} \in \X \times (-\infty,0)$ with a single non vertical connector $\gamma$ such that 
        $\ssf{s}(\bar{x},\gamma) >0$. By lemma \ref{L4} and a standard argument by contradiction, this implies that there exists $\tau>0$ such that for every $\bar{y}$ in a sufficiently small ball centered in $\bar{x}$, and every connector $\gamma_{\bar{y}}$ of $\bar{y}$, we have
        $
        p(\bar{y},\gamma_{\bar{y}})> \tau.
        $
        The conclusion immediately follows by Lemma \ref{L3}.
    \end{proof}
\end{proposition}

\begin{remark} \label{R4}
    A direct computation shows that
 $
 \Delta e^{-\bar{\sd}} \geq e^{-\bar{\sd}} |\nabla \bar{\sd}|^2.
 $
 Combining this with the previous proposition we deduce that there exists a ball $B \subset \X \times (-\infty,0)$ and $c>0$ such that
 \[
 \Delta e^{-\bar{\sd}} \geq c1_B.
 \]
\end{remark}

The next proposition follows from a standard computation and gives the explicit dependence of the Laplacian on the weight in a weighted $\RCD$ space.

 \begin{proposition}[Formulas for Laplacians in weighted spaces] \label{P3}
 Let $(\ssf{Y},\sd_y,\m_y)$ be an $\RCD(K,N)$ space and let $V \in \W^{1,2}_{loc}(\X)$. Suppose that the space $(\ssf{Y},\sd_y,\tilde{\m}_y)$ with $\tilde{\m}=e^{-V}\m_y$ is $\RCD(K',N')$ for a suitable couple $(K',N') \in \bb{R} \times [1,+\infty)$.
 Then for every $u \in \Lip_{loc}(\ssf{Y})$, denoting by $\tilde{\Delta}$ the Laplacian in $(\ssf{Y},\sd_y,\tilde{\m}_y)$, we have that $
 \tilde{\Delta}u=\Delta u-\nabla u \cdot \nabla V $ in distributional sense.
 \end{proposition}

 The next step to prove Theorem \ref{CT1} is to replace the product measure $\m_\times$ on $\X \times \bb{R}$ with a measure $\tilde{\m}_\times$, such that the resulting space is parabolic and the Laplacian bounds given by Remark \ref{R4} are preserved (modulo reducing $c$) w.r.t. the new measure. 
To this aim consider the modified metric measure space $( \bb{R},\sd_e,e^{-V} \lambda^1)$ where $\sd_e$ is the Euclidean distance and the function $V:\bb{R} \to \bb{R}_+$ is defined as follows. \par 
    Let $\bar{x}=(p_x,t_x) \in \X \times (-\infty,0)$ and $\tau,\epsilon >0$ be such that $|\nabla \bar{d}| \geq \tau$ on $B_\epsilon(\bar{x}) \subset \X \times (-\infty,0)$. 
    Note that this is possible thanks to Proposition \ref{P2}. Let $f:\bb{R} \to \bb{R}$ be an odd increasing function such that $f(0)=0$, $f(t)=1$ for every $t \geq 1$ and $f(t)=-1$ for every $t \leq -1$. Define
    \[
    V(s):=\int_{t_x}^sf(t-t_x) \, dt.
    \]
    Note that $V$ is symmetric w.r.t. $t_x$ and that $V'(s)=1$ if $s \in [t_x+1,+\infty)$, $V'(s) \in (-1,1)$ if $s \in (t_x-1,t_x+1)$ and $V'(s)=-1$ if $s \in (-\infty,t_x-1]$. In particular $\int_\bb{R} e^{-V(s)} \, ds < + \infty$. \par 
    As anticipated, we use such function to modify the measure in the product space $\X \times \bb{R}$. 

\begin{proposition}[Mild regularity for the weighted $\X \times \bb{R}$]
There exists a couple $(K,N)$ such that
    \[
    (\X \times \bb{R},\sd_\times,e^{-V(-\sd_0)}\m_\times),
    \]
    is an $\RCD(K,N)$ space.
    \begin{proof}
    Note that $( \bb{R},\sd_e,e^{-V(-\sd_0)} \lambda^1)$ is an $\RCD(K_0,N_0)$ space for some $K_0 \in (-\infty,0)$ and $N_0\in (1,+\infty)$. Since
    \[
    (\X \times \bb{R},\sd_\times,e^{-V(-\sd_0)}\m_\times)=
    (\X,\sd,\m) \otimes ( \bb{R},\sd_e,e^{-V} \lambda^1),
    \]
    we deduce that also $(\X \times \bb{R},\sd_\times,e^{-V(-d_0)}\m_\times)$ is an
    $\RCD(K_0,N+N_0)$ space, as desired.
    \end{proof}
\end{proposition}

\begin{proposition}[Parabolicity of the weighted $\X \times \bb{R}$] \label{P4}
    The space
    \[
    (\X \times \bb{R},\sd_\times,e^{-V(-\sd_0)}\m_\times)
    \]
    is parabolic.
    \begin{proof}
        As we noted earlier the factor $( \bb{R},\sd_e,e^{-V} \lambda^1)$ has finite mass. The statement follows applying Corollary \ref{C2}.
    \end{proof}
\end{proposition}

In what follows we denote by $\tilde{\m}_\times$ the measure $e^{-V(-\sd_0)}\m_\times$ and by $\tilde{\Delta}$ the laplace operator in $(\X \times \bb{R},\sd_\times,\tilde{\m}_\times)$.
We recall that both in $(\X \times \bb{R},\sd_\times,\tilde{\m}_\times)$ and  $(\X \times \bb{R},\sd_\times,\m_\times)$
the relaxed gradient of a Lipschitz function coincides with the local Lipschitz constant, so that in particular it does not change depending on the space that we consider. The next theorem contains the crucial computation showing that the Laplacian bounds of Proposition \ref{P2} are preserved in the modified product space.

\begin{thm}[Weighted Laplacian bounds for $e^{-\bar{\sd}}$] \label{T2}
    The function $e^{-\bar{\sd}}$ satisfies $\tilde{\Delta}e^{-\bar{\sd}} \geq 0$ in $\X \times (-\infty,0)$. Moreover there exist $\tau>0$ and $B_\epsilon(\bar{x}) \subset \X \times (-\infty,0)$ such that $\tilde{\Delta} e^{-\bar{\sd}} \geq \tau >0$ on $B_\epsilon(\bar{x})$.
    \begin{proof}
        Using the chain rule for Laplacians and taking into account Proposition \ref{P3} we have that 
        \begin{equation} \label{E4}
         \tilde{\Delta} e^{-\bar{\sd}}
        =e^{-\bar{\sd}}(-\tilde{\Delta} \bar{\sd}+|\nabla \bar{\sd}|^2)
        =
        e^{-\bar{\sd}}(-\Delta \bar{\sd}+\nabla \bar{\sd} \cdot \nabla V(-\sd_0)+|\nabla \bar{\sd}|^2)
        =
        e^{-\bar{\sd}}(-\Delta \bar{\sd}- V'(-\sd_0) \nabla \bar{\sd} \cdot \nabla \sd_0+|\nabla \bar{\sd}|^2).
        \end{equation}
        By the fact that $\Delta \sd_0=0$ and Proposition \ref{P14} we have that 
        $
        -\Delta \bar{\sd}=-\Delta \sd_C \geq 0.
        $
        Moreover 
        \[
        |\nabla \bar{\sd}|^2=2-2\nabla \sd_C \cdot \nabla \sd_0 \quad \text{and} \quad
        \nabla \bar{\sd} \cdot \nabla \sd_0=\nabla \sd_C \cdot \nabla \sd_0-1,
        \]
        so that \eqref{E4} is bounded below by
        \begin{align*}
        &e^{-\bar{\sd}} ( V'(-\sd_0)-V'(-\sd_0) \nabla \sd_C \cdot \nabla \sd_0+2-2\nabla \sd_C \cdot \nabla \sd_0)
        \\
        &=e^{-\bar{\sd}}(2+V'(-\sd_0))(1-\nabla \sd_C \cdot \nabla \sd_0)=
        e^{-\bar{\sd}}(2+V'(-\sd_0))\frac{|\nabla \bar{\sd}|^2}{2}.
        \end{align*}
        Summing up we obtain
        \begin{equation} \label{E5}
        \tilde{\Delta} e^{-\bar{\sd}} \geq
        e^{-\bar{\sd}}(2+V'(-\sd_0))\frac{|\nabla \bar{\sd}|^2}{2}.
        \end{equation}
        Since by construction $|V'(-\sd_0)| \leq 1$ the first part of the statement follows. For the second part it is sufficient to apply Proposition \ref{P2} to equation \eqref{E5}. 
    \end{proof}
\end{thm}

The next proposition concerns regularity of sets in the modified product space and is an immediate consequence of Corollary \ref{C5}.

\begin{proposition} \label{P15}
    Let $B \subset \X$ be a regular set in $(\X,\sd,\m)$ and let $I_t:=(t_x-t,t_x+t) \subset (\bb{R},\sd_e,e^{-V}\lambda^1)$. Then $B \times I_t$ is regular in $(\X \times \bb{R},\sd_\times,\tilde{\m}_\times)$.
\end{proposition}

The next two lemmas concern some simple geometric properties of $\bar{\sd}$.

\begin{lemma} \label{L17}
    Let $y \in \X$ and $t \in (-\infty,0)$. Then $\bar{\sd}(y,0) \geq \bar{\sd}(y,t)$.
    \begin{proof}
        By triangular inequality we have
        $
        \sd_C(y,t) \leq \sd_C(y,0)+t,
        $
        so that
        \[
        \bar{\sd}(y,t)=\sd_C(y,t) -t \leq \sd_C(y,0)=\bar{\sd}(y,0),
        \]
        as claimed.
    \end{proof}
\end{lemma}

The next lemma follows easily by the definition of product distance.

\begin{lemma} \label{L13}
    Suppose that $C$ and $\X \times \{0\}$ have positive distance and that there exists a geodesic $\gamma:[0,\sd_\times(C,\X \times \{0\})] \to \X \times \bb{R}$ realizing such distance. Then $\gamma$ is vertical, i.e.\;its image is contained in a set of the form $\{x\} \times \bb{R}$.
\end{lemma}

The next theorem is the key step to prove the Half Space Property. In the proof we follow the argument outlined in the Introduction.

\begin{thm}[Key step: constancy of $\bar{\sd}$] \label{T14}
    The function $\bar{\sd}: \X \times (-\infty,0) \to \bb{R}$ is constant.
    \begin{proof}
        Suppose that the statement does not hold. \par 
        Assume first that $\X$ has finite diameter. In this case $\X$ is compact, so that the restriction $\sd_C:\X \times \{-1\} \to \bb{R}$ has a minimum on a point $x \in \X \times \{-1\}$.
        Let $\gamma$ be the connector of $x$ to $C$ and note that it has to be vertical by Lemma \ref{L13}. 
        In particular $\bar{\sd}$ attains its minimum on all the points in its domain that are in the image of $\gamma$. By Proposition \ref{P14} and the maximum principle this implies that such function is constant.
        \par 
        Assume now that $\X$ has infinite diameter.
        Then by Theorem \ref{T2} there exist $\bar{x}=(x,t_x)$ and $\epsilon,\tau>0$ such that $\tilde{\Delta}e^{-\bar{\sd}} \geq \tau>0$ on $B_\epsilon(\bar{x})$. Modulo replacing $x$ with a sufficiently close point we may also suppose that
        $x$ is regular for $\X$ and that
    \[
    \lim_{r \to 0^+}\frac{\m(B_r(x))}{r^k} \in (0,+\infty),
    \]
    where $k \in \bb{N}$ is the essential dimension of $\X$. \par
        Let $\phi \in \Lip_c(B_\epsilon(\bar{x}))$ be a positive function, which is equal to $\tau$ on $B_{\epsilon/2}(\bar{x})$, takes values in $[0,\tau]$, and is symmetric with respect to $\X \times \{t_x\}$. It is clear that such a function exists. Let also $\{B_i\}_{i \in \bb{N}}$ be an exhaustion of $(\X,\sd,\m)$ with regular sets containing $x$.
        \par
        Consider now for every $i \in \bb{N}$ and $\bar{z} \in B_i \times [t_x-i,t_x+i]$ the Green's function $G_{\bar{z}}^i$ with pole $\bar{z}$ relative to the set 
        \[
        B_i \times [t_x-i,t_x+i] \subset (\X \times \bb{R},\sd_\times,\tilde{\m}_\times),
        \]
        and note that such function exists thanks to Propositions \ref{P15} and \ref{P1}.
        \par 
        In addition, since the space 
        $(\X \times \bb{R},\sd_\times,\tilde{\m}_\times)$ is parabolic by Proposition \ref{P4}, then for every $\bar{y} \neq \bar{x}$ (and $i$ sufficiently large) we get by Theorem \ref{CT4} that
\[
G_{\bar{x}}^i(\bar{y}) \to \infty \quad \text{as } i \uparrow + \infty.
\]
We claim that this implies that
for every $S \subset \subset \X \times \bb{R}$ we have
\begin{equation} \label{E6}
    \inf_{\bar{y} \in S} G_{\bar{x}}^i(\bar{y}) \to \infty \quad \text{as } i \uparrow + \infty.
\end{equation}
This follows because for every $\bar{y} \neq \bar{x}$ the sequence  $\{G_{\bar{x}}^i(\bar{y})\}_{i \in \bb{N}}$ increases as $i$ increases by Proposition \ref{P16} while, considering the $[0,+\infty]$-valued continuous representative given by Lemma \ref{L15}, we have that
$G_{\bar{x}}^i(\bar{x})=+\infty$ (here we are using our particular choice of the point $x \in \X$ and the fact that the essential dimension of the product space $\X \times \bb{R}$ is greater than or equal to $2$). 

Consider now for every $i \in \bb{N}$ the function $g_i:B_i \times [t_x-i,t_x+i] \to \bb{R}$ defined by
\[
g_i(\bar{z}):=\int_{\X \times \bb{R}}G_{\bar{z}}^i(\bar{y}) \phi(\bar{y}) \, d\tilde{\m}_\times (\bar{y}).
\]
By Proposition \ref{P20} each of these functions, in its domain of definition, satisfies $\tilde{\Delta}g_i=-\phi$, is locally Lipschitz and continuous up to the boundary, and takes value zero on the boundary itself.
\par
Moreover using \eqref{E6} and the fact that $\phi=\tau$ on $B_{\epsilon/2}(\bar{x})$ it is easy to check that 
\[
g_i(\bar{x}) \to + \infty \quad \text{as } i \to +\infty.
\]
In particular we can pick $i_0$ large enough so that $g_{i_0}(\bar{x})>2$ and $i_0 > |t_x|$. Consider now the function 
$
e^{-\bar{\sd}}+g_{i_0}.
$
This function must have a maximum in 
\[
B_{i_0} \times [t_x-i_0,t_x+i_0] \cap \X \times (-\infty,0]
\]
since it is continuous and $(e^{-\bar{\sd}}+g_{i_0})(\bar{x}) \geq 2 $ while $(e^{-\bar{\sd}}+g_{i_0})(\bar{y}) \leq 1$ for every 
\begin{equation} \label{E11}
\bar{y} \in \partial (B_{i_0} \times [t_x-i_0,t_x+i_0] \cap \X \times (-\infty,0]) \setminus \X \times \{0\}.
\end{equation}
Moreover by Theorem \ref{T2} we have
\[
\tilde{\Delta} (e^{-\bar{\sd}}+g_{i_0}) \geq 0 \quad \text{in } B_{i_0} \times [t_x-i_0,t_x+i_0] \cap \X \times (-\infty,0),
\]
so that if we can prove that the maximum lies in 
\[
B_{i_0} \times [t_x-i_0,t_x+i_0] \cap \X \times (-\infty,0),
\]
then by the maximum principle we would obtain that $e^{-\bar{\sd}}+g_{i_0}$ is constant, which contradicts the fact that $(e^{-\bar{\sd}}+g_{i_0})(\bar{x}) \geq 2 $ while $(e^{-\bar{\sd}}+g_{i_0})(\bar{y}) \leq 1$ on the set in \eqref{E11}. \par 
To prove that a maximum lies in the desired set we suppose by contradiction that all the maximum points lie in $\X \times \{0\}$. Let $(y_0,0)$ be such maximum point, and observe that by Lemma \ref{L17} we have 
\[
e^{-\bar{\sd}}(y_0,2t_x) \geq e^{-\bar{\sd}}(y_0,0),
\]
so that if we can prove that
\begin{equation} \label{E12}
g_{i_0}(y_0,2t_x)=g_{i_0}(y_0,0)
\end{equation}
we would obtain that also $(y_0,2t_x)$ is a maximum point, a contradiction. \par 
To this aim observe that the function $g_{i_0}$ is the only function in
\begin{equation} \label{E13}
\W^{1,2}_0(B_{i_0} \times (t_x-i_0,t_x+i_0)) \cap C(\bar{B}_{i_0} \times [t_x-i_0,t_x+i_0])
\end{equation}
taking value zero on the boundary and such that $\tilde{\Delta} g_{i_0}=-\phi$. On the other hand, given that $\phi$ and $\tilde{\m}_\times$ are symmetric w.r.t. $\X \times \{t_x\}$, also the reflected function $g'_{i_0}(a,t_a):=g_{i_0}(a,2t_x-t_a)$ belongs to the set in \eqref{E13}, is zero on the boundary, and satisfies $\tilde{\Delta} g'_{i_0}=-\phi$. In particular $g'_{i_0}=g_{i_0}$, implying \eqref{E12}.
    \end{proof}
\end{thm}

\begin{corollary} \label{C9}
    There exists $a \in (0,+\infty)$ such that $\X \times \{a\} \subset C$ and 
    \[
    C \cap \X \times (-\infty,a)=\emptyset.
    \]
    \begin{proof}
        By the previous proposition the function $\sd_C:\X \to \bb{R}$ is constant. Let $a>0$ be its constant value. It is clear that $C \cap \X \times (-\infty,a)=\emptyset$.
        To show that $\X \times \{a\} \subset C$ suppose by contradiction that there exists $x \in \X$ such that $(x,a) \notin C$. 
        Then there exists a ball $B_\epsilon(x,a) \subset \X \times \bb{R}$ such that $B_\epsilon(x,a) \cap C= \emptyset$ and this implies that $\sd_C(x)>a$, which is a contradiction.
    \end{proof}
\end{corollary}

The remaining part of the section is dedicated to showing that $C$ is actually a union of horizontal slices. The idea is that, using the previous corollary, we can prove that $C \setminus \X \times \{a\}$ is locally a boundary of a locally perimeter minimizing set. Then, reapplying Corollary \ref{C9} to the aforementioned set, we obtain that $C \setminus \X \times \{a\}$ contains a set of the form $\X \times \{b\}$ for some $b>a$ and $\X \times (-\infty,b) \cap C=\X \times \{a\}$. Finally, iterating this procedure, we obtain that $C$ is made of horizontal slices. \par 
To follow such a plan we first need a few intuitive results whose proofs are quite technical. These are presented in the next lemmas. We will repeatedly use the fact that $P(\X \times (-\infty,c),A \times \bb{R})=\m(A)$ for every $c \in \bb{R}$.

\begin{lemma} \label{LP28}
    Let $x \in \X$ be a regular point for $\X$ and let $U_{(x,a)}$ and $E \subset U_{(x,a)}$ be the sets given by Definition \ref{D3} applied to $C$ at $x$. Modulo passing to the complement of $E$ we can suppose that $E \subset \X \times (a,+\infty) \cap U_{(x,a)}$. Then there exists $\epsilon_0>0$ such that for every $\epsilon \in (0,\epsilon_0)$
    \[
    B_\epsilon(x) \times (a+\epsilon,a+2\epsilon) \subset E.
    \]
    \begin{proof}
    We first claim that there exists $\epsilon_0>0$ and an open set 
        \[
        B_{2\epsilon_0}(x) \times (a-2\epsilon_0,a+2\epsilon_0) \subset \X \times \bb{R}
        \]
        such that for every $(y,t) \in \X \times (a,+\infty) \cap \partial E \cap  B_{2\epsilon_0}(x) \times (a-2\epsilon_0,a+2\epsilon_0)$ we have
        \[
        \frac{|t-a|}{\sd(x,y)} < 1.
        \]
        To prove the claim we suppose by contradiction that there exists a sequence $\{(x_n,t_n)\} \subset \X \times (a,+\infty) \cap \partial E$ converging to $(x,a)$ and satisfying 
        \begin{equation} \label{E15}
        \frac{|t_n-a|}{\sd(x,x_n)} \geq 1 \quad \text{for every } n \in \bb{N}.
        \end{equation}
        Let $k$ be the essential dimension of $\X$. Let $(\ssf{Z},\sd_Z)$ be the space realizing the blow up of $\X$ at $x$ along the scales $\{|t_n-a|^{-1}\}_{n \in \bb{N}}$, and observe that $(\ssf{Z} \times \bb{R},\sd_e \times \sd_Z)$ 
        naturally realizes the blow up of $\X \times \bb{R}$ at $(x,a)$ along the same scales. When doing such blow up the set $E$ converges in $\sL^1_{loc}$ to a perimeter minimizing set $E' \subset \bb{R}^k \times \bb{R}$ such that
        \[
        \bb{R}^k \times \{0\} \subset \partial E', \quad E' \subset \bb{R}^k \times (0,+\infty).
        \]
        Moreover, because of \eqref{E15} and Proposition \ref{P26}, there exists $a'>0$ such that $(0,a') \in \partial E'$.
        This is a contradiction because of the global Bernstein Property on $\bb{R}^{k+1}$. This proves our initial claim. \par 
        Hence the set
        \[
        \Big\{ (y,t) \in B_{2\epsilon_0}(x) \times (a-2\epsilon_0,a+2\epsilon_0) \cap \X \times (a,+\infty): \frac{|t-a|}{\sd(x,y)} > 1 \Big\}
        \]
        is either contained in ${^c}E$ or in $E$. 
        If we manage to prove that it is contained in $E$ we are done. \par 
        So suppose by contradiction that the previous set is contained in ${^c}E$. Then repeating the same blow up argument that we used at the beginning of the proof we obtain a perimeter minimizing set $E' \subset \bb{R}^k \times \bb{R}$ such that
        \[
        \bb{R}^k \times \{0\} \subset \partial E', \quad E' \subset \bb{R}^k \times (0,+\infty)
        \]
        and
        \[
        \{(x,t) \in \bb{R}^k \times \bb{R}: |t| > |x| \} \subset {^c}E'.
        \]
        This is again a contradiction by the global Bernstein Property in Euclidean space.
    \end{proof}
\end{lemma}

\begin{lemma} \label{LP21}
    Let $x \in \X$ and let $U_{(x,a)}$ and $E \subset U_{(x,a)}$ be the sets given by Definition \ref{D3} applied to $C$ at $x$. We can assume w.l.o.g. that $U_{(x,a)}=A \times I$, where $A \subset \X$ is open and $I$ is an interval, and that $E$ lies above the slice $\X \times \{a\}$. Then
    for every $\delta>0$ there exist $0< \epsilon < \delta/2$ and an open set $A' \subset A$ such that
    \[
    \m(A \setminus A') < \delta, \quad \text{and } 
    A' \times (a+\epsilon,a+2\epsilon) \subset E.
    \]
    \begin{proof}
        For every regular $x \in A$ consider the neighbourhood $V_x \subset A $ and the number $\epsilon_x>0$ given by the previous lemma. Define then
        \[
        A_t:=\{x \in A: \epsilon_x >t \}
        \]
        and, choosing $\epsilon$ small enough, we have that $\m(A \setminus A_{\epsilon})< \delta$. Setting 
        \[
        A':=\bigcup _{x \in A_\epsilon} V_x
        \]
        we conclude.
    \end{proof}
\end{lemma}

In the next lemmas the symbols $U_{(x,a)}$, $A$, $I$ and $E$ will denote the objects introduced in the statement of the previous lemma.

\begin{lemma} \label{LP22}
    We have
    \[
    P(E,A \times \{a\}) \geq \m(A).
    \]
    \begin{proof}
    Fix $\delta >0$ and consider the open set $A' \subset A$ and the number $0 < \epsilon <\delta/2$ such that
    \[
    \m(A \setminus A') < \delta, \quad \text{and } 
    A' \times (a+\epsilon,a+2\epsilon) \subset E.
    \]
    We claim that
    \begin{equation} \label{E16}
    P(E,A' \times (a-\delta,a+\delta)) \geq \m(A').
    \end{equation}
    If the claim holds we then obtain that
    \[
    P(E,A \times (a-\delta,a+\delta)) \geq \m(A)-\delta,
    \]
    and passing to the limit on both sides as $\delta$ goes to zero we conclude. \par 
    To prove \eqref{E16} consider a sequence $\{f_i\}_{i \in \bb{N}} \in \Lip_{loc}(A' \times (a-\delta,a+\delta))$ such that 
    \[
    f_i \to 1_{E} \quad \text{in } \sL^1(A' \times (a-\delta,a+\delta))
    \]
    and
    \begin{equation} \label{E17}
    P(E,A' \times (a-\delta,a+\delta))= \lim_{i \to +\infty}
    \int_{A' \times (a-\delta,a+\delta)} \lip(f_i) \, d\m_\times.
    \end{equation}
     We claim that we can replace each $f_i$ with functions $g_i$ mantaining the same properties and such that $g_i=0$ on $A' \times \{a-3\epsilon/2\}$ and $g_i=1$ on $A' \times \{a+3\epsilon/2\}$. This is possible because \eqref{E17} implies, though a lower semi-continuity argument and the fact that $A' \times (a+\epsilon,a+2\epsilon) \subset E$ and $A' \times (a-2\epsilon,a-\epsilon) \subset {^c}E$, that
     \[
     \lim_{i \to +\infty}
    \int_{A' \times (a+5\epsilon/4,a+5\epsilon/3)} \lip(f_i) \, d\m_\times=
     \lim_{i \to +\infty}
    \int_{A' \times (a-5\epsilon/3,a-5\epsilon/4)} \lip(f_i) \, d\m_\times=0.
     \]
These identities allow, though a cut off argument, to construct $g_i \in \Lip_{loc}(A' \times (a-\delta,a+\delta))$ such that
\begin{align*}
    g_i \to 1_{E} & \quad \text{in } \sL^1(A' \times (a-\delta,a+\delta)),\\
    P(E,A' \times (a-\delta,a+\delta))&= \lim_{i \to +\infty}
    \int_{A' \times (a-\delta,a+\delta)} \lip(g_i) \, d\m_\times
    \end{align*}
    and
    \[
    g_i=0 \text{ on } A' \times \{-3\epsilon/2\}; \quad  g_i=1 \text{ on } A' \times \{3\epsilon/2\}.
    \]
    In particular, denoting by $\lip_t$ the local Lipschitz constant of a function restricted to the vertical component of the product space, we obtain
     \begin{align*}
    \int_{A' \times (a-\delta,a+\delta)} \lip(g_i) \, d\m_\times 
    \geq 
    \int_{A' \times (a-\delta,a+\delta)}
    \lip_t(g_i) \, d\m_\times =\int_{A'} \int_{(a-\delta,a+\delta)} \lip_t(g_i) \, ds \, d\m,
    \end{align*}
    and by the condition on the values of $g_i$ on the horizontal slices the  quantity $\int_{(a-\delta,a+\delta)} \lip_t(g_i) \, ds$ is greater than $1$ for every $x \in A'$, so that the last integral is greater than or equal to $\m(A')$. Summing up we obtained
    \[
    P(E,A' \times (a-\delta,a+\delta))= \lim_{i \to +\infty}
    \int_{A' \times (a-\delta,a+\delta)} \lip(g_i) \, d\m_\times \geq \m(A'),
    \]
    as desired.
    \end{proof}
\end{lemma}

\begin{lemma} \label{LP23}
    For every $\delta_1,\delta_2,\delta_3>0$ there exists $0< \epsilon < \delta_1$ such that
    \[
    P(E \cap \X \times (-\infty,a+\epsilon),A \times (a,a+\delta_2)) \geq \m(A) - \delta_3.
    \]
    \begin{proof}
        Let $\delta :=\min \{\delta_1,\delta_2,\delta_3\}$. 
        By Lemma \ref{LP21} there exist
        $0< \epsilon' < \delta/2$ and an open set $A' \subset A$ such that
    \[
    \m(A \setminus A') < \delta, \quad \text{and } 
    A' \times (a+\epsilon',a+2\epsilon') \subset E.
    \]
    Setting $\epsilon:=3\epsilon'/2$, and taking into account that  we have
    \[
    E \cap \X \times (-\infty,a+\epsilon) \cap A' \times (a+\epsilon',a+2\epsilon')
    =\X \times (-\infty,a+\epsilon) \cap A' \times (a+\epsilon',a+2\epsilon'),
    \]
    we get that
    \[
    P(E \cap \X \times (-\infty,a+\epsilon),A \times (a,a+\delta)) \geq \m(A') \geq \m(A) - \delta,
    \]
    which immediately implies the statement.
    \end{proof}
\end{lemma}

In what follows we will use the notation $E':=E \cup \X \times (-\infty,a)$.
The next lemma is the key step to show that $C \setminus \X \times \{a\}$ is locally a boundary of a locally perimeter minimizing set.

\begin{lemma} \label{L12}
    We have that
    \[
    P(E,A \times I)-P(\X \times (-\infty,a),A \times I) \geq
    P(E',A \times I).
    \]
    \begin{proof}
        By Lemma \ref{LP22} we have that
        \[
        P(E,A \times I)-P(\X \times (-\infty,a),A \times I) \geq 
        P(E,A \times (a,b))
        \]
        and this last quantity can be rewritten as
        \[
        P(E,A \times (a,b))=P(E',A \times (a,b))
        =P(E',A \times I)
        -
        P(E',A \times \{a\}).
        \]
        In particular to conclude we only need to show that
        \begin{equation} \label{E19}
        P(E',A \times \{a\})=0.
        \end{equation}
        To this aim fix $\delta>0$ and, using Lemma \ref{LP23}, pick a sequence decreasing to zero $\{\epsilon_i\}_{i \in \bb{N}}$ such that
        \begin{equation} \label{E18}
    P(E \cap \X \times (-\infty,a+\epsilon_i),A \times (a,a+\delta)) \geq \m(A) - 1/i.
    \end{equation}
        So, by the lower semi-continuity of perimeters,
        \begin{align*}
        P(E',A \times (a-\delta,a+\delta))
        & \leq \liminf_{i \to + \infty}
        P(E \cup \X \times (-\infty,a+\epsilon_i),A \times (a-\delta,a+\delta))\\
        & =
        \liminf_{i \to +\infty}
        P(E \cup \X \times (-\infty,a+\epsilon_i),A \times (a,a+\delta)).
        \end{align*}
        Using \eqref{E18} we get
        \begin{align*}
        P(E \cup \X& \times (-\infty,a+\epsilon_i), A \times (a,a+\delta)) \\
       & \leq P(E,A \times (a,a+\delta))+
        \m(A)-P(E \cap \X \times (-\infty,a+\epsilon_i),A \times (a,a+\delta))\\
       & \leq 
        P(E,A \times (a,a+\delta)) + 1/i.
        \end{align*}
       Summing up we obtained that
       \[
       P(E',A \times (a-\delta,a+\delta)) \leq P(E,A \times (a,a+\delta)),
       \]
       which implies \eqref{E19}.
    \end{proof}
\end{lemma}

\begin{proposition}[Key Step: $E'$ is locally perimeter minizing] \label{P25}
    $E'$ is perimeter minimizing in $U_{(x,a)}$.
    \begin{proof} 
        Let $B \subset \X \times \bb{R}$ be a Borel set such that $B \Delta E' \subset \subset U_{(x,a)}$ and note that $B \cap \X \times (a,+\infty)$ is a competitor for $E$ in $U_{(x,a)}$.
        In particular, using subadditivity of perimeters first and the fact that $E$ minimizes the perimeter in $U_{(x,a)}$ later, we get
        \[
        P(B,U_{(x,a)}) \geq 
        P(B \cap \X \times (a,+\infty),U_{(x,a)})-P( \X \times (a,+\infty),U_{(x,a)})
        \]
        \[
        \geq P(E,U_{(x,a)})-P( \X \times (a,+\infty),U_{(x,a)}).
        \]
        By Proposition \ref{L12} this last quantity is greater than or equal to $P(E',U_{(x,a)})$, so that summing up we obtained
        \[
        P(B,U_{(x,a)}) \geq P(E',U_{(x,a)}),
        \]
        concluding the proof.
    \end{proof}
\end{proposition}

\begin{lemma} \label{L18}
    Let $y \in A$ be a regular point of $\X$, then $(y,a) \notin  \partial E'$. In particular there exists a small ball $B \subset \X \times \bb{R}$ containing $(y,a)$ such that $B \cap \partial E'=\emptyset$.
    \begin{proof}
        Suppose by contradiction that $y \in \partial E'$. 
        Now let $(\ssf{Z},\sd_Z)$ be the space realizing the blow up of $\X$ at $y$ and observe that $(\ssf{Z} \times \bb{R},\sd_e \times \sd_Z)$ naturally realizes the blow up of $\X \times \bb{R}$ at $(y,a)$ along the same scales. 
        When doing such blow up the set $E'$ converges in $\sL^1_{loc}$ to a perimeter minimizing set $E'' \subset \bb{R}^k \times \bb{R}$ such that
        \[
        \bb{R}^k \times (-\infty,0) \subset E''  
        \quad \text{and} \quad 0 \in \partial E''.
        \]
        Moreover $0 \in \partial (E'' \cap \bb{R}^k \times (0,+\infty))$ since $E'' \cap \bb{R}^k \times (0,+\infty)$ is the blow up of $E$ at $(y,a)$. In particular the complement ${^c}E''$ cannot be a half space, contradicting the global Bernstein Property in Euclidean space.
    \end{proof}
\end{lemma}

\begin{corollary} \label{C14}
    $\partial E' \cap \X \times \{a\} \cap U_{(x,a)}= \emptyset$.
    \begin{proof}
        Suppose by contradiction that $(y,a) \in \partial E' \cap \X \times \{a\} \cap U_{(x,a)}$. Consider then the function $\sd_{\partial E'}+|\sd_0|$ defined in a small neighbourhood of $(y,a)$ in $\X \times (-\infty,a)$. If the neighbourhood is small enough this function is super-harmonic (since $E'$ is locally perimeter minizing) and attains its minimum on each point of the vertical line connecting $(y,a)$ and $(y,0)$ (since $(y,a) \in \partial E'$). Hence the aforementioned function has to be constant, contradicting the previous lemma. 
    \end{proof}
\end{corollary}

\begin{lemma} \label{L22}
    $\overline{(C \setminus \X \times \{a\})} \cap \X \times \{a\}= \emptyset$. In particular $C \setminus \X \times \{a\}$ is still locally a boundary of a locally perimeter minizing set. 
    \begin{proof}
        It is immediate to check that for every $x \in \X$ we have
        \[
        U_{(x,a)} \cap \overline{(C \setminus \X \times \{a\})} \cap \X \times \{a\} \subset
        U_{(x,a)} \cap \partial E' \cap \X \times \{a\}
        \]
        and this last set is empty by the previous lemma.
    \end{proof}
\end{lemma}

The next theorem gives the explicit description of $C$ and this immediately implies the Half Space Property for $\RCD(0,N)$ spaces, i.e. Theorem \ref{CT1}.

\begin{thm}[$C$ is made of horizontal slices] \label{T11}
    There exists a sequence $\{a_i\}_{i \in \bb{N}} \subset (0,+\infty]$ increasing to $+ \infty$ and $N \in \bb{N} \cup \{+ \infty \}$ such that
    \[
    C=\bigcup_{i \in \bb{N}, \, i \leq N} \X \times \{a_i\}.
    \]
    \begin{proof}
        We know that there exists $a_1:=a>0$ such that
        \[
        \X \times (-\infty,a_1) \cap C= \emptyset \quad \text{and} \quad \X \times \{a_1\} \subset C.
        \]
        If $C=\X \times \{a_1\}$ we have nothing left to prove, so suppose that this is not the case.
        Then the set $C_1:=C \setminus \X \times \{a_1\}$ is still locally a boundary of a perimeter minimizing set, so that
        there exists $a_2 > a_1$ such that
        \[
        \X \times (-\infty,a_2) \cap C_1= \emptyset
        \quad \text{and} \quad 
        \X \times \{a_2\} \subset C_1.
        \]
        If $C_2:=C_1 \setminus \X \times \{a_2\}=\emptyset$ the proof is concluded, otherwise we apply the same argument to $C_2:=C_1 \setminus \X \times \{a_2\}$ and we keep iterating this process as long as $C_k \neq \emptyset$. If there exists $N \in \bb{N}$ such that $C_{N} = \emptyset$ we set $a_i=+\infty$ for every $i \geq N+1$ and the proof is concluded.
    If this is not the case, and $a_i \to + \infty$ we set $N:=+\infty$ and the proof is again concluded. \par 
    We claim that the previous two cases are the only possible. Indeed if $C_k \neq \emptyset$ for every $k$ and there exists $b \in \bb{R}$ such that $a_i \uparrow b$, then in any neighbourhood of a point in $\X \times \{b\}$ the set $E$ whose boundary locally coincides with $C$ would have infinite perimeter, a contradiction.
    \end{proof}
\end{thm}

\section{Proof of Theorem \ref{CT2Mod}} \label{SGlobal}

The goal of this section is to prove Theorem \ref{T|GlobalMin}, which implies Theorem \ref{CT2Mod} from the Introduction. The proof of this result is similar in spirit to the proofs of the Half Space Property on manifolds given in \cite{RSS,CMMR,DingCap}. The idea is the following. 
Let $\X$ be a parabolic $\RCD(K,N)$ space and let $E \subset \X \times \bb{R}$ be a perimeter minimizing set contained in $\X \times (0,+\infty)$. 
If $\partial E$ is not made of horizontal slices, one is able to construct functions $u_i \in \BV(\X)$ which solve an obstacle problem and have large oscillations. This is due to the fact that $\partial E$ acts as a barrier, so that the graph of each $u_i$ lies at the same time above the obstacle and below $E$, forcing $u_i$ to oscillate. 
In \cite{RSS,CMMR, DingCap}, $\partial E$ can be used as a barrier since, in a Riemannian manifold, minimal hypersurfaces that are tangent to each other must coincide due to the maximum principle. We recall that the maximum principle can be applied thanks to the ellipticity of the minimal surface equation, which crucially relies on the smoothness of the ambient space.
In our setting, we can use $\partial E$ as a barrier thanks to a comparison argument that relies instead on the assumption that $E$ is perimeter minimizing.

Then, by using the fact that $\X$ is parabolic, one obtains that the functions $u_i$ have to converge to a constant function, violating the previous condition on the oscillation. In \cite{RSS,CMMR, DingCap}, this is obtained by working on the graphs of the functions $u_i$ and exploiting the regularity properties of minimal graphs (e.g.,\ harmonicity of the height function on the  graph and second variation formulas). In our setting, we only rely on integration by parts arguments which require no additional regularity for the functions $u_i$ (which are only $\BV$ in our setting).

We now recall the definition of area functional on $\RCD$ spaces (see \cite{Tens}, \cite{Brena22}, \cite{relax}, \cite{HKLreg}, and \cite{AreaFormula}) and a few basic properties that are needed to prove Theorem \ref{T13}.

\begin{definition} [Area functional]\label{D|Area functional}
    Let $(\X,\sd,\m)$ be an $\RCD(K,N)$ space, let $\Omega \subset \X$ be an open set and let $u \in \BV(\Omega)$. For every measurable $E \subset \Omega$ we define the area of $u$ on $E$ as
    \[
    \A(u,E):=\int_E \sqrt{1+|\nabla u|^2} \, d\m + |D^s u|(E).
    \]
\end{definition}

Let $\Omega \subset \X$ and let $u:\Omega \to \bb{R}$. We set
\[
\Hyp(u):=\{(x,t) \subset \Omega \times \bb{R}:t \leq u(x)\}, \quad \Epi(u):=\{(x,t) \subset \Omega \times \bb{R}:t \geq u(x)\}.
\]
The next result can be found in \cite{Brena22} (see also \cite{Tens} and \cite{AreaFormula}).

\begin{thm}[Area Formula]
Let $(\X,\sd,\m)$ be an $\RCD(K,N)$ space.
    Let $\Omega \subset \X$ be an open bounded set and let $u \in \BV(\Omega)$. For every Borel set $E \subset \Omega$ it holds
\[
P(\ssf{Hyp}(u),E \times \bb{R})=\A(u,E).
\]
\end{thm}


\begin{definition}[Symmetrized function]
    Let $(\X,\sd,\m)$ be an $\RCD(K,n)$ space. Let $\Omega \subset \X$ be open.
Let $E \subset \X \times \bb{R}$ be such that $\Omega \times (-\infty,0) \subset E$ and $ \Omega \times (-\infty,0) \Delta E \subset \subset \Omega \times \bb{R}$. We then define $w(E):\Omega \to \bb{R}$ by
    \[
    w(E)(x):=\int_0^{+\infty} 1_E(x,s) \, ds.
    \]
\end{definition}

The proof of the next proposition follows by adapting arguments from \cite[Section $3.2$]{C2} and we report it for the sake of completeness.

\begin{proposition} \label{P40}
Let $(\X,\sd,\m)$ be an $\RCD(K,n)$ space. Let $\Omega \subset \X$ be open.
Let $E \subset \X \times \bb{R}$ be such that $\Omega \times (-\infty,0) \subset E$ and $ \Omega \times (-\infty,0) \Delta E \subset \subset \Omega \times \bb{R}$.  Then, $w(E) \in \BV(\Omega)$ and $\A(w(E),\Omega) \leq P(E,\Omega \times \bb{R}).$
\begin{proof}
Let $c>0$ be such that $ \Omega \times (-\infty,0) \Delta E \subset \subset \Omega \times (-\infty,c)$ and let $\epsilon>0$. 
    Consider a sequence $f_n \in \Lip(\Omega \times (-\epsilon,c))$ converging in $\sL^1(\Omega \times (-\epsilon,c))$ to $1_E$ and such that
    \[
    \lim_{n \to + \infty} |D f_n|(\Omega \times (-\epsilon,c))=P(E,\Omega \times \bb{R}).
    \]
    Modulo truncating, we can assume that $f_n \equiv 1$ on $\Omega \times \{-\epsilon\}$ and $f_n \equiv 0$ on $\Omega \times \{c\}$.
    We then define $w_{\epsilon}(f_n): \Omega \to \bb{R}$ and $w_\epsilon(E):\Omega \to \bb{R}$ as
    \[
    w_\epsilon(f_n)(x):=\int_{-\epsilon}^c f_n(x,s) \, ds, \quad 
    w_{\epsilon}(E)(x):=\int_{-\epsilon}^{+\infty} 1_E(x,s) \, ds.
    \]
    The functions $w_\epsilon(f_n)$ are Lipschitz since
    \[
    \frac{|w_{\epsilon}(f_n)(x)-w_{\epsilon}(f_n)(y)|}{\sd(x,y)} \leq \int_{-\epsilon}^c \frac{|f_n(x,s)-f_n(y,s)|}{\sd(x,y)} \, ds \leq (c+\epsilon) \ssf{L}(f_n).
    \]
    We now use the notation defined before Proposition \ref{P38}.
    By reverse Fatou Lemma and the fact that each $f_n$ is Lipschitz, for $\m$-a.e. $x \in \X$ it holds
    \begin{equation} \label{E30}
    |\nabla w_{\epsilon}(f_n)|(x)=\lip (w_{\epsilon}(f_n))(x) \leq \int_{-\epsilon}^c |\nabla f_n^s|(x) \, ds.
    \end{equation}
    Moreover, using the tensorization property of Proposition \ref{P38}, it holds
    \begin{align*}
    |D f_n|(\Omega \times (-\epsilon,c))
    & =
    \int_{\Omega \times (-\epsilon,c)} |\nabla f_n |(x,s) \, d\m \, ds =
    \int_{\Omega \times (-\epsilon,c)} \sqrt{|\nabla f_n^s|^2(x)+|\nabla f_n^x|^2(s)} \, d\m \, ds \\
    & \geq 
    \sup_{\substack{(a,b) \in C(\Omega) \times C(\Omega) \\ a^2+b^2 \leq 1 \\ a,b \geq 0}}
    \int_{\Omega \times (-\epsilon,c)} a(x) |\nabla f_n^s|(x)+b(x)|\nabla f_n^x|(s) \, d\m \, ds.
    \end{align*}
    Combining with \eqref{E30} and the fact that $f_n=1$ on $\Omega \times \{-\epsilon\}$ and $f_n=0$ on $\Omega \times \{c\}$, we deduce
    \begin{align*}
        |D f_n|(\Omega \times (-\epsilon,c)) \geq 
    \sup_{\substack{(a,b) \in C(\Omega) \times C(\Omega) \\ a^2+b^2 \leq 1 \\ a,b \geq 0}}
    \int_\Omega a(x) |\nabla w_{\epsilon}(f_n)|(x)+b(x) \, d\m=\A(w_{\epsilon}(f_n),\Omega).
    \end{align*}
    Since
    \[
    \int_\Omega |w_{\epsilon}(f_n) - w_{\epsilon}(E)| \, d\m \leq \int_{\Omega \times (-\epsilon,c)}|f_n(x,s)-1_E(x,s)| \, d\m \, ds,
    \]
    it holds $w_\epsilon(f_n) \to w_\epsilon(E)$ in $\sL^1(\Omega)$ as $n \to + \infty$. 
    These facts imply that $w_\epsilon(E) \in \BV(\Omega)$. Concerning the area of $w_\epsilon(E)$, we get
    \[
    \A(w_\epsilon(E),\Omega) \leq \liminf_{n \to + \infty} \A(w_\epsilon(f_n),\Omega) \leq \liminf_{n \to + \infty} |D f_n|(\Omega \times (-\epsilon,c))=P(E,\Omega \times \bb{R}).
    \]
    Since $w(E)=w_\epsilon(E)-\epsilon$, the statement follows.
    \end{proof}
    \end{proposition}

\begin{lemma} \label{L23}
Let $(\X,\sd,\m)$ be an $\RCD(K,N)$ space.
    Let $B \subset \X$ be a perimeter minimizing set. Let $A \subset \X$ be a set of finite perimeter such that $A \cap B \subset \subset \Omega$ for some open set $\Omega \subset \subset \X$. Then, $P(A \cap {^c}B,\Omega) \leq P(A,\Omega)$.
    \begin{proof}
        By the usual perimeter identities (see \cite{AmbPeri}),
        \[
        P(A \cap {^c}B,\Omega) \leq  P(A,\Omega)+P(B,\Omega)-P(A \cup {^c}B,\Omega).
        \]
        Since $B$ is perimeter minimizing, it holds $P(B,\Omega)-P(A \cup {^c}B,\Omega) \leq 0$, concluding the proof.
    \end{proof}
\end{lemma}

The next theorem is the key result to prove Theorem \ref{T|GlobalMin}.

\begin{thm}[Key Step for Theorem \ref{CT2Mod}] \label{T13}
    Let $(\X,\sd,\m)$ be a parabolic $\RCD(K,N)$ space. Let $E \subset \X \times \bb{R}$ be a perimeter minimizing set such that $\partial E \subset \X \times (0,+\infty)$. Then, there exists $a>0$ such that $\partial E \cap \X \times (-\infty,a]=\X \times \{a\}$. 
\end{thm}
\begin{proof}
    In this proof, $E$ will denote the open representative. Modulo passing to the complement, we assume that $E \subset \X \times (0,+\infty)$. If the statement fails, then there exist $x \in \X$, $s,h>0$ such that $B_s(x) \times [0,h] \cap E=\emptyset$, while $\X \times [0,h] \cap E \neq \emptyset$. For simplicity, we assume $h=1$.

    For every $i \in \bb{N}$ with $i>s$ consider functions $u_i \in \BV(\X)$ such that
    \begin{equation} \label{E28}
    \ssf{A}(u_i,B_{i+1}(x))=\min \{\A(u,B_{i+1}(x)): u \in \BV(\X) , \, u=1 \text{ on }B_s(x), \, u=0 \text{ on }\X \setminus B_{i}(x)\}.
    \end{equation}
    These functions exist since the area functional is lower semicontinuous. Note that each $u_i$ only takes values in $[0,1]$. We now divide the proof in steps
    
    \textbf{Step 1}: We set $v_i:=w(\ssf{Hyp}(u_i)\setminus E)$ and we claim that 
    \[
    \A(v_i,B_{i+1}(x))=\ssf{A}(u_i,B_{i+1}(x)).
    \]  
    To prove the claim, note that by Proposition \ref{P40}, it holds
    \[
    \ssf{A}(v_i,B_{i+1}(x)) \leq P(\ssf{Hyp}(u_i) \setminus E,B_{i+1}(x) \times \bb{R}).
    \]
    Since $E$ is perimeter minimizing, by Lemma \ref{L23} it holds
    \[
    P(\ssf{Hyp}(u_i) \setminus E,B_{i+1}(x) \times \bb{R})
    \leq P(\ssf{Hyp}(u_i),B_{i+1}(x) \times \bb{R})=\ssf{A}(u_i,B_{i+1}(x)).
    \]
    Since $w(\ssf{Hyp}(u_i) \setminus E)$ is equal to $1$ on $B_s(x)$ and to $0$ on $\X \setminus B_i(x)$, we also have $\ssf{A}(u_i,B_{i+1}(x)) \leq \ssf{A}(v_i,B_{i+1}(x))$.
    This concludes the proof of Step $1$.
    
    \textbf{Step 2}: We claim that, modulo passing to a subsequence, the functions $v_i$ converge to $1$ in $\sL^1_{loc}(\X)$.
    
    To prove the claim, it suffices to show that for every compact $K \subset \X$, it holds $|D v_i|(K) \to 0$ as $i \uparrow + \infty$. If this holds, the claim follows by lower semicontinuity of total variations.
    
    Fix then $K \subset \subset \X$ and let $i \in \bb{N}$ be large enough. Let $\phi \in \Lip_c(\X)$ be a non-negative function compactly supported in $B_i(x)$ which is equal to $1$ on $K$. Observe that $(1-v_i)\phi^2$ is nonnegative. 
    By the chain rule (see \cite[Proposition 4.35]{BrenaLeibniz}), for $t$ small enough it holds
    \[
    |D^s v_i(1-t\phi^2)|(\X)=\int_\X(1-t\phi^2) \, d|D^s v_i|.
    \]
    
    Combining this with standard arguments involving Dominated Convergence Theorem (see \cite[Theorem 2.27]{Folland}), we obtain that the function $f(t):=\ssf{A}(v_i+t(1-v_i)\phi^2,B_{i+1}(x))$ is differentiable in $t=0$ with
    \begin{equation} \label{E35}
    f'(0)=\int_\X \frac{\nabla v_i \cdot  \nabla ((1-v_i)\phi^2)}{\sqrt{1+|\nabla v_i|^2}} \, d\m-\int_\X \phi^2 \, d|D^s v_i|.
    \end{equation}
    By Step $1$, $v_i$ is also a minimizer of the problem \eqref{E28}. Combining with \eqref{E35}, it holds
    \begin{align*}
    0 \leq -\int_\X \phi^2 \frac{|\nabla v_i|^2}{\sqrt{1+|\nabla v_i|^2}} \, d\m+
    \int_\X (1-v_i) \frac{\nabla v_i \cdot \nabla \phi^2}{\sqrt{1+|\nabla v_i|^2}} \, d\m
    -\int_\X \phi^2 \, d|D^s v_i|.
    \end{align*}
    Rearranging and using Young's inequality, it holds
    \begin{align*}
        \int_\X \phi^2 \frac{|\nabla v_i|^2}{\sqrt{1+|\nabla v_i|^2}} \, d\m+\int_\X \phi^2 \, d|D^s v_i| \leq \frac{1}{2} \int_\X \phi^2 \frac{|\nabla v_i|^2}{\sqrt{1+|\nabla v_i|^2}} \, d\m+8\int_\X |\nabla \phi|^2 \, d\m.
    \end{align*}
    
    Hence, using that $\ssf{Cap}(K)=0$ by Theorem \ref{CT4}, it holds $|D v_i|(K) \to 0$ as $i \uparrow + \infty$ as claimed, concluding the proof of Step $2$.
 
    We reach a contradiction since $\X \times [0,1] \cap E \neq \emptyset$, so that, by definition of $v_i$, it holds $v_i \leq 1-\epsilon$ on an open set for $\epsilon>0$ small enough.
\end{proof}

Let $E \subset \X$ be a perimeter minimizing set such that $\partial E \subset \X \times (0,+\infty)$ and $E \subset \X \times (0,+\infty)$. Let $a>0$ be given by Theorem \ref{T13} and set $E':=E \cup \X \times (-\infty,a)$.
    Note that all the results from Lemma \ref{LP28} to Lemma \ref{L18} hold for this choice of $E$ and $E'$ with $U_{(x,a)}=V \times \bb{R}$, where $V$ is any open bounded set of $\X$ (the arguments are formally the same replacing Corollary \ref{C9} with Theorem \ref{T13} and using our stronger assumptions on $E$).

    The next lemma shows that an analogue of Corollary \ref{C14} holds as well. This does not follow automatically, since in the proof of Corollary \ref{C14} we use the assumption that $\X$ is $\RCD(0,N)$. In the next lemma, $\partial E'$ is the boundary of the closed (or open) representative of $E'$, which exists by Proposition \ref{P25}.

\begin{lemma} \label{L19}
Let $(\X,\sd,\m)$ be a parabolic $\RCD(K,N)$ space. Let $E \subset \X \times \bb{R}$ be a perimeter minimizing set such that $\partial E \subset \X \times (0,+\infty)$ and $E \subset \X \times (0,+\infty)$. Let $a>0$ be given by Theorem \ref{T13} and set $E':=E \cup \X \times (-\infty,a)$. Then, $\partial E' \cap \X \times \{a\} = \emptyset$.
\begin{proof}
    Observe that by Lemma \ref{L18}, $\partial E' \cap \X \times \{a\} \neq \X \times \{a\}$. By Proposition \ref{P25} and Theorem \ref{T13} (applied to the complement of $E'$), it holds $\partial E' \cap \X \times \{a\} = \emptyset$.
    \end{proof}
\end{lemma}

\begin{thm}[General version of Theorem \ref{CT2Mod}] \label{T|GlobalMin}
    Let $(\X,\sd,\m)$ be a parabolic $\RCD(K,N)$. Let $E \subset \X \times \bb{R}$ be a perimeter minimizing set such that $E \subset \X \times (0,+\infty)$. 
    There exists a sequence $\{a_i\}_{i \in \bb{N}} \subset (0,+\infty]$ increasing to $+ \infty$ and $N \in \bb{N} \cup \{+ \infty \}$ such that
    \[
    \partial E=\bigcup_{i \in \bb{N}, \, i \leq N} \X \times \{a_i\}.
    \]
    In particular, if $\partial E$ is connected, then it is a horizontal slice.
    \begin{proof}
    The proof is now similar the one of Theorem \ref{T11}.
         We know that there exists $a_0>0$ such that
        \[
        \X \times (-\infty,a_0) \cap \partial E= \emptyset \quad \text{and} \quad \X \times \{a_0\} \subset \partial E.
        \]
        If $\partial E=\X \times \{a_0\}$ we have nothing left to prove, so suppose that this is not the case.
        Then, the set $E_0:=E \cup \X \times (-\infty,a_0)$ is still a perimeter minimizing set by Proposition \ref{P25} and it has $\partial E_0 \cap \X \times \{a_0\} = \emptyset$ by Lemma \ref{L19}. Hence, by Theorem \ref{T13} (applied to ${^c}E_0$),
        there exists $a_1 > a_0$ such that
        \[
        \X \times (-\infty,a_1) \cap \partial E_0= \emptyset
        \quad \text{and} \quad 
        \X \times \{a_1\} \subset \partial E_0 \subset \partial E.
        \]

        If $\partial E_0 \setminus \X \times \{a_1\}=\emptyset$ the proof is concluded, otherwise we set $E_1:={^c}E_0 \cup \X \times (-\infty,a_1)$. Hence,
        we find $a_2 >a_1$ such that
         \[
        \X \times (-\infty,a_2) \cap \partial E_1= \emptyset
        \quad \text{and} \quad 
        \X \times \{a_2\} \subset \partial E_1 \subset \partial E.
        \]
        We then define $E_2:={^c}E_1 \cup \X \times (-\infty,a_2)$
        and we keep iterating this process, setting $E_k:={^c}E_{k-1} \cup \X \times (-\infty,a_k)$ as long as $\partial E_{k-1} \setminus \X \times \{a_k\} \neq \emptyset$. If there exists $N \in \bb{N}$ such that $\partial E_{N-1} \setminus \X \times \{a_N\}
        = \emptyset$ we set $a_i=+\infty$ for every $i \geq N+1$ and the proof is concluded.
    If this is not the case, and $a_i \to + \infty$ we set $N:=+\infty$ and the proof is again concluded. 
    
    We claim that the previous two cases are the only possible. Indeed if $\partial E_{k-1} \setminus \X \times \{a_k\} \neq \emptyset$ for every $k$ and there exists $b \in \bb{R}$ such that $a_i \uparrow b$, then $E$ would have infinite perimeter, a contradiction.
    \end{proof}
\end{thm}

The next theorem deals with the case where the space $(\X,\sd,\m)$ has slow volume growth (see Definition \ref{Dslow}). We remark that if $X$ is an $\RCD(0,N)$ space, having slow volume growth is equivalent to being parabolic by Theorem \ref{CT5}.

\begin{thm}[Variant of Theorem \ref{CT2Mod} for spaces with slow volume growth] \label{T|MinPerSlow}
    Let $(\X,\sd,\m)$ be an $\RCD(K,N)$ space with slow volume growth. Let $E \subset \X \times \bb{R}$ be a perimeter minimizing set such that $E \subset \X \times (0,+\infty)$. Then, there exists $a \geq 0$ such that $E=\X \times [a,+\infty)$. 
    \begin{proof}
        By the previous theorem, the boundary of $E$ is made of horizontal slices. We assume by contradiction that $\partial E$ has more than $1$ connected component. 
        
    Suppose first that $\X$ is compact. In this case subtracting from $E$ its lowest connected component we obtain a competitor for $E$ with strictly less perimeter, a contradiction. So we can suppose that $\X$ has infinite diameter. 
    
    Let $C>0$ be the distance between the lowest and the second lowest connected component of $\partial E$. We claim that there exists $s>0$ such that
    \begin{equation} \label{E34}
    CP(B_s(\x),\X) < 2 \m(B_s(\x)).
    \end{equation}
    Recall that, by coarea formula, the function $ s \mapsto \m(B_s(\x))$ is absolutely continuous and satisfies $\frac{d}{ds} \m(B_s(\x)) = P(B_s(\x))$ for a.e. $s>0$.
    Hence, if \eqref{E34} fails for every $s>0$, it would follow that  $\frac{d}{ds} \m(B_s(\x)) \geq 2C^{-1} \m(B_s(\x))$ and thus $\m(B_s(\x))$ grows exponenentially in $s$. This contradicts the slow volume growth assumption and proves the claimed inequality \eqref{E34}. 
    
    So let $s_0>0$ be a value satisfying \eqref{E34}, let $\X \times (a,a+C) \subset E$ be the lowest connected component of $E$ (up to replacing $E$ with its complement, we may assume that the lowest connected component has this form), let $\epsilon>0$ be small enough so that $\X \times (a,a+C+\epsilon) \cap E=\X \times (a,a+C)$ and consider the set
    \[
    A:= (\X \setminus B_{s_0}(\x)) \times (a,a+C).
    \]
    Observe that $A \Delta E \subset \subset B_{s_0+1}(\x) \times (a-\epsilon,a+C+\epsilon)$, while at the same time
    we have that
    \begin{align*}
    P &(A,B_{s_0+1}(\x) \times (a-\epsilon,a+C+\epsilon))
    \\
    &=C P(B_{s_0}(\x),\X)+P(\X \times (a,a+C),(B_{s_0+1}(\x) \setminus \bar{B}_{s_0}(\x)) \times (a-\epsilon,a+C+\epsilon))
    \\
    &< 2 \m(B_{s_0}(\x))+
    P(\X \times (a,a+C),(B_{s_0+1}(\x) \setminus \bar{B}_{s_0}(\x)) \times (a-\epsilon,a+C+\epsilon))
    \\
    &= P(E,B_{s_0+1}(\x) \times (a-\epsilon,a+C+\epsilon)),
    \end{align*}
    contradicting the fact that $E$ is a perimeter minimizer and proving that $\partial E$ has only $1$ connected component, which then implies the statement.
    \end{proof}
\end{thm}

We conclude the section with an example showing that there exists an infinitesimally Hilbertian metric measure space which is parabolic (meaning that it admits a point $x$ with no positive Green's function with pole $x$) and does not have the Half Space Property.

\begin{example} \label{Ex1}
    Consider the metric measure space $(\X,\sd,\m)$ obtained by gluing $([0,+\infty),\sd_e,(x \wedge 1) \lambda^1)$ and $(\bb{R},\sd_e,\lambda^1)$ in the respective origins.
    Using the fact that weighted manifolds are infinitesimally Hilbertian (see \cite{LPinf}), and that such space is a weighted manifold outside of the gluing point, we obtain that $\X$ is infinitesimally Hilbertian as well. We now show that $\X$ is parabolic in $2$ steps. \par
    \textbf{Step 1}:
    Denote by $(a_{\bb{R}},b_{\bb{R}})$ intervals in $(\bb{R},\sd_e,\lambda^1)$ and by $(a_{\bb{\ssf{Y}}},b_{\bb{\ssf{Y}}})$  intervals in $([0,+\infty),\sd_e,(x \wedge 1) \lambda^1)$. We claim that, given $c_{\bb{R}} \in (0,1)$, if a function $u \in \W^{1,1}((-c_{\bb{R}},c_{\bb{R}}) \cup (0_\ssf{Y},c_\ssf{Y}))$ is harmonic in its domain, then the restriction of $u$ to $(-c_{\bb{R}},c_{\bb{R}})$ is affine. \\
    To this aim note that the restriction of $u$ to $(-c_{\bb{R}},c_{\bb{R}})$ belongs to $\W^{1,1}((-c_{\bb{R}},c_{\bb{R}}))$ and in particular it is continuous. \par 
    Moreover by restricting $u$ to open sets of its domain that do not contain the gluing point we obtain, by harmonicity, that $u$ is affine in $(-c_{\bb{R}},0_{\bb{R}})$ and $(0_{\bb{R}},c_{\bb{R}})$ (possibly with different slopes) and that $u$ is smooth in $(0_\ssf{Y},c_\ssf{Y})$. Hence, if we manage to prove that the restriction of $u$ to $(-c_{\bb{R}},c_{\bb{R}})$ is $C^1$ we conclude the proof of Step $1$. 
    \par 
    To this aim let $\phi \in \Lip_c((-c_{\bb{R}},c_{\bb{R}}) \cup (0_\ssf{Y},c_\ssf{Y}))$ and consider the harmonicity condition (we are using that relaxed gradients coincide with local Lipschitz constants in our space and we are denoting by $u'$ and $\phi'$ standard derivatives):
    \[
    0=\int_{(-c_{\bb{R}},c_{\bb{R}}) \cup (0_\ssf{Y},c_\ssf{Y})} \nabla u \cdot \nabla \phi \, d\m=\int_{-c_{\bb{R}}}^0 u'(x) \phi'(x) \, d\lambda^1+\int_{0}^{c_{\bb{R}}} u'(x) \phi'(x) \, d\lambda^1+\int_{0_\ssf{Y}}^{c_\ssf{Y}} u'(x) \phi'(x)x \, d\lambda^1(x).
    \]
    Integrating by parts and using the previously found regularity we obtain
    \[
    0=\phi(0)\Big(\lim_{\substack{x \in (-c_{\bb{R}},0) \\ x \to 0}} u'(x)-\lim_{\substack{x \in (0,c_{\bb{R}}) \\ x \to 0}}u'(x)\Big),
    \]
    which shows that the restriction of $u$ to $(-c_{\bb{R}},c_{\bb{R}})$ is $C^1$ and concludes the proof of Step $1$. \par 
    \textbf{Step 2}: We now prove that $\X$ is parabolic. \\
    By step $1$ we deduce that if $x \in \bb{R} \setminus \{0\}$ and $G_x$ is the corresponding Green's function on $\X$, then the restriction of $G_x$ to $\bb{R}$ is a Green's function on the standard real line. This in particular implies that $G_x$ cannot be positive, proving that $\X$ is parabolic. \par 
    We finally show that the Half Space Property fails on $\X$. This is done again in two steps. \par 
    \textbf{Step 1}.
    Consider the function $1_{\bb{R}}:\X \to \bb{R}$ defined as $1_{\bb{R}}\equiv 1$ on $\bb{R}$ and $1_{\bb{R}}\equiv 0$ on $ \X\setminus \bb{R}$.
    We claim that its epigraph $\Epi(1_{\bb{R}}) \subset \X \times \bb{R}$ satisfies $P(\Epi(1_{\bb{R}}),\{0 \} \times \bb{R})=0$. \\
    Since $P(\Epi(1_{\bb{R}}),\{0 \} \times \{0,1\})=0$, to prove the above claim
     it is sufficient to show that
    $P(\Epi(1_{\bb{R}}),\{0 \} \times [\epsilon,1-\epsilon])=0$ for every $\epsilon>0$ small enough. This would follow if we can show that for $\delta>0$ fixed we have
    \[
    P(\Epi(1_{\bb{R}}),B^\X_\delta(0) \times [\epsilon,1-\epsilon]) =0.
    \]
    So let $\delta>0$ be fixed and consider the sequence $f_n$ given by
    \[
    f_n((x,t)):=1_{(0_\ssf{Y},1_\ssf{Y})}(x)((1-n\sd(0,x)) \vee 0)+1_\bb{R}(x).
    \]
    Using this as approximating sequence in the definition of perimeter we get
    \[
    P(\Epi(1_{\bb{R}}),B^\X_\delta(0) \times [\epsilon,1-\epsilon]) \leq \liminf_{n \to + \infty} (1-2\epsilon) n\int_0^{1/n} x \, dx  = 0,
    \]
    concluding the proof of Step $1$. \par 
    \textbf{Step 2}: We now show that $\Epi(1_{\bb{R}})$ is a perimeter minimizer. \\
    By Step $1$ we obtain that in any set of the form $B^\X_R(0) \times (-1,2) \subset \X \times \bb{R}$ we have
    \[
    P(\Epi(1_{\bb{R}}),B^\X_R(0) \times (-1,2))=P(\X \times [0,+\infty),B^\X_R(0) \times (-1,2)),
    \]
    while one can check that if $E \Delta \Epi(1_{\bb{R}}) \subset \subset B^\X_R(0) \times (-1,2)$, then 
    \[
    P(E,B^\X_R(0) \times (-1,2)) \geq P(\X \times [0,+\infty),B^\X_R(0) \times (-1,2)).
    \]
    Putting these facts together we deduce that $E$ is a perimeter minimizer, as claimed.
\end{example}

\section{Applications} \label{S1}
In this section, whenever we consider a manifold $\M$ and a hypersurface $S \subset \M$, we implicitly mean that the hypersurface is properly embedded.
Given a manifold $(\M,g)$ with boundary we say that $\partial \M$ is convex if its second fundamental form $\Pi$ w.r.t. the inward pointing normal is positive. Moreover, given a hypersurface $S \subset \M$, we denote by $\mathbf{H}$ its mean curvature vector. 
The next proposition can be obtained repeating an argument in \cite[Theorem $2.4$]{H20}.

\begin{proposition} \label{P37}
    Let $(\M^n,\sd_g,e^{-V} d \m_g)$ be a  weighted manifold with convex boundary such that for a number $N>n$
    \[
    \Ric_\M+\Hess_V-\frac{\nabla V \otimes \nabla V}{N-n} \geq K \quad \text{ on } \M \setminus \partial \M,
    \]
    then $(\M^n,\sd_g,e^{-V} d \m_g)$ is an $\RCD(K,N)$ space.
\end{proposition}

\begin{definition}[Minimal hypersurfaces in weighted manifolds]
    Let $(\M^n,\sd_g,e^{-V} d \m_g)$ be a weighted manifold (possibly with boundary). We say that a hypersurface $S \subset \M$ is minimal if its mean curvature vector satisfies $\mathbf{H}=\nabla V^{\perp}$ on $S \setminus \partial \M$, where $\nabla V^{\perp}$ is the projection of $\nabla V$ on the normal bundle of $S$.
\end{definition}

As in the classical case, one can characterize minimality in terms of variations of the weighted area functional. In \cite{G12} it is shown that in a weighted manifold $(\M^n,\sd_g,e^{-V} d \m_g)$ without boundary a hypersurface $S$ is minimal if and only if the first variation of the weighted area functional vanishes at $S$. Moreover the same paper shows that $S$ is minimal in a weighted manifold without boundary $(\M^n,\sd_g,e^{-V} d \m_g)$ if and only if it is minimal in the manifold $(\M,g')$, where $g':=e^{-\frac{2}{n-1}V}g$. It is convenient to notice that the weighted area in $(\M^n,\sd_g,e^{-V} d\m_g)$ and the area with respect to the conformal metric $g'$ coincide, and that orthogonality of vectors is preserved under conformal changes of the metric. \par 
Now we state two results that are needed to obtain Corollary \ref{C13}, which shows that minimal hypersurfaces are locally boundaries of locally perimeter minimizing sets.
The next lemma is a consequence of the general properties of tubular neighbourhoods and Fermi coordinates in smooth manifolds (see for instance \cite[Lemma $2.7$]{tubes}).

\begin{lemma} \label{L21}
Let $(\M,g)$ be a manifold with boundary and let $\gamma:[0,1] \to \M$ be a geodesic intersecting $\partial \M$ orthogonally in $\gamma(1)$. Then there exists an interval $I=[1-\epsilon,1)$ such that for every $t \in I$ the curve $\gamma$ realizes the distance between $\partial \M$ and $\gamma(t)$. In particular, in an appropriate neighbourhood of $\gamma(1)$, any other curve branching from $\gamma$ and ending on $\partial \M$ has length greater than or equal to the one of $\gamma$.
\end{lemma}

The proof of the next lemma can be performed adapting the one in \cite[Theorem $2.1$]{LG96} and is just sketched, while the subsequent corollary follows from standard approximation arguments.

\begin{lemma} \label{L20}
    Let $(\M^n,\sd_g,e^{-V} d \m_g)$ be a weighted manifold (possibly with boundary). Let $S \subset \M$ be a minimal hypersurface intersecting $\partial \M$ orthogonally. Then $S$ locally minimizes the weighted area among smooth competitors.
    \begin{proof}
        Modulo considering the manifold $(\M,g')$, where $g':=e^{-\frac{2}{n-1}V}g$, we can suppose that the weight $V$ is identically zero.
        We then repeat the argument of \cite[Theorem $2.1$]{LG96}. \par 
        We first consider $x \in S \setminus \partial \M$ and we claim that $S$ minimizes the area w.r.t. smooth competitors in a neighbourhood of $x$. Let $U$ be a sufficiently small neighbourhood of $x$ in $\M$. To prove the claim let $C \subset S$ be an $n-2$ dimensional submanifold of $S$ containing $x$, let $P_0:S \cap U \to C$ be the nearest point projection from $S \cap U$ to $C$, and let $P_1: U \to S$ be nearest point projection from $U$ to $S$, and let $P:=P_0 \circ P_1$. Using the coarea formula we can rewrite the area of $S \cap U$ as
        \[
        \aH^{n-1} (S \cap U)=\int_C \int_{U \cap S_c} 1/J_{n-2}(P_{|S}) \, d\aH^{1} \,d \aH^{n-2}(c) 
        \]
        where $S_c:=P^{-1}(\{c\}) \cap S$ and $J_{n-2}(P_{|S})$ is the appropriate coarea factor. Since $P^{-1}(\{c\})$ intersects $S$ orthogonally, we have that $J_{n-2}(P_{|S})=J_{n-2}(P)$ on $S_c$ (while in general we would have
        $J_{n-2}(P_{|S}) \leq J_{n-2}(P)$), so that we can rewrite
        \[
        \aH^{n-1} (S \cap U)=\int_C \int_{U \cap S_c} 1/J_{n-2}(P) \, d\aH^{1} \,d \aH^{n-2}(c) 
        =
        \int_C \text{length}(S_c \cap U) \, d \aH^{n-2}(c),
        \]
        where the length of each $S_c \cap U$ is computed w.r.t. the metric obtained multiplying the standard metric on $\{P=c\}$ by $1/J_{n-2}(P)$. The minimality of $S$ implies, through the previous equation, that each curve $S_c$ is a geodesic in $\{P=c\}$ with the modified metric. In particular, modulo restricting $U$, we have that each $S_c \cap U$ is also length minimizing in $\{P=c\}$ with the modified metric. \par 
        With this in mind we have that if $S'$ is a competitor of $S$, denoting $S'_c:=S' \cap \{P=c\}$, then
        \begin{equation} \label{E33}
        \aH^{n-1} (S' \cap U)
        =
        \int_C \text{length}(S'_c \cap U) \,d \aH^{n-2}(c)
        \geq \int_C \text{length}(S_c \cap U) \,d \aH^{n-2}(c)
        = \aH^{n-1} (S \cap U).
        \end{equation}
        This proves the statement for points lying outside of $\partial \M$. \par 
        For boundary points, one can repeat the previous argument, replacing the fact that geodesics locally minimize the length with Lemma \ref{L21} to justify the inequality in \eqref{E33}.
    \end{proof}
\end{lemma}

\begin{corollary} \label{C13}
    Let $(\M^n,\sd_g,e^{-V} d\m_g)$ be a weighted manifold (possibly with boundary). Let $S \subset \M$ be a minimal hypersurface intersecting $\partial \M$ orthogonally, then $S$ is locally a boundary of a locally perimeter minimizing set.
\end{corollary}


Next, we deduce Theorem \ref{CT2} from the Introduction.

\begin{thm} \label{T12}
Let $N \in (1,\infty)$. Let $(\M^n,\sd_g,e^{-V} d\m_g)$ be a parabolic weighted manifold with convex boundary such that $N>n$ and
    \[
    \Ric_\M+\Hess_V-\frac{\nabla V \otimes \nabla V}{N-n} \geq 0 \quad \text{ on } \M \setminus \partial \M.
    \]
    Let $S \subset \M \times (0,+\infty)$ be a connected minimal hypersurface intersecting $\partial \M \times \bb{R}$ orthogonally, then $S$ is a horizontal slice.
    \begin{proof}
        The manifold $(\M^n,\sd_g,e^{-V} d \m_g)$ is an $\RCD(0,N)$ space by Proposition \ref{P37}. 
       The manifold $S$ is locally a boundary of a locally perimeter minimizing set and it is a slice by the Half Space Property.
    \end{proof}
\end{thm}

\begin{definition}[Modulus of parabolicity]
    Let $(\X,\sd,\m,\x)$ be an $\RCD(K,N)$ space. We say that a function
    $P:[0,+\infty) \to (0,+\infty)$ is a modulus of parabolicity if
    \[
    \int_1^{+\infty} \frac{t}{P(t)} \, dt = + \infty.
    \]
    We say that $\X$ has modulus of parabolicity $P$ if $\m(B_r(\x)) \leq \m(B_1(\x)) P(r)$ for every $r\geq 0$.
\end{definition}

By Theorem \ref{CT5}, if an $\RCD(K,N)$ space $(\X,\sd,\m,\x)$ has a modulus of parabolicity, then it is parabolic, and this condition is also necessary if $K=0$. In general, there are manifolds that are parabolic but do not have a modulus of parabolicity, as shown, for example, in \cite[Example $7.2$]{grig}. \par 
We say that a set $S \subset \X$ is an area minimizing boundary in an open set $A \subset \X$ if the exists a set $E \subset A$ minimizing the perimeter in $A$ such that $\partial E \cap A=S$. Let $S \subset \X \times \bb{R}$ be such that $(x,0) \in S$. We define
    \[
    \Osc_{x,r}(S):=\sup \{|t|:(y,t) \in S \cap B_r(x) \times (-r,r)\};
    \]
    \[
    \Osc^+_{x,r}(S):=\sup \{|t|:(y,t) \in S \cap B_r(x) \times (0,r)\};
    \]
    and
    \[
    \Osc^-_{x,r}(S):=\sup \{|t|:(y,t) \in S \cap B_r(x) \times (-r,0)\}.
    \]

The next result implies Theorem \ref{CT3} from the Introduction.

\begin{thm}[General version of Theorem \ref{CT3}] \label{T9}
    Fix $K \in \bb{R}$, $N \in (1,+\infty)$.
    Let $P$ be a modulus of parabolicity.
    For every $t,r,T>0$ there exists $R(K,N,P,t,r,T)>0$ such that if $(\X,\sd,\m,\x)$ is an $\RCD(K,N)$ space with modulus of parabolicity $P$ and $E \subset B_R (\x) \times (-R,R)$ is a perimeter minimizer in $B_R (\x) \times (-R,R)$ such that $(\x,0) \in \partial E$ and $\Osc_{\x,r}(\partial E) \geq t$, then
     \[
    \Osc^+_{\x,R}(\partial E) \geq T \quad \text{and} \quad \Osc^-_{\x,R}(\partial E) \geq T.
    \]
    \begin{proof}
        Suppose by contradiction that the statement is false.  
        Then, there exist $t,r,T>0$, a sequence $\{R_i\}_{i \in \bb{N}} \subset (0,+\infty)$ increasing to $+\infty$, a sequence of $\RCD(K,N)$ spaces $\{(\X_i,\sd_i,\m_i,\x_i)\}_{i \in \bb{N}}$ with modulus of parabolicity $P$,
        and sets $E_i \subset \X_i \times (-R_i,R_i)$ such that $(\x_i,0) \in \partial E_i$, $\Osc_{\x_i,r}(\partial E_i) \geq t$ and
     \[
    \Osc^+_{\x_i,R_i}(\partial E_i) < T \quad \text{or} \quad \Osc^-_{\x_i,R_i}(\partial E_i) < T.
    \]
    Moreover, if we replace each measure $\m_i$ with its normalized version $\m_i(B_1(\x_i))^{-1}\m_i$, we have that each set $E_i$ is still perimeter minimizing in the normalized product space and that the spaces $(\X_i,\sd_i,\m_i(B_1(\x_i))^{-1}\m_i,\x_i)$ still have modulus of parabolicity $P$.
    Hence it is not restrictive to assume that the measures $\m_i$ satisfy $\m_i(B_1(\x_i))=1$. \par 
    Without loss of generality and passing to a (non relabeled) subsequence we may then suppose that for every $i \in \bb{N}$ we have 
    \[
    \Osc^-_{\x_i,R_i}(\partial E_i) < T,
    \]
    so that modulo replacing $E_i$ with its complement we may also suppose that
    \begin{equation} \label{E22}
    E_i \cap B_{R_i}(\x_i) \times [-R_i,R_i] \subset B_{R_i}(\x_i) \times [-T,+\infty).
    \end{equation}
    Up to passing to another subsequence we have that the spaces $(\X_i,\sd_i,\m_i,\x_i)$ converge in pmGH sense to an $\RCD(K,N)$ space $(\X,\sd,\m,\x)$. By the continuity of the mass under pmGH convergence, $\X$ has modulus of parabolicity $P$ and, in particular, it is parabolic. \par 
    Passing to yet another subsequence we have that the sets $E_i$ converge in $\sL^1_{loc}$ sense to a  perimeter minimizing set $E \subset \X \times \bb{R}$. Observe moreover that if $(\ssf{Z},\sd_z)$ is the metric space realizing the convergence of the sets $\X_i$ to $\X$, then $(\ssf{Z} \times \bb{R},\sd_z \times \sd_e)$ realizes the convergence of the corresponding product spaces. \par 
    Combining this with the Kuratowski convergence in $\ssf{Z} \times \bb{R}$ of the boundaries of the sets $E_i$ to the boundary of the set $E$ given by Proposition \ref{P26}, we deduce that
    \begin{equation} \label{Eosc}
    \Osc_{\x,r}(\partial E) \geq t
    \end{equation}
    and
    \[
    E \subset \X \times [-T,+\infty).
    \]
    By Theorem \ref{T|MinPerSlow}, $E$ is a perimeter minimizing set whose boundary is an horizontal slice, contradicting the oscillation \eqref{Eosc}.
    \end{proof}
\end{thm}



In Theorem \ref{T9} asking for the parabolicity of the spaces involved without requiring the existence of a common modulus of parabolicity is not enough, as the next example shows.

\begin{example}
    Consider a bounded non constant solution $u$ of the minimal surface equation in the hyperbolic space $(\bb{H}^2,g,x)$. Such solution exists thanks to \cite{NR}. We claim that for every $i \in \bb{N}$ there exists a metric $g_i$ on $B^g_{2i}(x)$ which coincides with $g$ on $B^g_i(x)$, and such that $(B^g_{2i}(x),g_i)$ is complete, parabolic and satisfies $\Ric_{(B^g_{2i}(x),g_i)} \geq -1$. \par 
    If we are able to do so, then $\Graph(u) \cap B^g_{i}(x) \times \bb{R} \subset (B^g_{i}(x),g_i)$ will be area minimizing in $B_i^{g_i}(x)$ for every $i \in \bb{N}$, but these minimal submanifolds will have uniformly bounded oscillation, showing that Theorem \ref{T9} fails in this setting.
    \par 
    To construct the metric $g_i$ we first consider a smooth surjective function $f_i:[0,+ 2i) \to [0,+ \infty)$ which is constant if $t \leq i$ and has positive first and second derivative. Then we consider the graph of the function $y \mapsto f_i(\sd_g(x,y))$ on $B^g_{2i}(x)$ and we define $g_i$ to be the pullback metric of such graph on $B^g_{2i}(x)$. The fact that $\lim_{r \to 2i}f(r)=+\infty$ makes $g_i$ a complete metric and implies that there exists a constant $c_i$ such that
    \[
    \m_{g_i}(B^{g_i}_r(x)) \leq c_i r,
    \]
    making $(B^g_{2i}(x),g_i)$ parabolic by Theorem \ref{CT5}. 
    At the same time, using the convexity of $f_i$ together with the Gauss-Codazzi equations one can check that $\Ric_{(B^g_{i}(x),g_i)} \geq -1$.
\end{example}

\footnotesize
\printbibliography
\end{document}